\documentclass[notitlepage, 12pt, a4paper]{article}

\usepackage[utf8]{inputenc}
\usepackage[T1]{fontenc}
\usepackage[natbib=true, backend=biber, style=numeric-comp, giveninits=true, maxbibnames=9]{biblatex}
\usepackage{csquotes}
\usepackage{amsmath}
\usepackage{amsfonts}
\usepackage{amssymb}
\usepackage{amsthm}
\usepackage[affil-it]{authblk}
\usepackage[normalem]{ulem}
\usepackage{dsfont}
\usepackage{float}
\usepackage{times}
\usepackage{bbm}
\usepackage{mathtools}
\usepackage{caption}
\usepackage{mathrsfs}
\usepackage{comment}
\usepackage{bm}
\usepackage{enumerate}
\usepackage{lipsum}
\usepackage{titling}
\usepackage{romannum}
\usepackage[colorlinks,linkcolor=blue,citecolor=blue,urlcolor=blue]{hyperref}
\usepackage[capitalise]{cleveref}
\usepackage[left=3cm, right=3cm, bottom=2cm, top=2cm]{geometry}
\usepackage[toc,page]{appendix}

\newcommand*\diff{\mathop{}\!\mathrm{d}}

\newcommand{\cadlag}{c\`adl\`ag }

\newcommand{\DR}{\ensuremath{D_\R}}
\newcommand{\MR}{\ensuremath{\mathbf{M}_{\le 1}(\R)}}

\newcommand{\N}{\ensuremath{\mathbb{N}}}

\newcommand{\R}{\ensuremath{\mathbb{R}}}

\newcommand{\1}{\ensuremath{\mathds{1}}}

\newcommand{\Expec}[1]{\mathbb{E}\left[#1\right]}
\newcommand\norm[1]{\left\lVert#1\right\rVert}
\newcommand\nnorm[1]{|#1|}
\newcommand\nnnorm[1]{\left|#1\right|} 

\newcommand{\prob}{\ensuremath\mathbb{P}}
\newcommand{\tBar}{\ensuremath\Bar{\mathcal{T}}_{\varepsilon}}
\newcommand{\osq}{\ensuremath o_{\textnormal{sq}}(1)}
\newcommand{\ad}{\ensuremath \partial_x^{-1}} 
\newcommand{\Teps}{\ensuremath \mathcal{T}_{\varepsilon}}
\newcommand{\E}{\mathbb{E}}

\newcommand{\ind}{\mathbbm{1}}
\newcommand{\pref}{\ensuremath{p_\mathrm{ref}}}
\newcommand{\pres}{\ensuremath{p_\mathrm{res}}}
\newcommand{\mures}{\ensuremath{\mu_{\mathrm{res}}}}
\newcommand{\suppref}{\cite{SUPPLEMENT_2025}}
\newcommand{\supprop}[1]{Proposition~#1 in the supplement~\suppref}
\newcommand{\supcor}[1]{Corollary~#1 in the supplement~\suppref}

\newtheorem{theorem}{Theorem}[section]
\newtheorem{lemma}[theorem]{Lemma}
\newtheorem{proposition}[theorem]{Proposition}
\newtheorem{corollary}[theorem]{Corollary}
\newtheorem{assumption}[theorem]{Assumption}

\theoremstyle{definition}

\theoremstyle{remark}
\newtheorem{remark}[theorem]{Remark}

\numberwithin{equation}{section}
\crefname{example}{Example}{Examples}

\begin{document}
\pagenumbering{arabic}
\title{A stochastic Fokker--Planck equation for the mean-field limit of a population of noisy integrate-and-fire neurons}

\author{Ben Hambly\thanks{Mathematical Institute, University of Oxford, United Kingdom, Email: hambly@maths.ox.ac.uk},
\and Alda\"ir Petronilia\thanks{Mathematical Institute, University of Oxford, United Kingdom, Email: aldairpetronilia@gmail.com},
\and Christoph Reisinger\thanks{Mathematical Institute, University of Oxford, United Kingdom,\\\emph{Email: christoph.reisinger@maths.ox.ac.uk}},
\and Andreas S{\o}jmark\thanks{Department of Statistics, London School of Economics, United Kingdom,\\\emph{Email: a.sojmark@lse.ac.uk}}}
\date{\today} 

\maketitle


\begin{abstract}
We study a densely connected system of excitatory integrate-and-fire neurons which are subject to common noise. The system incorporates a gradual transmission of action potentials and captures the re- and hyperpolarization phases through a random refractory period followed by a reset to a randomized level below the rest potential. As the number of neurons tends to infinity, we show that there is weak convergence to a unique membrane potential density governed by a stochastic Fokker--Planck equation with a well-defined spike transmission rate. The latter is driven by the mean cumulative spike count, which is non-differentiable but shown to satisfy a generalized flux condition. We obtain the uniqueness of the Fokker--Planck equation from energy estimates in the dual of the first Sobolev space. Finally, we give a conditional McKean--Vlasov representation of the membrane potential density as the law of a representative neuron given the common noise.
\end{abstract}

\section{Introduction}\label{sec: INTRODUCTION SECTION IN THE GENERALISED INTEGRATE AND FIRE NEURON MODELS}

When analysing a large population of neurons, each of which may be in a different internal state at any given time, a fruitful and much-used approach is to adopt a mean-field point of view and thus derive, at least informally, a continuity equation for the time evolution of the state distribution \citep[Part III]{neuronal_dyn}. In the case of integrate-and-fire models, which are the focus of the present paper, the state of each neuron is given by its membrane potential, so the limiting object is known as the membrane potential density. If the system is subject to noisy input currents, as we shall consider here, then the continuity equation will typically correspond to a suitable nonlinear and non-local Fokker--Planck equation \cite{Brunel2013}.

Mean-field approaches to the analysis of networks of spiking integrate-and-fire neurons, at the level of the membrane potential density, go at least as far back as the early 70's \cite{knight_1972}, followed by several important contributions in the early 90's, see e.g.~\cite{abbott1993asynchronous, Amit_1991, Treves_1993}. The formulations that we shall study here, however, are rooted in subsequent developments starting with the influential series of papers \cite{Amit_Brunel_1997, brunel2000dynamics, brunel1999fast}.

Specifically, we are interested in convergence to and well-posedness of the continuity equation for the membrane potential density, as we pass to the mean-field limit in fully connected networks of leaky integrate-and-fire models with noisy input currents and coupling through gap junctions, as studied in \cite{ostojic2009synchronization}. Following \cite{lewis2003dynamics, ostojic2009synchronization},  we let the latter be approximated by simple ohmic conductances between the membranes, and our starting point is thus the general model derived in \cite{ostojic2009synchronization}. Letting $X^i$ denote the membrane potential of the $i$-th neuron, and denoting by $\tau^i_k$ the $k$-th time that it fires a spike, the dynamics are of the form
\begin{equation}\label{eq:membrane_potentials}
	\mathrm{d}X^i_t = b(t)\diff t -\mu X^i_t \diff t +  \frac{\gamma}{N} \sum_{j\neq i} (X^j_t - X^i_t)\diff t + \frac{\alpha}{N}  \sum_{j\neq i} \sum_{k\ge 1} \diff H(t-\tau^j_k) + \sigma \diff B_t^i.
\end{equation}
Here, each $B^i$ is a Brownian motion, $H$ is the Heaviside function, and it is implicitly understood that: (i) $\tau^i_k$ is given by the $k$-th time that $X^i$ reaches a given firing threshold $x_{\text{th}}$, and (ii) $X^i$ is instantly reset to a rest potential $x_{\text{r}}<x_{\mathrm{th}}$ whenever it fires a spike.

In fact, $X^{i}$ should have a continuous voltage trace $X^{i,\text{spike}}_t \gg x_{\text{th}} $ for the duration $\Delta_T>0$ of each spike, with the spike transmission occurring as a fast ohmic conductance at rate $ \frac{\beta}{N}(X^{i,\text{spike}}_t- X^{j}_t)$ for a given $\beta>0$. As $\Delta _T$ is small and $X^{i,\text{spike}}_t \gg x_{\text{th}} > X^{j}_t$, the full amplitude (i.e., the current from $i$ to $j$ integrated over the spike duration $\Delta_T$) is approximated by an instantaneous pulse $\frac{\alpha}{N}$ at $\tau^i_k$ in \eqref{eq:membrane_potentials}, for a suitable $\alpha >0$. It is standard to treat the spikes themselves as stereotyped events, focusing instead on modelling the number and timing of spikes \citep{DiLorenzoPatriciaM2013Ster, neuronal_dyn}. However, it is desirable that the dynamics of the membrane potentials account for the fast but gradual shape of spike transmissions and the subsequent refractory periods. We shall return to this in Section \ref{subsect:stoch_FP}. 

\subsection{The Fokker--Planck approach and existing literature}

Following the mean-field formalism of \cite{Amit_Brunel_1997, brunel2000dynamics, brunel1999fast}, certain heuristics allow \cite{ostojic2009synchronization} to derive a Fokker--Planck equation for the membrane potential density $V_t$, which would correspond to a hydrodynamic limit of \eqref{eq:membrane_potentials}. It may be expressed as
\begin{equation}\label{eq:PDE}
	\begin{cases}
		\partial_t V_t(x) + \partial_x\bigl( (b(t,x,V_t) - \alpha f(t))V_t(x)\bigr) - \frac{\sigma^2}{2} \partial_{xx} V_t(x) =  f(t)\delta_{x_\text{r}} \\
		f(t) = -\frac{\sigma^2}{2}\partial_x V_t(x_\text{th}),\quad V_t(x_{\text{th}})=0,\quad V_t(-\infty)=0,
	\end{cases}
\end{equation}
for $x\in(-\infty,x_{\text{th}})$, where $\delta_{x_\text{r}}$ denotes a Dirac mass at the rest potential $x_\text{r}$ and $b(t,x,V_t)$ is of the form $- b(t) +\mu x - \gamma(\int y V_t(y)\diff y - x)$. Here $f(t)$ is the mean \emph{firing rate}, which also serves as the rate at which spikes are transmitted in the drift term of \eqref{eq:PDE}. By the second line of \eqref{eq:PDE}, $f(t)$ equals the probability flux across the firing threshold. Since this flux also appears as a singular probability source in the rest potential $x_\text{r}$ (spiking neurons are assumed to instantaneously reset to $x_\text{r}$), we have conservation of probability.

A rigorous notion of a global weak solution to \eqref{eq:PDE} was first formulated in \cite{caceres2011analysis} and it was shown that, depending on $\alpha>0$ and the initial condition $V_0$, such solutions may fail to exist due to a finite-time blow-up of the firing rate $f(t)$. Later, it was proved in \cite{carrillo2013classicalNonLinearFokkerPlanck} that, when $V_0$ is $C^1$ with $V_0(x_{\text{th}})=0$, there exists a unique classical solution up until the blow-up time. Conditions for ruling out a blow-up, based on the initial condition and properties of stationary states, were then explored in \cite{carillo_nonlinearity}. By
focusing on the increasing (mean) cumulative spike count $F(t)=\int_0^tf(s)\diff s$, it becomes possible to make sense of $F$ even if the firing rate $f$ fails to exist at some points. Exploiting this perspective, \cite{ delarue2015global, delarue2015particle} provide a generalized probabilistic notion of solution to \eqref{eq:PDE} in the form of a McKean--Vlasov SDE whose marginal densities satisfy \eqref{eq:PDE} if $F$ is $C^1$. Firstly, it was shown in \cite{delarue2015particle} that the empirical measures for a system of SDEs of the form \eqref{eq:membrane_potentials} have limit points that are \cadlag solutions of this McKean--Vlasov SDE globally in time and it was confirmed that any jumps of these limit points are `physical' in a suitable sense. Secondly, \cite{delarue2015global} established that, for any Dirac initial condition $V_0=\delta_{x_0}$, one can take $\alpha\in(0,1)$ small enough so that the McKean--Vlasov SDE is well-posed within the class of solutions for which the mean cumulative spike count $F$ is $C^1$.

Generalizing the approach of \cite{brunel2000dynamics}, \cite{Caceres_Perthame_2014} introduces refractory periods in \eqref{eq:PDE} and observes that finite-time blow-up still applies. A synaptic delay between the spiking (i.e., the exiting flux) and the spike transmission is studied in \cite{Caceres_roux_etal} and it is proved that there is then global well-posedness. Moreover, \cite{roux1} derives refined results for \eqref{eq:PDE} itself, concerning finite-time blow-up and conditions for global well-posedness. Recently, \cite{perthame_jump} has studied the continuation of solutions after blow-ups, based on passing to the limit in certain regularized formulations, but a corresponding uniqueness result for \eqref{eq:PDE} remains an open problem. Steady states and numerical analysis are explored in \cite{canizo} and \cite{Zhennan}. 

In the case of a single neuron, \cite{Liu_etal} studies the precise connection between the SDE for its membrane potential and the PDE \eqref{eq:PDE} for the marginal densities (with $b(x)=\mu x$ and $\alpha=0$). Working with a particular variant of \eqref{eq:PDE}, motivated by a diffusive approximation of point process spikes in \cite{brunel1999fast, brunel2000dynamics}, \cite{dou_zhou} studies a generalized notion of solution with jumps in the firing rate $f$ that now occur without an explosion of the flux across the boundary (due to the particular structure). By performing a time-change in terms of the (deterministic) firing rate, global existence and uniqueness is obtained. Independently, but based on related ideas, a delayed Poissonian version of \eqref{eq:PDE} is introduced in \cite{Taillefumier1}, also with blow-ups of $f$ now occurring without an explosion in flux, thus making the analysis more tractable than \eqref{eq:PDE} and enabling \cite{Taillefumier2} to prove existence and uniqueness in the presence of countably many jumps through a time-change argument. While no convergence results are derived, we note that the model is motivated by a variant of \eqref{eq:membrane_potentials}: between spikes, $\mathrm{d}X_t^i = b \mathrm{d}t + \mathrm{d}B^i_t$, but, at each $\tau^i_k$, the instantaneous pulse from $i$ to $j$ is now sampled independently at random from the Normal distribution with equal mean and variance $\alpha /N$, followed by a refractory period and then a reset.

\subsection{A generalized stochastic Fokker--Planck equation}\label{subsect:stoch_FP}

The component $b(t)\diff t + \sigma \diff B^i_t$ in \eqref{eq:membrane_potentials} is the net effect of all \emph{external} input currents flowing to neuron $i$ from sources outside the given network \cite{ostojic2009synchronization}. To arrive at \eqref{eq:PDE}, one assumes their fluctuations are uncorrelated across the population, so the $\{B^i\}_{i=1}^N$ are independent. As we discuss in Section \ref{sect:common_noise_IF} below, it can be important to incorporate correlation through a common noise, i.e., $B^i=\sqrt{1-\rho^2}W^i + \rho W^0$ for independent Brownian motions $\{W^i\}_{i=0}^N$. We will do this and show that the membrane potential density can then be characterized by a stochastic partial differential equation driven by $W^0$.

Beyond this, we extend the modelling of the \emph{internal} input current (from couplings) in \eqref{eq:membrane_potentials} in three directions, to better capture the real nature of action potentials. A given neuron fires a spike (i.e., triggers an action potential) whenever its membrane potential reaches the given threshold voltage $x_\text{th}$. The action potential then undergoes a de-, re-, and hyper-polarization phase. Without explicitly modelling these phases (which is rather the domain of Hodgkin--Huxley type conductance models \cite{Beeman2015}), we wish to account for their different effects on the system dynamics in a tractable way. Firstly, the initial depolarization phase is not instantaneous, but instead given by a rapid continuous increase in voltage (above $x_\text{th}>0$) that, in turn, is transmitted gradually to the other neurons. Moreover, the resulting currents flow continuously through electrical synapses, thus adding another rapid but gradual effect. Still denoting by $\alpha/N$ the total amplitude transmitted to each neuron, we therefore model the \emph{spike transmission} of the $k$-th spike of neuron $i$ by $t\mapsto \frac{\alpha}{N}\int_{0}^{t} \mathfrak{K}(t-s)H(s-\tau^i_k)\diff s$, for an appropriate kernel $\mathfrak{K}$ with $\Vert \mathfrak{K} \Vert_{L^1}=1$, in place of the idealized instantaneous pulse $t\mapsto \frac{\alpha}{N}H(t-\tau^i_k)$ in \eqref{eq:membrane_potentials}. Next, after a spike, the neuron enters an \emph{absolute refractory period}, where it cannot fire again, lasting at least until near the end of the repolarization phase. To capture this, we hold a spiking neuron in an auxiliary state for some (short) random amount of time. Finally, there is the relative refractory period during the hyperpolarization, where the membrane potential first undershoots its rest potential before gradually reverting to it. We model this by resetting to a randomly sampled position below the rest potential $x_\text{r} \le 0$ when it leaves the aforementioned auxiliary state (i.e., immediately after the absolute refractory period) and then, as in \eqref{eq:membrane_potentials}, the term $-\mu X^i_t \mathrm{d}t $ begins to enforce a return towards the rest potential.

At the start of Section \ref{sec: MAIN RESULTS SECTION IN THE GENERALISED INTEGRATE AND FIRE NEURON MODELS}, we give a precise formulation of the relevant particle system, incorporating the above considerations. Our results give rigorous meaning to the membrane potential density $V_t$ in $L^2(-\infty , x_{\text{th}})$ as the unique mean-field limit of this system. Instead of \eqref{eq:PDE}, it is characterized by the stochastic Fokker--Planck equation
\begin{equation}\label{eq:SPDE_simplified}
	\begin{aligned}
		\mathrm{d} V_t(x) & = -  \partial_x\bigl ( (b(t,x,V_t) - \alpha \mathfrak{f}_t )V_t(x)\bigr) \mathrm{d}t  + \frac{\sigma^2}{2} \partial_{xx}V_t(x) 	\diff t  \\
		&\qquad - \rho\sigma \partial_x V_t(x) \diff W^0_t  + \pres(x) \mathfrak{r}_t \diff t \\
		\mathfrak{f}_t & = \int_0^t \mathfrak{K} (t-s) \diff F_s ,\quad  \mathfrak{r}_t = \int_0^t \pref (t -s) \diff F_s,
	\end{aligned}
\end{equation}
in the sense of distributions on $(-\infty , x_{\text{th}})$, with the Dirichlet condition $\lim_{\varepsilon \downarrow 0} \langle V_t,\psi_\varepsilon \rangle =0$ and a generalized flux condition for the mean cumulative spike count $F$, namely
\begin{equation}\label{eq:weak_flux_cond}
	\E \bigl[ F_t - F_s|\, \mathcal{F}_s \bigr] =  \lim_{\varepsilon \downarrow 0} \frac{\sigma^2}{2} \int_s^t \E [ \bigl\langle V_u, \partial_{xx} \psi_\varepsilon \rangle |\, \mathcal{F}_s \bigr] \mathrm{d} u,\;\; \text{for all}\;\;s,t\ge 0,
\end{equation}
where $(\psi_\varepsilon)_{\varepsilon>0}$ is any family of smooth approximations of the Dirac mass at $x_\text{th}$ supported on $(x_\text{th},x_\text{th}+\varepsilon)$, and $\langle \cdot, \cdot \rangle$ denotes the $L^2$ inner product.

In \eqref{eq:SPDE_simplified}, we refer to $\mathfrak{f}_t$ as the \emph{spike transmission rate}. It is well-defined even if the mean spike count $F$ is not absolutely continuous but only of finite variation. Next, $\mathfrak{r}_t$ is the effective \emph{rate of resets} given in terms of $F$ and the \emph{refractory density} $\pref$, i.e., the probability density function for the length of the absolute refractory period. Finally, $\pres$ is the \emph{reset density}, i.e., the probability density function for the voltage level that the membrane potential is reset to after the absolute refractory period. If $F$ and $V$ were sufficiently regular, \eqref{eq:weak_flux_cond} would be equivalent to $\frac{\mathrm{d}}{\mathrm{d}t}F_t = -\frac{\sigma^2}{2}\partial_xV_t(x_\text{th})$ as in \eqref{eq:PDE}.

Closely related SPDE formulations have been considered in the mathematical and computational neuroscience literature, see e.g.~\cite{brunel1999fast, brunel2000dynamics, Mattia2002, Ullner2018}. However, these works all assume informally that $F$ is continuously differentiable, and they do not consider the approaches to spike transmission and refractory periods that we have introduced here. In particular, they generally take $\pres(x) = \delta_{x_\text{r}}(x)$, possibly after a fixed delay, and set $\mathfrak{f}_t = \mathfrak{r}_t = \frac{\mathrm{d}}{\mathrm{d}t} F_t$, so the SPDE is stated exactly as the PDE \eqref{eq:PDE} except for the additional noise term. This is not meaningful when $\rho \neq 0$, as $F$ will be nowhere absolutely continuous, and $V_t$ will fail to be differentiable at $x_\text{th}$ in general (see e.g.~\cite{Krylov2003} in a simpler setting). Thus, the formulation \eqref{eq:SPDE_simplified}-\eqref{eq:weak_flux_cond} is important for two reasons: firstly, it incorporates key features of an idealized action potential; secondly, it is critical in allowing us to establish well-posedness of the SPDE for the membrane potential density and confirm that it arises as the weak limit of the corresponding particle system (provided of course that the coefficients satisfy certain reasonable criteria). As far as the authors are aware, this constitutes the first rigorous treatment of the mean-field approach to excitatory networks of spiking integrate-and-fire neurons with input currents affected by common noise.

\subsection{Common noise in Integrate-and-Fire models}\label{sect:common_noise_IF}

As mentioned above, several works have informally derived SPDE versions of \eqref{eq:PDE} driven by  Brownian motion, due to input currents correlated through common noise. Specifically, in \cite[Section 3.5]{brunel1999fast}, \cite[Section 6.1]{brunel2000dynamics}, and \cite[Section F]{Mattia2002}, the input current for a given neuron is modelled as a diffusive approximation of the global spike activity within the network with two parts: an idiosyncratic noise capturing fluctuations of the spikes transmitted directly to the given neuron and a common noise due to intrinsic fluctuations in the spike activity of the whole network. If neurons share a negligible amount of common input, the latter effect vanishes as $N\rightarrow \infty$: with $\rho \sim 1/N$, we arrive at the PDE \eqref{eq:PDE}. However, a small noise effect in the limit is also considered, giving rise to a Brownian term in the limit, as in \eqref{eq:SPDE_simplified} with $\rho>0$. In \cite[Equation (4)]{Ullner2018}, this arises from every neuron in the system receiving a shared input current from a separate group of neurons.

Various related works could also be treated within our framework. Firstly,  \citep{kruscha2016partial,kruscha2015spike} study a system of $N$ leaky integrate-and-fire neurons of the form
\begin{equation}\label{eq:no_coupling}
	\diff X_t^i = -X_t^i\diff t + b\diff t + \sigma \sqrt{1 - \rho^2 }\diff {W}_t^i + \sigma \rho \diff {W}_t^0,
\end{equation}
between resets, to understand synchronized spiking behaviour in a homogeneous population driven by a common time-dependent stimulus, due to overlapping receptive fields, rather than couplings between the neurons. As noted in those works, such populations of uncoupled neurons with overlapping receptive fields are present, e.g., at the first stage of sensory input processing with examples including auditory nerve cells \citep{kandel2000principles}, olfactory receptor neurons in insects \citep{Galizia2014olfactory, Wilson2013olfactory} and electroreceptors in weakly electric fish \citep{Maler2009electricfish}. Secondly, \citep{doiron2003inhibitory} studies a leaky integrate-and-fire neuronal model with a common noise and an OU process for the independent dynamics. The common noise captures `spatiotemporal sensory input' or `spatially extended sensory signals in neural systems'. Thirdly,  \cite{lewis2006} studies how a common input current can lead to `synergistic' synchronization of the spike transmissions, through couplings modelled by ohmic conductances as in \eqref{eq:membrane_potentials}. Finally, \cite{doiron2004oscillatory, lindner2005theory} study the effect of the common noise in a version of \eqref{eq:no_coupling} with coupling through spike transmission, namely
\begin{equation}\label{eq:coupled_spikes}
	\diff X_t^i = -X_t^i \diff t + b  \diff t +  \sigma\sqrt{1 - \rho^2}\diff {W}_t^i + \sigma\rho \diff {W}_t^0 + \alpha \tilde{\mathfrak{f}}^N_t \diff t,
\end{equation}
where $\tilde{\mathfrak{f}}^N$ is defined similarly to the finite particle system analogue of $\mathfrak{f}$ from \eqref{eq:SPDE_simplified}.

\subsection{Key contributions and overview of our approach}

In the existing literature, the well-posedness results of \cite{inglis2015mean} and \cite{Caceres_roux_etal} come closest to the analysis we undertake here. The former work establishes global existence and uniqueness of solutions to a McKean--Vlasov SDE for a representative membrane potential \cite[Equation (2.2)]{inglis2015mean} with gradual spike transmission modelled by a suitable kernel in the same way as discussed above. This is derived by use of a cable equation to describe how the input current flows across a dendrite before reaching the (somatic) membrane potential, as in \cite{bressloff1997synchrony}. The kernel $\mathfrak{K}$ then arises from the fundamental solution of the cable equation. The other work \cite{Caceres_roux_etal} focuses instead on the Fokker--Planck point of view, as in the present paper. They also have non-instantaneous spike transmission, however it takes the slightly different form of a fixed non-zero delay. The resulting PDE for the membrane potential density \cite[Equations (1.1)--(1.3)]{Caceres_roux_etal} is recast as a free boundary problem to get local well-posedness through a fixed point argument similarly to \cite{Caceres_2011} and this is then extended globally through the use of supersolutions and estimates on the flux $\partial_x V_t(x_\mathrm{th})$. In \cite{inglis2015mean}, the analysis is based on a rather different but nonetheless related fixed point argument via estimates on the continuous firing rate $f$, building on ideas developed in \cite{delarue2015global}.

With the common noise, we no longer have a well-defined flux and firing rate, so the approaches of \cite{inglis2015mean, Caceres_roux_etal} do not generalize to our setting. Instead, we only work with $L^2$-integrability of the membrane potential density and prove well-posedness via suitable energy estimates in the dual of the first Sobolev space, based on a weak formulation of the stochastic Fokker--Planck equation. We also note that the above and other existing works do not consider random refractory periods. Thus, we must rely on a more careful construction of the underlying particle system and, in turn, a more delicate analysis is required in order to obtain convergence to the postulated mean-field limit. Refractory periods of various forms have been considered at the PDE level, see for example~\cite{Caceres_Perthame_2014, Caceres_refractory_delay}, but the convergence of a corresponding particle system was not considered. Moreover, the convergence to a (unique) mean-field limit has not previously been addressed for networked integrate-and-fire models with common noise. We conclude by providing a McKean--Vlasov characterization of the membrane potential density as the conditional law of a representative membrane potential. Compared to the analysis in \cite{inglis2015mean}, we have to deal with the common noise, we treat more general coefficients, and we place less strict assumptions on the transmission kernel and the initial condition.

For the proofs, we will build on the analysis in \cite{hamblyseanhalfline, hambly2019spde} which consider related SPDE problems set on a half-line. Our overall approach is inspired by these works, but significant new ideas are needed to handle our new setting. In the terminology of integrate-and-fire models, the setting of \cite{hambly2019spde} can be mapped to a situation where neurons are simply killed upon their first spike instead of undergoing a refractory period and being reinserted into the system. At the level of the particle system, new arguments are needed to obtain tightness when dealing with the additional terms arising from refractory periods and reinsertion. In particular, substantial technical work goes into controlling the amount of firing, the increments of the cumulative spike count, and the increments of the number of reinsertions. Moreover, in order to identify the  mean-field limit, the refractory periods and resets necessitate a more detailed convergence analysis that carefully exploits the way in which the particle system is constructed. Notably, we face a delicate new task of ascertaining that the firing times, refractory periods, and reset positions decorrelate as the number of neurons tends to infinity. To obtain $L^2$ estimates for the membrane potential density, we work directly with the weak formulation of the SPDE instead of relying on comparison with an auxiliary system as in previous works. Concerning uniqueness, our energy estimates must deal with a non-standard source term and a more complicated nonlinearity coming from how the cumulative spike count affects the membrane potential density. Furthermore, the energy estimates require new arguments to control the membrane potential density in expectation, which we achieve through a probabilistic comparison with an absorbed Brownian motion that is suitably restarted after each spike.


\section{Main results}\label{sec: MAIN RESULTS SECTION IN THE GENERALISED INTEGRATE AND FIRE NEURON MODELS}

Throughout the paper, we work on a fixed filtered probability space $(\Omega, \mathcal{F},(\mathcal{F}_t)_{t \ge 0},\prob)$, where the filtration $(\mathcal{F}_t)_{t \ge 0}$ is taken to satisfy the usual conditions. Furthermore, we will assume that this space supports an infinite sequence of random variables, $\{W^i,\varsigma_k^i,\xi_k^i\}_{i,k \ge 0}$, which are all mutually independent of each other. Here each $W^i$ is an $\mathcal{F}_t$-standard Brownian motion, while $\{\varsigma_k^i\}_{k, i}$ and $\{\xi_k^i\}_{k,i}$ are i.i.d.~random variables on $\R$ that are $\mathcal{F}_0$-measurable and satisfy \cref{ass: ASSUMPTIONS FROM SOJMARK SPDE PAPER APPLIED TO THE NEUROSCIENCE SETTING} below. For simplicity, when the indexing is not important, we will simply refer to $\xi_k^i$ and $\varsigma_k^i$ by $\xi$ and $\varsigma$ respectively.

The various features of the membrane potential dynamics, as discussed in Section \ref{subsect:stoch_FP}, are made precise in the following particle system. For simplicity of notation, we take the firing threshold to be $x_\mathrm{th}=0$. We study the system in the general form 
\begin{equation}\label{eq: GENERALISED FINITE PARTICLE SYSTEM IN THE INTEGRATE AND FIRE NEURON MODELS}
	\left\{
	\begin{array}{r@{{}={}}l}
		\diff{}X_t^{i} &\begin{array}[t]{@{}l}
			b(t,X_t^{i},\nu_t^{N},\mathfrak{f}_t^N)\ind_{\{X_t^i < 0\}}\diff{}t + \sigma(t,\,X_t^{i})\rho(t,\nu^N_t,\mathfrak{f}_t^N)\ind_{\{X_t^i < 0\}}\diff{}W_t^0\\[0.25em]
			+ \  
			\sigma(t,\,X_t^{i})\sqrt{1 \!-\! \rho^2(t,\nu^N_t,\mathfrak{f}_t^N)}\ind_{\{X_t^i < 0\}}\diff{}W_t^i - \diff \sum_{k \ge 1} \xi_k^i\ind_{[0,t]}(\tau_k^i + \varsigma_k^i), \\[0.25em]
		\end{array} \\[0.25em]
		\tau_k^{i} & \inf\{t > \tau_{k-1}^{i} + \varsigma_{k-1}^{i} \; : \; X_{t-}^{i} \ge 0\},\qquad \tau_0^{i} = 0,\\[0.25em]\nu_t^N & \frac{1}{N}\sum_{i=1}^N \delta_{X_t^{i}}\ind_{\{X_t^i < 0\}},\quad \mathfrak{f}_t^N = \int_0^t \mathfrak{K}(t-s) \diff{F_s^N}\\[0.25em]
		F_t^N & \frac{1}{N}\sum_{i=1}^N J_t^{i},\quad J_t^{i} = \sum_{k \ge 1} \ind_{[0,t]}(\tau_k^i).
	\end{array}
	\right.
\end{equation}
Here, $\tau^i_k$ is the $k$-th time that neuron $i$ fires and the random variable $\varsigma_k^i$ gives the corresponding refractory period, i.e., the amount of time after it fires during which neuron $i$ is not affected by any stimulus. Following this refractory period, the term $-\xi_k^i$ models the level to which the membrane potential of the neuron is reset. Note that the empirical measure $\nu_t^N$ tracks the neurons that are not currently deactivated, due to the refractory period. Moreover, we recall that $F_t^N$ gives the cumulative spike count up until time $t$ (rescaled by $1/N$) and $\mathfrak{f}_t^N$ is the corresponding spike transmission rate.

For our analysis, we shall also need the following notation
\begin{equation}\label{eq: THE F D N AND J D I WHICH HAVE BEEN EXTRACTED FROM FROM THE EQUATION ABOVE TO MAKE IT MORE STREAMLINED}
	F_t^{D,N} = \textstyle \frac{1}{N}\sum_{i=1}^N J_t^{D,i},\quad 
	J_t^{D,i} = \textstyle\sum_{k \ge 1} \ind_{[0,t]}(\tau_k^i + \varsigma_k^i),
\end{equation}
which counts the number of resets. Another key quantity is the empirical mean $M_t^N \coloneqq \langle\nu_t^N,\operatorname{Id}\rangle$. Recall that, in \eqref{eq:membrane_potentials}, \eqref{eq:PDE}, and \eqref{eq:SPDE_simplified}, the coefficients $\sigma$ and $\rho$ are constants, while the drift $b$ involves the mean $M^N$ and has linear growth in $X^i$ which our assumptions will allow for. The well-posedness of \eqref{eq: GENERALISED FINITE PARTICLE SYSTEM IN THE INTEGRATE AND FIRE NEURON MODELS} is addressed in \cref{sec: THE FINITE PARTICLE SYSTEM}.

\subsection{Assumptions and Notation}

Before stating our assumptions, we first introduce some basic notation. First of all, we fix an arbitrary finite time horizon $T > 0$. We will then phrase our results in terms of functional weak convergence on the arbitrary time interval $[0,T]$. Next, we let $\mathcal{P}(A)$ denote the set of probability measures on a measurable space $(A,\,\mathcal{A})$. Whenever $A$ is a metric space, we write $\mathcal{B}(A)$ for the associated Borel $\sigma$-algebra. Let further $\mathbf{M}_{\le 1}(A)$ denote the space of sub-probability measures, which we endow with the topology of weak convergence of measures. For any interval $I$ and metric space $X$, we let $C(I, X)$ denote the space of continuous functions from $I$ to $X$. Similarly, $D(I, X)$ denotes the space of \cadlag functions from $I$ to $X$. We will use the shorthand notation $\mathcal{C}_X$ and $D_X$ for $C(I, X)$ and $D(I, X)$, respectively, whenever the interval $I$ is clear.

We shall be working under suitable Lipschitz conditions with respect to the membrane potentials $X_t^i$, the empirical measure $\nu_t^N$, and the spike transmission rate $\mathfrak{f}_t^N$. Regarding the empirical measure, this will involve two distances on the space of sub-probability measures $\mathbf{M}_{\le 1}(\R)$, namely $d_0$ and $d_1$ defined below. We note that our assumptions allow for linear growth and include the cases discussed in the introduction.

\begin{assumption}[Structural assumptions]\label{ass: ASSUMPTIONS FROM SOJMARK SPDE PAPER APPLIED TO THE NEUROSCIENCE SETTING}
	We assume that the following structural conditions are satisfied by the particle system \eqref{eq: GENERALISED FINITE PARTICLE SYSTEM IN THE INTEGRATE AND FIRE NEURON MODELS}, for some $\gamma,\, R > 0$:
	\begin{enumerate}[(i)]
		\item \label{ass: ASSUMPTIONS FROM SOJMARK SPDE PAPER APPLIED TO THE NEUROSCIENCE SETTING ONE}(Growth and differentiability) The map $x \mapsto b(t,x,\mu,f)$ is $\mathcal{C}^2(\R)$ and $(t,x) \mapsto \sigma(t,x)$ is $\mathcal{C}^{1,2}([0,T]\times \R)$. Moreover, there exist $C_b,C_\sigma > 0$ such that
		\begin{align*}
			&\nnorm{b(t,x,\mu,f)} \le C_b(1 + \nnorm{x}+ \langle \mu,\nnorm{\cdot}\rangle + \nnorm{f}),\quad\nnorm{\partial_x^{(n)}b(t,x,\mu,f)}\le C_b,\quad n = 1,\,2,\\
			&\nnorm{\sigma(t,x)}\le C_\sigma,\quad\nnorm{\partial_t\sigma(t,x)}\le C_\sigma,\quad \nnorm{\partial_x^{(n)}\sigma(t,x)}\le C_\sigma,\quad n = 1,\,2.
		\end{align*}
		\item \label{ass: ASSUMPTIONS FROM SOJMARK SPDE PAPER APPLIED TO THE NEUROSCIENCE SETTING TWO}(Lipschitzness) There exists $C > 0$ such that
		\begin{align*}
			\nnorm{b(t,x,\mu,f) - b(t,\tilde{x},\tilde{\mu},\tilde{f})} &\le C(\nnorm{x - \tilde{x}} + d_0(\mu,\tilde{\mu}) + \nnorm{f - \tilde{f}}),\\
			\nnorm{\sigma(t,x) - \sigma(t,\tilde{x})} &\le C(\nnorm{x - \tilde{x}}),\\
			\nnorm{\rho(t,\mu,f) - \rho(t,\tilde{\mu},\tilde{f})} &\le C(d_1(\mu,\tilde{\mu}) + \nnorm{f - \tilde{f}}),
		\end{align*}
		where
		\begin{align*}
			d_0(\mu,\,\tilde{\mu}) &= \sup\left\{\nnnorm{\langle\mu - \tilde{\mu},\, \psi\rangle}\,:\, \norm{\psi}_{\operatorname{Lip}} \le 1,\, \nnnorm{\psi(0)} \le 1\right\},\\
			d_1(\mu,\,\tilde{\mu}) &= \sup\left\{\nnnorm{\langle\mu - \tilde{\mu},\, \psi\rangle}\,:\, \norm{\psi}_{\operatorname{Lip}} \le 1,\, \norm{\psi}_{\infty} \le 1\right\}.
		\end{align*}
		\item \label{ass: ASSUMPTIONS FROM SOJMARK SPDE PAPER APPLIED TO THE NEUROSCIENCE SETTING THREE} (Non-degeneracy) $C_\sigma$ above can be chosen such that $0 < C_\sigma^{-1} \le \sigma(t,x)$ and that $0 \le \rho(t,\mu,f)\le 1 - \gamma$.
		\item \label{ass: ASSUMPTIONS FROM SOJMARK SPDE PAPER APPLIED TO THE NEUROSCIENCE SETTING FIVE} (Spike transmission) The kernel $\mathfrak{K}$ is non-negative with $\norm{\mathfrak{K}}_1 = 1$ and $\mathfrak{K} \in \mathcal{W}^{1,1}_0(\R_+)$, the Sobolev space with one weak derivative in $L^1$ and zero trace. 
		\item \label{ass: ASSUMPTIONS FROM SOJMARK SPDE PAPER APPLIED TO THE NEUROSCIENCE SETTING FOUR}(Random inputs)
		The random variables $\{X_0^i, W^i,\varsigma_k^i,\xi_k^i\}_{i,k \ge 0}$ are all mutually independent of each other, and the  $\{X_0^i\}_{i\ge 1}$, $\{W^i\}_{i\ge 0}$, $\{\varsigma_k^i\}_{i,k \ge 0}$, and $\{\xi_k^i\}_{i,k \ge 0}$ form i.i.d.~sequences. The common law $\nu_0$ of the starting values $X_0^i$ has a density $V_0$ in $L^2(-\infty,0)$ and
		\begin{equation*}
			\nu_0(-\infty,-\lambda) = O(e^{-\gamma \lambda^2}) \quad \text{as} \quad \lambda \to \infty.
		\end{equation*}
		The common law of the refractory periods $\varsigma_k^i$ has a density $\pref \in \mathcal{W}^{1,1}_0(\R_+)$. Writing $\mu_\mathrm{res}$ for the common law of the reset positions $-\xi^{i}_k$, we have $\operatorname{supp}(\mu_\mathrm{res}) \subseteq [-R,-\gamma]$.
	\end{enumerate}
\end{assumption}

To establish uniqueness of the limiting SPDE, stated in \eqref{eqn: THE LIMITTING SPDE EQUATION} below, we restrict ourselves to a class of processes that have sufficient regularity. Naturally, every limit point for the particle system \eqref{eq: GENERALISED FINITE PARTICLE SYSTEM IN THE INTEGRATE AND FIRE NEURON MODELS} will be shown to live in this class. Moreover, we will show that the unique solution is measurable with respect to the filtration generated by the Brownian motion driving the SPDE, but we stress that the uniqueness arguments do not rely on this.

\begin{assumption}[Conditions for uniqueness]\label{ass: ASSUMPTIONS NEEDED TO PROVE UNIQUENESS WHICH PARALLEL THOSE FROM THE SPDE PAPER}
	Uniqueness will be established within the class of processes $(\tilde{\nu},\tilde{F})$ that belong to the product space $D_{\mathbf{M}_{\le 1}(\R)} \times D_\R$ and satisfy the following conditions:
	\begin{enumerate}[(i)]
		\item \label{ass: ASSUMPTIONS NEEDED TO PROVE UNIQUENESS WHICH PARALLEL THOSE FROM THE SPDE PAPER ONE} (Support on $\R_-$) For every $t \in [0,T]$, $\tilde{\nu}_t$ is supported in $\R_- = (-\infty,0]$.
		\item \label{ass: ASSUMPTIONS NEEDED TO PROVE UNIQUENESS WHICH PARALLEL THOSE FROM THE SPDE PAPER TWO} (Mean spike count) The mean cumulative spike count $\tilde{F}$ is increasing and obeys the relation \[
		\tilde{\nu}_t(-\infty,0) + \tilde{F}_t - \int_0^t \pref(t - s)\tilde{F}_s \diff s = 1.
		\]         
		\item \label{ass: ASSUMPTIONS NEEDED TO PROVE UNIQUENESS WHICH PARALLEL THOSE FROM THE SPDE PAPER THREE} (Exponential tails) For every $\delta > 0$, we have 
		\begin{equation*}
			\E\left[\sup_{t \le T} \tilde{\nu}_t(-\infty,-a)\right] = o(\exp\{-\delta a\}) \quad \textnormal{as} \quad a \to  \infty.
		\end{equation*}
		\item \label{ass: ASSUMPTIONS NEEDED TO PROVE UNIQUENESS WHICH PARALLEL THOSE FROM THE SPDE PAPER FOUR} (Dirichlet boundary decay) There exists a $\beta > 0$ such that \[\E\int_0^T \tilde{\nu}_t(-\varepsilon,0)\diff t = o(\varepsilon^{1 + \beta}) \quad \textnormal{as} \quad \varepsilon \to  0.\]
		\item \label{ass: ASSUMPTIONS NEEDED TO PROVE UNIQUENESS WHICH PARALLEL THOSE FROM THE SPDE PAPER FIVE} (Spatial concentration) There exist $C,\delta > 0$ such that \[\E\int_0^T (\tilde{\nu}_t(a,b))^2\diff t \le C\nnnorm{a - b}^{\delta} \quad \quad \forall \, a,b \in \R.\]
	\end{enumerate}
\end{assumption}

We note that \cref{ass: ASSUMPTIONS NEEDED TO PROVE UNIQUENESS WHICH PARALLEL THOSE FROM THE SPDE PAPER} \eqref{ass: ASSUMPTIONS NEEDED TO PROVE UNIQUENESS WHICH PARALLEL THOSE FROM THE SPDE PAPER TWO} captures the way in which neurons that fire are gradually reinserted into the system. This is a natural formulation for the limiting equation as $\tilde{F}_t$ represents the cumulative mass of neurons that have fired up to time $t$. Similarly, $\int_0^t \pref(t - s)\tilde{F}_s \diff s$ is the mass of the system that has been reinserted by time $t$. Hence, their difference represents the mass of particles that is currently in the refractory period. The mass of particles in their refractory period and the mass of active particles must sum to one. We also stress that \cref{ass: ASSUMPTIONS NEEDED TO PROVE UNIQUENESS WHICH PARALLEL THOSE FROM THE SPDE PAPER} \eqref{ass: ASSUMPTIONS NEEDED TO PROVE UNIQUENESS WHICH PARALLEL THOSE FROM THE SPDE PAPER TWO} can be seen to be equivalent to 
\begin{align}
	\tilde{F}_t &= 1 - \tilde{\nu}_t(-\infty,0) \notag \\
	&+ \sum_{n \ge 1} \int_0^t\!\int_0^{t_1}\!\!\!\!\cdots\!\int_0^{t_n}\!\!(1 - \tilde{\nu}_{t_n}(-\infty,0))\pref(t - t_n)\prod_{j = 1}^{n-1}\pref(t_j - t_{j+1})\diff t_1 \cdots \diff t_n. \label{eq: REWRITING THE FIRING FUNCTION ONLY IN TERMS OF THE MEAUSRE}
\end{align}
In particular, we may view $\tilde{F}$ as a functional of $\tilde{\nu}$, so it will be possible to characterize the SPDE solution in terms of only the sub-probability measure-valued process $\tilde{\nu}$.

\subsection{The SPDE for the Membrane Potential Density}\label{subsect:SPDE_main_results}

Rather than working directly with $D_{\mathbf{M}_{\le 1}}=D_{\mathbf{M}_{\le 1}(\mathbb{R})}$, we will work on the product space $(D_{\mathscr{S}^\prime},M_1) \times (D_{\R},M_1) \times (\mathcal{C}_\R,\norm{\cdot}_{\infty})$ to establish tightness and weak convergence of the tuple ${(\nu^N,F^N,W^0)}_N$. Here, $\mathscr{S}$ is the space of Schwartz functions on $\mathbb{R}$, the space of rapidly decreasing infinitely differentiable functions, and $\mathscr{S}^\prime$ is its dual, the space of tempered distributions. Thus, $(D_{\mathscr{S}^\prime},M_1)$ is the space of $\mathscr{S}^\prime$-valued \cadlag processes on $[0,T]$, equipped with the $M_1$-topology as constructed in \citep{ledger2016skorokhod}. We stress that the empirical measures $\nu^N=(\nu^N_t)_{t\in[0,T]}$ live in this space and we shall recover the limit as a sub-probability measure valued process. The space $(D_{\mathscr{S}^\prime},M_1)$ has good properties and the $M_1$-topology will be very useful for our setting, as it allows us to exploit the monotonicity of the cumulative spike count $F^N$. We refer to \citep{ledger2016skorokhod} for details on the space $(D_{\mathscr{S}^\prime},M_1)$.

Our first main result confirms that weak limit points of the particle system \eqref{eq: GENERALISED FINITE PARTICLE SYSTEM IN THE INTEGRATE AND FIRE NEURON MODELS} are characterized by a general version of the stochastic Fokker--Planck equation \eqref{eq:SPDE_simplified}.

\begin{theorem}[Tightness and limiting SPDE]\label{thm: THE LIMITING SPDE THEOREM MAIN RESULT OF THE WORK}
	Let \cref{ass: ASSUMPTIONS FROM SOJMARK SPDE PAPER APPLIED TO THE NEUROSCIENCE SETTING} be satisfied. Then $\{(\nu^N,F^N,W^0)\}_N$ is tight on $(D_{\mathscr{S}^\prime},M_1) \times (D_{\R},M_1) \times (\mathcal{C}_\R,\norm{\cdot}_{\infty})$. Moreover, for any limit point $(\nu, F, W^0)$, $\nu$ belongs to  $D_{\mathbf{M}_{\le 1}}$ with probability 1, the pair $(\nu,F)$ satisfies \cref{ass: ASSUMPTIONS NEEDED TO PROVE UNIQUENESS WHICH PARALLEL THOSE FROM THE SPDE PAPER}, and, with probability 1, the triple $(\nu,F,W^0)$ obeys the SPDE
	\begin{align}\label{eqn: THE LIMITTING SPDE EQUATION}
		\begin{split}
			\diff{}\langle\nu_t,\,\phi\rangle =\,&
			\langle\nu_t,\,b(t,\cdot,\nu_t,\mathfrak{f}_t)\partial_x\phi\rangle\diff{t}
			+\frac{1}{2}\langle\nu_t,\sigma^2(t,\cdot)\partial_{xx}\phi\rangle\diff{t}\\
			&+\langle\nu_t,\,\sigma(t,\cdot)\rho(t,\nu_t,\mathfrak{f}_t)\partial_x\phi\rangle\diff{W_t^0}
			+ \langle \mures, \phi \rangle \mathfrak{r}_t \diff{t} - \phi(0)\diff{}F_t ,
		\end{split}
	\end{align}
	for all $\phi \in \mathscr{S}$, with initial condition $\nu_0(\mathrm{d}x)=V_0(x) \diff x$, and with
	\[
	\mathfrak{r}_t = \int_0^t \pref (t -s) \mathrm{d} F_s \quad \text{and}\quad \mathfrak{f}_t = \int_0^t \mathfrak{K}(t-s)\diff F_s.
	\]
	Here, $\mathfrak{K}$ is the spike transmission kernel, $\pref$ is the density of the refractory period, and $\mures$ is the law of the reset values of the membrane potential, as per Assumption \ref{ass: ASSUMPTIONS FROM SOJMARK SPDE PAPER APPLIED TO THE NEUROSCIENCE SETTING}.
\end{theorem}

It is implicit in \eqref{eqn: THE LIMITTING SPDE EQUATION} that $(\nu, F)$ is adapted to a filtration for which $W^0$ is a Brownian motion. As regards the models discussed in the introduction, $\mathfrak{f}$ is the mean spike transmission rate and $\mathfrak{r}$ is the mean reset rate, both of which are defined in terms of the change in the mean cumulative spike count $F$. One can observe that, by taking a test function $\phi \in \mathscr{S}$ to approximate $1$ at the origin (i.e., the firing threshold $x_\mathrm{th}=0$) and $0$ elsewhere, the weak formulation \eqref{eqn: THE LIMITTING SPDE EQUATION} encodes a relationship between $F$ and the flux of the membrane potential distribution $\nu$ at the origin. This is made precise in the following statement, which confirms  the condition \eqref{eq:weak_flux_cond} from the introduction.

\begin{proposition}[Weak flux condition]\label{lem: THE WEAK FLUX CONDITION}
	Let $(\nu, F)$ be any solution to the SPDE \eqref{eqn: THE LIMITTING SPDE EQUATION} and let Assumptions \ref{ass: ASSUMPTIONS FROM SOJMARK SPDE PAPER APPLIED TO THE NEUROSCIENCE SETTING} and \ref{ass: ASSUMPTIONS NEEDED TO PROVE UNIQUENESS WHICH PARALLEL THOSE FROM THE SPDE PAPER} be in force. 
	If $\psi_\varepsilon$ is a smooth approximation of $\ind_{\{0\}}$, then it holds almost surely and in $L^1$ that
	\begin{equation*}
		\E\left[\left.F_t - F_s\,\right| \mathcal{F}_s\right]   = \frac{1}{2}\lim_{\varepsilon \downarrow 0} \int_s^t \E\left[\left. \langle \nu_u,\partial_x[\sigma(u,\cdot)^2\partial_x \psi_\varepsilon] \rangle \,\right| \mathcal{F}_s\right] \diff u.
	\end{equation*}
\end{proposition}

Next, we are able to establish a pathwise uniqueness result for the SPDE \eqref{eqn: THE LIMITTING SPDE EQUATION} within the class of solutions satisfying \cref{ass: ASSUMPTIONS NEEDED TO PROVE UNIQUENESS WHICH PARALLEL THOSE FROM THE SPDE PAPER}. Crucially,  the latter is satisfied by all limit points of the particle system in view of the first part of Theorem \ref{thm: THE LIMITING SPDE THEOREM MAIN RESULT OF THE WORK}.

\begin{theorem}[Uniqueness]\label{thm: UNIQUENESS PROP}
	Let \cref{ass: ASSUMPTIONS FROM SOJMARK SPDE PAPER APPLIED TO THE NEUROSCIENCE SETTING} be satisfied, and let $(\nu, F, W^0)$ and $(\tilde{\nu}, \tilde{F}, W^0)$ be any two solutions to \eqref{eqn: THE LIMITTING SPDE EQUATION} satisfying \cref{ass: ASSUMPTIONS NEEDED TO PROVE UNIQUENESS WHICH PARALLEL THOSE FROM THE SPDE PAPER} on the same probability space. Then, with probability $1$,
	\begin{equation*}
		\nu_t(A) = \tilde{\nu}_t(A) \quad \textnormal{and} \quad F_t = \tilde{F}_t \quad \forall \, t \in [0,T],\;  \forall A \in \mathcal{B}(\R).
	\end{equation*}
\end{theorem}

Based on this result, we can finally deduce that we also have uniqueness in law and, in turn, we get full weak convergence to a unique strong solution.

\begin{theorem}[Mean-field limit]\label{thm: CONDITIONAL LLN STATING THAT WE CONVERGE TO A UNIQUE LIMIT POINT THAT DEPENDS ON THE COMMON NOISE}
	The SPDE \eqref{eqn: THE LIMITTING SPDE EQUATION} admits a unique law on $(D_{\mathscr{S}^\prime},M_1) \times (D_{\R},M_1) \times (\mathcal{C}_\R,\norm{\cdot}_{\infty})$ and $(\nu^N,F^N,W^0)$ converges weakly to this law as $N\rightarrow \infty$. Furthermore, in the limit, there exists a measurable map $Q: \mathcal{C}_\R \to D_{\mathscr{S}^\prime}\times D_\R$ such that $(\nu, F) = Q(W^0)$. Finally, we also have that the empirical means $M^N=\langle \nu^N, \mathrm{Id} \rangle$ converge weakly to the mean $M$ given by $M_t = \langle \nu_t , \mathrm{Id} \rangle$.
\end{theorem}

As part of the uniqueness proof, we confirm in Proposition \ref{prop:L2_density} that the measure $\nu_t$ has a (random) density $V_t$ in $L^2$, for each $t\ge 0$, such that $\sup_{t \in [0,T]} \norm{V_t}_2 $ is finite with probability $1$. Consequently, we have shown that the empirical measures of the finite particle system \eqref{eq: GENERALISED FINITE PARTICLE SYSTEM IN THE INTEGRATE AND FIRE NEURON MODELS} converge weakly to a unique membrane potential density $V_t \in L^2(-\infty,0)$ which evolves according to the stochastic Fokker--Planck equation \eqref{eqn: THE LIMITTING SPDE EQUATION}.

Our final result confirms that the membrane potential density can be represented as the conditional law of a certain conditional McKean--Vlasov SDE. This extends the findings of \cite{inglis2015mean} to the common noise setting, and it makes a precise connection between the SDE point of view in \cite{inglis2015mean} and the Fokker--Planck point of view in \cite{Caceres_roux_etal}.

\begin{theorem}[Conditional McKean--Vlasov representation]\label{thm: CONDITIONAL MCKEAN VLASOV FORMATION FOR NEURO MODEL WITH DELAYS AND RESETS}
	Let $(\nu, F,W^0)$ be the unique strong solution to the SPDE \eqref{eqn: THE LIMITTING SPDE EQUATION}. Then, for any Brownian motion $W$ and random variables $\{\xi_k\}_k$ and $\{\varsigma_k\}_k$, all mutually independent and independent of both $X_0$ and $W^0$, we have
	\begin{equation*}
		\nu_t = \prob\bigl[\left.X_t \in \cdot\,,\, X_t < 0\,\right| W^0\bigr],
	\end{equation*}
	for $t\ge 0$, where $X$ is the unique solution to the conditional McKean--Vlasov diffusion
	\begin{equation}\label{eq: LIMITTING MCKEAN VLASOV SYSTEM IN THE INTEGRATE AND FIRE NEURON MODELS}
		\left\{
		\begin{array}{r@{{}={}}l}
			\diff X_t &\begin{array}[t]{@{}l}
				b(t,X_t,\nu_t,\mathfrak{f}_t)\ind_{\{X_t< 0\}}\diff{}t + \sigma(t,X_t)\rho(t,\nu_t,\mathfrak{f}_t)\ind_{\{X_t < 0\}}\diff{}W_t^0 \\[0.35em]
				+ 
				\sigma(t,X_t)\sqrt{1 \!-\! \rho^2(t,\nu_t,\mathfrak{f}_t)}\ind_{\{X_t< 0\}}\diff{}W_t
				- \diff \sum_{k \ge 1} \xi_k\ind_{[0,t]}(\tau_k + \varsigma_k) ,\\[0.35em]
			\end{array} \\[0.35em]
			\tau_k & \inf\{t > \tau_{k-1} + \varsigma_{k-1} \; : \; X_{t-} \ge 0\},\quad \tau_0 = 0,\\[0.35em]
			\mathfrak{f}_t  & \int_0^t \mathfrak{K}(t-s)\diff{F_s},\quad  F_t = \sum_{k\ge1} \prob\left[\left.\tau_k \le t\,\right|W^0\right].
		\end{array}
		\right.
	\end{equation}
\end{theorem}

\begin{proof} In essence, this is a consequence of the above results, so we outline the proof already here, even if it will involve some arguments that are only to be developed later on in the paper. First of all, 
	\cref{prop: EXISTENCE OF A PATHWISE UNIQUE SOLUTION TO THE FINITE PARTICLE SYSTEM} shows how to iteratively construct a solution to the finite particle system \eqref{eq: LIMITTING MCKEAN VLASOV SYSTEM IN THE INTEGRATE AND FIRE NEURON MODELS}. A straightforward modification of this procedure allows us to also construct a solution, say $\tilde{X}$, to the SDE \eqref{eq: LIMITTING MCKEAN VLASOV SYSTEM IN THE INTEGRATE AND FIRE NEURON MODELS}, where the pair $(\nu, F)$, appearing in the coefficients of the system, is given exogenously by a realised limit point of the particle system. For such a solution, we can then define $\tilde{\nu}_t = \prob[\tilde{X}_t \in \cdot, \tilde{X}_t < 0\mid W^0]$ along with the corresponding firing process $\tilde{F}$. Next, we can observe that
	\begin{equation*}
		1 = \ind_{\{\tilde{X}_t = 0\}} + \ind_{\{\tilde{X}_t < 0\}} = \sum_{k \ge 1} \ind_{[0,t]}(\tilde{\tau}_k) - \sum_{k \ge 1} \ind_{[0,t]}(\tilde{\tau}_k + \varsigma_k) + \ind_{\{\tilde{X}_t < 0\}}.
	\end{equation*}
	Taking expectations conditional on $W^0$ in this expression, we obtain
	\begin{equation*}
		\tilde{\nu}_t(-\infty,0) + \tilde{F}_t - \int_0^t \pref(t - s)\tilde{F}_s \diff s = 1.
	\end{equation*}
	Thus, the pair $(\tilde{\nu},\tilde{F})$ constructed above satisfies \cref{ass: ASSUMPTIONS NEEDED TO PROVE UNIQUENESS WHICH PARALLEL THOSE FROM THE SPDE PAPER} \eqref{ass: ASSUMPTIONS NEEDED TO PROVE UNIQUENESS WHICH PARALLEL THOSE FROM THE SPDE PAPER TWO}. Based on the results in Section \ref{sec: THE FINITE PARTICLE SYSTEM}, it is an easy exercise to confirm that $(\tilde{\nu},\tilde{F})$ satisfies the remaining conditions of \cref{ass: ASSUMPTIONS NEEDED TO PROVE UNIQUENESS WHICH PARALLEL THOSE FROM THE SPDE PAPER} as well. Hence, as the pair $(\nu, F)$ is $W^0$-measurable by \cref{thm: CONDITIONAL LLN STATING THAT WE CONVERGE TO A UNIQUE LIMIT POINT THAT DEPENDS ON THE COMMON NOISE}, we may argue as in \citep[Section~9]{hamblyseanhalfline} and apply It\^o's formula to $\phi(\tilde{X}_t)$ for $\phi \in \mathscr{S}$ to show that $\tilde{\nu}$ solves a linear version of SPDE \eqref{eqn: THE LIMITTING SPDE EQUATION} where the nonlinear dependence on $\tilde{\nu}$ itself is replaced by the exogenously given $\nu$. Then, by uniqueness of this linear SPDE, which follows from the proofs of \cref{thm: UNIQUENESS PROP} and \cref{thm: CONDITIONAL LLN STATING THAT WE CONVERGE TO A UNIQUE LIMIT POINT THAT DEPENDS ON THE COMMON NOISE}, we finally conclude that we must have $\tilde{\nu} = \nu$ and $\tilde{F} = F$ which completes the proof.
\end{proof}

\subsection{Overview of the rest of the paper}

The remaining sections are concerned with the proofs of the above results.

In Section \ref{sec: THE FINITE PARTICLE SYSTEM}, we construct a unique solution to the particle system, derive certain bounds for the empirical measures, and estimate the increments of the expected cumulative spike count as well as the expected number of reinsertions.

In Section \ref{sec: TIGHTNESS AND EXISTENCE SECTION IN THE GENERALISED INTEGRATE AND FIRE NEURON MODELS} we show that the empirical measures and other relevant quantities are tight for Skorokhod's $M_1$-topology, and we then characterize the weak limit points as solutions to an SPDE. Moreover, we derive key regularity properties for the limit points which we need for our uniqueness arguments.

Finally, Section \ref{sec: UNIQUENESS SECTION IN THE GENERALISED INTEGRATE AND FIRE NEURON MODELS} proves the uniqueness of solutions to the SPDE. This goes via  suitable energy estimates for mollified solutions. We also confirm that the unique solution is adapted to the driving noise only.

In Sections \ref{sec: TIGHTNESS AND EXISTENCE SECTION IN THE GENERALISED INTEGRATE AND FIRE NEURON MODELS} and \ref{sec: UNIQUENESS SECTION IN THE GENERALISED INTEGRATE AND FIRE NEURON MODELS}, we make heavy use of several uniform estimates on the tail and boundary behaviour of the particles in the particle system \eqref{eq: GENERALISED FINITE PARTICLE SYSTEM IN THE INTEGRATE AND FIRE NEURON MODELS}. We defer this analysis to the supplement \cite{SUPPLEMENT_2025}, which we refer to throughout. In the case without a common noise and with more restrictive assumptions on the coefficients and the initial data, some of the results in this supplement are related to estimates obtained in \cite{delarue_supplement} and \cite{inglis2015mean}, but their techniques do not extend to our setting.


\section{Properties of the Finite Particle System}\label{sec: THE FINITE PARTICLE SYSTEM}

This section addresses the well-posedness of the particle system, the independence between stopping times and the refractory periods, asymptotic decorrelation of key quantities, and finally recalls some key results from the supplement \suppref.

\subsection{Well-Posedness of the Particle System}

We begin by showing that the finite particle system for the membrane potentials is well-posed. That is, we can construct a solution and there is pathwise uniqueness.

\begin{proposition}\label{prop: EXISTENCE OF A PATHWISE UNIQUE SOLUTION TO THE FINITE PARTICLE SYSTEM}
	There exists a pathwise unique solution to the particle system \eqref{eq: GENERALISED FINITE PARTICLE SYSTEM IN THE INTEGRATE AND FIRE NEURON MODELS}.
\end{proposition}

\begin{proof}
	By \cref{ass: ASSUMPTIONS FROM SOJMARK SPDE PAPER APPLIED TO THE NEUROSCIENCE SETTING}, our coefficients are globally Lipschitz. Our strategy is as follows. As the value of $J^i$ is constant between hitting times of the boundary at $0$, we have that $F^N$ is constant between hitting times. Therefore, we may iteratively construct a pathwise unique solution to \eqref{eq: GENERALISED FINITE PARTICLE SYSTEM IN THE INTEGRATE AND FIRE NEURON MODELS} by viewing the coefficients as functions of time, $t$, and the diffusions that are below $0$ on the intervals of constancy of $J^i$.
	
	To be precise, we define the process $\mathbf{X}^{N} = (X^1,\ldots,X^N)$ for $t \in [0, \tau^1 \wedge T]$ to be the solution to the vector-valued SDE
	\begin{equation*}
		\diff{}X_t^{i} = b(t,X_t^{i},\nu_t^{N},0)\diff{}t + \sigma(t,\,X_t^{i})\sqrt{1 - \rho^2(t,\nu^N_t,0)}\diff{}W_t^i + \sigma(t,\,X_t^{i})\rho(t,\nu^N_t,0)\diff{}W_t^0. 
	\end{equation*}
	This process is well-defined by standard results for Lipschitz coefficients, as guaranteed by Assumption \ref{ass: ASSUMPTIONS FROM SOJMARK SPDE PAPER APPLIED TO THE NEUROSCIENCE SETTING} \eqref{ass: ASSUMPTIONS FROM SOJMARK SPDE PAPER APPLIED TO THE NEUROSCIENCE SETTING TWO} and \eqref{ass: ASSUMPTIONS FROM SOJMARK SPDE PAPER APPLIED TO THE NEUROSCIENCE SETTING THREE}. Consequently, there exists a unique strong solution. Now, we set $\tau^{1,\,i} = \inf \{t \in [0,T]\; ; \; X_t^i \ge 0 \}$ ($\inf\{\emptyset\} = \infty$). Clearly, we have $0 < \tau^{1,i} $ a.s., so $0 < \inf\{\tau^{1,i}\} $ a.s. Set $\tau^1 = \inf\{\tau^{1,i}\}$ and $i^1 = \operatorname{argmin} \{\tau^{1,i}\}$. On the event $\{\tau^1 < \infty\}$ we set 
	\begin{equation*}
		J_t^i = 0,\,t \in [0,\,\tau^{1}),\qquad \textnormal{and} \qquad J_{t}^i = 
		\begin{cases*}
			$1$ & if  $i = i^1$, $t \ge \tau^1$, \\
			$0$ & if $i \neq i^1$, $t \ge \tau^1$.
		\end{cases*}
	\end{equation*}
	We define the random index set $\mathcal{I}_{\textnormal{ref}}^1 = \{i^1\}$, which represents the set of diffusions that are currently deactivated, and set $\tau^{2,i^1} = \tau^1 + \varsigma_{J_{\tau^1}^{i^1}}^{i^1}$, the time when particle $i^1$ will be reinserted back into the system. 
	
	For the induction step, we assume that we have defined a sequence of stopping times $\tau^1,\ldots,\tau^k$ and $\tau^{k+1,i}$ for $i \in \mathcal{I}_{\textnormal{ref}}^k$. On the event $\{\tau^k < \infty\}$, we define $\mathbf{X}^{N}$ for $t \in (\tau^k,\tau^{k+1}\wedge T]$ to be the solution to the vector-valued SDE
	\begin{align*}
		X_t^{i} =& X_{\tau^k}^i + \int_{\tau^k}^t b(s,X_s^{i},\nu_s^{N},\mathfrak{f}_s^N)\diff{}s + \int_{\tau^k}^t \sigma(s,\,X_s^{i})\sqrt{1 - \rho^2(s,\nu^N_s,\mathfrak{f}_s^N)}\diff{}W_s^i \\
		&+ \int_{\tau^k}^t \sigma(s,\,X_s^{i})\rho(s,\nu^N_s,\mathfrak{f}_s^N)\diff{}W_s^0
	\end{align*}
	for $i \not\in \mathcal{I}_{\textnormal{ref}}^k$, where $\mathfrak{f}_t^N = (\mathfrak{K}^\prime \ast F^N)_t$. For $i \in \mathcal{I}_{\textnormal{ref}}^k$, we set $X_t^i = 0$. This process is well-defined by standard results for Lipschitz coefficients; therefore, there exists a unique strong solution. 
	
	For $ i \not \in \mathcal{I}_{\textnormal{ref}}^k$, we set $\tau^{k+1,i} = \inf \{t \in [0,T]\; ; \; X_t^i \ge 0 \}$  ($\inf\{\emptyset\} = \infty$). Set $\tau^{k+1} = \inf\{\tau^{k+1,i}\}$ and $i^{k+1} = \operatorname{argmin}\{\tau^{k+1,\,i}\}$. On the event $\{\tau^{k + 1} = \infty\}$, we set $J_t^i = J_{\tau^k}^i$ for $t \ge \tau^k$. On the event $\{\tau^{k + 1} < \infty\}$, we shall proceed in two cases.
	
	\noindent \underline{Case 1: $i^{k+1} \in \mathcal{I}_{\textnormal{ref}}^k$}\\[1 ex]
	In this case, we are reinserting a particle back into the system; therefore, we need to update the values $J^i$ and the particle being reinserted, $X^{i^{k+1}}$. As no new active particles hit the boundary, set $J_t^i = J_{\tau^k}^i$ for $t \ge \tau^{k + 1}$ for any $i$. Furthermore, we set $X_{\tau^{k+1}}^{i^{k+1}} = -\xi^{i^{k+1}}_{J_{\tau^{k+1}}^{i^{k+1}}}$. Next, we update the set of deactivated particles by defining $\mathcal{I}_{\textnormal{ref}}^{k +1 } = \mathcal{I}_{\textnormal{ref}}^k \setminus \{i^{k+1}\}$. Lastly, we set $\tau^{k+2,i} = \tau^{k+1,i}$ for $i \in \mathcal{I}_{\textnormal{ref}}^{k+1}$.
	
	\noindent \underline{Case 2: $i^{k+1} \not\in \mathcal{I}_{\textnormal{ref}}^k$}\\[1 ex]
	In this case, an active particle has hit the boundary. Hence, we first update the values of $J^i$ to 
	\begin{equation*}
		J_{t}^i = \begin{cases*}
			J_{\tau^k}^i + 1 &if  $i = i^{k+1}$, $t \ge \tau^{k+1}$, \\
			J_{\tau^k}^i &if $i \neq i^{k+1}$, $t \ge \tau^{k+1}$,
		\end{cases*}
	\end{equation*}
	and set $\mathcal{I}_{\textnormal{ref}}^{k +1} = \mathcal{I}_{\textnormal{ref}}^k \cup \{i^{k+1}\}$. Lastly, we set $\tau^{k+2,i^{k+1}} = \tau^{k+1} + \varsigma^{i^{k+1}}_{J_{\tau^{k+1}}^{i^{k+1}}}$ and $\tau^{k+2,i} = \tau^{k+1,i}$ for $i \in \mathcal{I}_{\textnormal{ref}}^{k+1}$.
	
	Hence, by the principle of induction, we have iteratively constructed a pathwise unique solution. In our construction, we acquired a sequence of increasing stopping times $\{\tau^k\}_k$. We would like to show that these stopping times do not accumulate in finite time, hence our construction will be a solution on $[0,T]$. Let $\tau^* = \lim_{k \to \infty} \tau^k$. On the event $\{\tau^* \le T\}$, there must be a particle $i$ which has hit $0$ infinitely many times in finite time. That is
	\[\{\tau^* \le T\} \subset \bigcup_{i = 1}^N \left\{\lim_{k \to \infty} \tau_k^i \le T\right\}.\]
	As a particle needs to wait at least a $\varsigma_k^i$ amount of time between hitting times, it is clear that 
	\[\{\tau^* \le T\} \subset \bigcup_{i = 1}^N \left\{\lim_{k \to \infty} \tau_k^i \le T\right\} \subset \bigcup_{i = 1}^N \left\{\sum_{k \ge 1} \varsigma_k^i \le T\right\}.\] 
	As $\{\varsigma_k^i\}$ are i.i.d. positive random variables with a density by \cref{ass: ASSUMPTIONS FROM SOJMARK SPDE PAPER APPLIED TO THE NEUROSCIENCE SETTING} \eqref{ass: ASSUMPTIONS FROM SOJMARK SPDE PAPER APPLIED TO THE NEUROSCIENCE SETTING FOUR}, the Strong Law of Large Numbers implies that $\left\{\sum_{k \ge 1} \varsigma_k^i \le T\right\}$ is a null set for every $i$. Hence, $\tau^* = \infty$ almost surely.
\end{proof}

\begin{remark}[Relationship between $\nu^N,\,F^N$ and $F^{D,N}$]\label{rem: REMARK ON THE RELATIONSHIP BETWEEN NU, THE FIRING FUNCTION AND THE DELAYED FIRING FUNCTION}
	We observe that by \eqref{eq: GENERALISED FINITE PARTICLE SYSTEM IN THE INTEGRATE AND FIRE NEURON MODELS} and \eqref{eq: THE F D N AND J D I WHICH HAVE BEEN EXTRACTED FROM FROM THE EQUATION ABOVE TO MAKE IT MORE STREAMLINED}, for any $i = 1,\ldots,N$, $J_t^i - J_t^{D,i} = \ind_{\{X_t^i = 0\}}$. Taking an average over $i$, we have $\nu_t^N(-\infty,0) + F_t^N - F_t^{D,N} = 1$ for any $N \in \N$.
\end{remark} 

\subsection{Independence Properties and Asymptotic Decorrelation}

We can observe that $\tau_k^i$ is independent of $\varsigma_k^i$ and $\xi_k^i$. To see this, consider the particle system $\tilde{\mathbf{X}}^{N} = (\tilde{X}^1,\ldots,\tilde{X}^N)$ constructed in exactly the same fashion as $\mathbf{X}^{N} = (X^1,\ldots,X^N)$, as outlined in \cref{prop: EXISTENCE OF A PATHWISE UNIQUE SOLUTION TO THE FINITE PARTICLE SYSTEM}. The only difference is that when $\tilde{X}^i$ hits the boundary for the $k^{\textnormal{th}}$ time, it is not reset. We also let $\tilde{\tau}_k^i$ denote the corresponding stopping times. As both $\tilde{\mathbf{X}}^{N}$ and ${\mathbf{X}}^{N}$ have a pathwise unique solution between the intervals of constancy of the firing function, we must have that $\tau_k^i = \tilde{\tau}_k^i$ almost surely. Now, it is clear that $\varsigma_k^i$ and $\xi_k^i$ are independent of $\tilde{\tau}_k^i$ as $\tilde{X}^i$ is not reset after it hits the boundary for the $k^{\textnormal{th}}$ time. Hence, we must have $\tau_k^i$ is independent of $\varsigma_k^i$ and $\xi_k^i$. Furthermore, as the paths of $\tilde{\mathbf{X}}^{N} $ and ${\mathbf{X}}^{N} $ must correspond before time $\tau_k^i$, $\{\tau_\ell^j < \tau^i_k \} = \{\tilde{\tau}_\ell^j < \tilde{\tau}^i_k \}$. Hence, the event $\{\tau_\ell^j < \tau^i_k \}$ is independent of $\varsigma_k^i$ and $\xi_k^i$. These two observations give us the following lemma.

\begin{lemma}[Independence of stopping times and refractory period]\label{lem: INDEPENDENCE BETWEEN THE STOPPING TIME AND THE REFRACTORY PERIOD AND THE RESET POINT}
	For any fixed $i,\,j,\,\ell,$ and $k$, we have that $\varsigma_k^i$ and $\xi_k^i$ are independent of $\tau_k^i$ and the event $\{\tau_\ell^j < \tau^i_k \}$.
\end{lemma}

Moreover, we have the following asymptotic result on the decorrelation of the spike times, the duration of refractory periods, and the reset positions. This will be at the very core of our weak convergence arguments.

\begin{proposition}[Asymptotic decorrelation]\label{prop: ASYMPTOTIC DECORRELATION OF THE HITTING TIMES RESETS AND RANDOM WAITING TIMES}
	Let $\mu^{N,k} \coloneqq N^{-1}\sum_{i =  1}^N \delta_{\xi_k^i,\tau_k^i,\varsigma_k^i}$. Then $\{\mu^{N,k}\}_N$ is tight on $\mathcal{P}(\Bar{\R}^3)$, where $\Bar{\R}$ is the one-point compactification of $\R$. For any limit point, denoted by $\mu^{*,k}$, we have $\mu^{*,k} = \prob[\xi \in \diff x] \otimes \mu_2^{*,k}(\diff y) \otimes \prob[\varsigma \in \diff z]$ almost surely.
\end{proposition}

\begin{proof}
	As $\Bar{\R}$ is compact and the product of compact sets is compact in the product topology, $\Bar{\R}^3$ is a compact metric space. Consequently, $\mathcal{P}(\Bar{\R}^3)$ is a compact metric space. Trivially, $\{\mu^{N,k}\}_N$ is tight on $\mathcal{P}(\Bar{\R}^3)$. 
	
	We fix a limit point, $\mu^{*,k}$, along a subsequence, denoted by $N$ for simplicity. Let $\Pi^{*,k}$ denote the law of $\mu^{*,k}$ and $\pi_i : \Bar{\R}^3 \to \Bar{\R}$ be the projection onto the $i^{\textnormal{th}}$ component. Then by the Continuous Mapping Theorem, $\pi_i^\#\mu^{N,k} \implies \pi_i^\#\mu^{*,k}$ for every $i$. First, we fix $i = 1$ and choose a function $f : \Bar{\R} \to \R$, continuous and bounded. Let $G: \R \to \R$ be a continuous function. Then by the Continuous Mapping Theorem,
	\begin{equation}\label{eq: FIRST EQUATION IN ASYMPTOTIC DECORRELATION PROOF}
		\E\left[G(\langle\pi_1^\#\mu^{N,k},f\rangle)\right] \to  \E\left[G(\langle\pi_1^\#\mu^{*,k},f\rangle)\right].
	\end{equation}
	By definition, $\E[G(\langle\pi_1^\#\mu^{N,k},f\rangle)]  = \E[G({N}^{-1}\sum_{i=1}^{N}f(\xi_k^i))]$. Therefore, by the Law of Large Numbers, as $\xi_k^i$ are i.i.d., and the Dominated Convergence Theorem,
	\begin{equation}\label{eq: SECOND EQUATION IN ASYMPTOTIC DECORRELATION PROOF}
		\E\left[G({N}^{-1}\sum_{i=1}^{N}f(\xi_k^i))\right] \to  \E\left[G(\E[f(\xi)])\right] = G(\E[f(\xi)]).  
	\end{equation}
	Therefore, we have shown by \eqref{eq: FIRST EQUATION IN ASYMPTOTIC DECORRELATION PROOF} and \eqref{eq: SECOND EQUATION IN ASYMPTOTIC DECORRELATION PROOF},
	\begin{equation}\label{eq: THIRD EQUATION IN ASYMPTOTIC DECORRELATION PROOF}
		\E\left[G(\langle\pi_1^\#\mu^{*,k},f\rangle)\right] = G(\E[f(\xi)]).
	\end{equation}
	As \eqref{eq: THIRD EQUATION IN ASYMPTOTIC DECORRELATION PROOF} holds for any continuous function $G$, by replacing $G$ with $\tilde{G}(\cdot) = (\cdot  - \E[f(\xi)])^2$, it follows from \eqref{eq: THIRD EQUATION IN ASYMPTOTIC DECORRELATION PROOF} that $\langle\pi_1^\#\mu^{*,k},f\rangle = \E[f(\xi)]$ almost surely. Therefore, by employing a Dynkin's Lemma argument, we deduce $(\pi_1^\#\mu^{*,k})(\diff x) = \prob[\xi \in \diff x]$. Similarly, we deduce $(\pi_3^\#\mu^{*,k})(\diff z) = \prob[\varsigma \in \diff z]$.
	
	It remains only to show that the limiting measure is a product measure. For every $a_1,a_2,a_3: \Bar{\R} \to \R$, continuous and bounded,
	\begin{align*}
		&\E\left[\nnnorm{\langle\mu^{*,k},a_1\otimes a_2 \otimes a_3\rangle - \prod_{i = 1}^3\langle\pi_i^\#\mu^{*,k},a_i\rangle}^2\right]\\
		&\qquad \qquad = \E\left[\nnnorm{\langle\mu^{*,k},a_1\otimes a_2 \otimes a_3\rangle - \E[a_1(\xi)]\langle\pi_2^\#\mu^{*,k},a_2\rangle \E[a_3(\varsigma)]}^2\right]\\
		&\qquad \qquad = \lim_{N \to  \infty}\E\left[\nnnorm{\langle\mu^{N,k},a_1\otimes a_2 \otimes a_3\rangle - \E[a_1(\xi)]\langle\pi_2^\#\mu^{N,k},a_2\rangle \E[a_3(\varsigma)]}^2\right]\\
		&\qquad \qquad = \lim_{N \to  \infty} \E\left[\nnnorm{{N}^{-1}\sum_{i = 1}^{N} a_1(\xi_k^i)a_2(\tau_k^i)a_3(\varsigma_k^i )- N^{-1} \sum_{i =1}^{N} a_2(\tau_k^i) \E[a_1(\xi)]\E[a_3(\varsigma)]}^2\right]\\
		&\qquad \qquad = \lim_{N \to  \infty} N^{-2} \sum_{i,j = 1}^N \E\left[a_2(\tau_k^i)a_2(\tau_k^j)\Delta_i\Delta_j\right],
	\end{align*}
	where $\Delta_i \coloneqq a_1(\xi_k^i)a_3(\varsigma_k^i) - \E[a_1(\xi)]\E[a_3(\varsigma)]$. Recall that by construction, on the event that $\tau_k^j$ has not occurred yet, i.e. $\{\tau_k^i < \tau_k^j\}$, $(\xi_k^j,\varsigma_k^j)$ is independent of $(\tau_k^i,\tau_k^j,\xi_k^i,\varsigma_k^i)$ for $i \neq j$. Therefore,
	\begin{equation}\label{eq: FOURTH EQUATION IN ASYMPTOTIC DECORRELATION PROOF}
		\E\left[a_2(\tau_k^i)a_2(\tau_k^j)\Delta_i\Delta_j\ind_{\{\tau_k^i < \tau_k^j\}}\right] = \E\left[a_2(\tau_k^i)a_2(\tau_k^j)\Delta_i\ind_{\{\tau_k^i < \tau_k^j\}}\right]\E\left[\Delta_j\right] = 0.
	\end{equation}
	The last equality follows from the independence between $\xi_k^i$ and $\varsigma_k^i$. Equation \eqref{eq: FOURTH EQUATION IN ASYMPTOTIC DECORRELATION PROOF} holds true if we swap $i$ and $j$. Hence,
	\begin{equation*}
		\E\left[\nnnorm{\langle\mu^{*,k},a_1\otimes a_2 \otimes a_3\rangle - \prod_{i = 1}^3\langle\pi_i^\#\mu^{*,k},a_i\rangle}^2\right] = \lim_{N \to  \infty} N^{-2} \sum_{i= 1}^N \E\left[a_2(\tau_k^i)^2\Delta_i^2\right] = 0.
	\end{equation*}
	The last equality follows from the boundedness of the $a_i$'s. Therefore, by the Monotone Class Theorem, we may conclude that $\mu^{*,k} = \prob[\xi \in \diff x] \otimes \mu_2^{*,k}(\diff y) \otimes \prob[\varsigma \in \diff z]$ almost surely, where $\mu_2^{*,k} = \pi_2^\#\mu^{*,k}$.
\end{proof}

\subsection{Boundary and Tail Estimates}

In this section, we recall some key estimates from the supplement \cite{SUPPLEMENT_2025}, as they will be used frequently in the proofs that follow.

\begin{proposition}[Uniform moment bound]\label{prop: BOUNDS ON THE SUP OF THE P TH MOMENTS OF THE SUP NORM OF X}
	For any $p \ge 1$, we have 
	\begin{equation*}
		\E\left[\sup_{0\le t \le T}\nnorm{X_t^i}^p + \left(J_T^i\right)^p + \left(F_T^N\right)^p\right] \le C_{T,\,p},
	\end{equation*}
	where $C_{T,\,p} > 0$ is independent of $i$ and $N$.
\end{proposition}

\begin{proposition}[Exponential decay in $k$ of hitting times]\label{prop: EXPONENTIAL DECAY OF THE HITTING PROBABILITIES}
	For any $i,k \ge 1$ and for all $p > 1$, sufficiently close to $1$, we have
	\begin{equation*}
		\prob\left[\tau_k^i \le T\right] \le C e^{-\frac{(k-1)^2\gamma^2(p-1)}{2pTC_\sigma^2}},
	\end{equation*}
	for some constant $C > 0$ that depends on $p$ and $T$ but holds uniformly in $i,\,k$, and $N$, and $ \gamma$ is the constant in \cref{ass: ASSUMPTIONS FROM SOJMARK SPDE PAPER APPLIED TO THE NEUROSCIENCE SETTING} \eqref{ass: ASSUMPTIONS FROM SOJMARK SPDE PAPER APPLIED TO THE NEUROSCIENCE SETTING FOUR}.
\end{proposition}

\begin{proposition}[Concentration and tail estimate]\label{cor: REGULARITY CONTROL ON THEM EMPIRICAL MEASURES}
	The empirical measures $\nu^N$ satisfy, uniformly in $N \ge 1$ and $t \in (0,T]$,
	\begin{equation}\label{eq: REFULARITY PROPERTIES OF THE EMPIRICAL MEASURES OF THE NEURO PARTICLE SYSTEM WITH JUMPS}
		\begin{cases}
			\exists \varepsilon > 0 \qquad &\E\left[\nu_t^N(-\infty,-a)\right] = O(e^{-\varepsilon a^2})\quad \text{as}\quad a \to \infty,\\
			\exists \delta \in (0,1] \qquad &\E\left[\nu_t^N(a,b)\right] \le Ct^{-\delta/2}\nnorm{b-a}^\delta \quad \text{for}\quad a < b < 0.
		\end{cases}
	\end{equation}
\end{proposition}

\begin{proposition}[Boundary estimate]\label{prop: BOUNDARY DECAY OF THE EMPIRICAL MEAUSRES IN EXPECTATION}
	There exists a $\delta \in (0,1]$ and $\beta > 0$ such that uniformly in $N$ and $t \in (0,T]$, we have 
	\begin{equation*}
		\E\left[\nu_t^N(-\varepsilon,0)\right] = t^{-\delta/2}O(\varepsilon^{ 1 + \beta})\quad \text{as}\quad \varepsilon \to  0.
	\end{equation*}
\end{proposition}

These results are established in Corollary 2.2, Proposition 2.8, Proposition 3.2, and Corollary 3.3 of the supplementary note \suppref.


\section{Tightness and SPDE Characterization of Limit Points}\label{sec: TIGHTNESS AND EXISTENCE SECTION IN THE GENERALISED INTEGRATE AND FIRE NEURON MODELS}

The aim of this section is to show that the limit points are within the support of solutions to SPDE \eqref{eqn: THE LIMITTING SPDE EQUATION}. To achieve this, we pass to the limit of the empirical distribution for the finite particle system (along a suitable subsequence) and then employ results on the convergence of (stochastic) integrals. First, we present the following proposition, which describes the evolution equation for the empirical distribution of the finite particle system.

\begin{proposition}\label{prop: PROP ON THE FINITE DIMENSIONAL EVOLUTION EQUATION AND THE DECAY OF THE NOISE IN THE LP NORM}
	Given $N \ge 1$, for all $\phi \in \mathcal{C}^2$, the space of twice continuously differentiable functions on $\R$, we have that
	\begin{align}
		\diff{}\langle\nu_t^N,\,\phi\rangle =&
		\frac{1}{2}\langle\nu_t^N,\,\sigma^2(t,\,\cdot)\partial_{xx}\phi\rangle\diff{t}
		+\langle\nu_t^N,\,b(t,\cdot,\nu_t^N,\mathfrak{f}_t^N)\partial_x\phi\rangle\diff{t} \label{eq: STOCHASTIC EVOLUTION EQUATION} \\
		&+\langle\nu_t^N,\,\sigma(t,\cdot)\rho(t,\nu_t^N,\mathfrak{f}_t^N)\partial_x\phi\rangle\diff{W_t^0}
		+ \diff \frac{1}{N}\sum_{i = 1}^N\sum_{k \ge 1} \phi(-\xi_k^i)\ind_{[0,t]}(\tau_k^i+ \varsigma_k^i) \notag \\
		&- \phi(0)\diff{}F_t^N
		+\frac{1}{N}\sum_{i  =1}^N \sigma(t,X_t^i)\sqrt{1 - \rho^2(t,\nu_t^N,\mathfrak{f}_t^N)}\partial_x\phi(X_t^i)\ind_{\{X_t^i < 0\}}\diff{}W_t^i. \notag
	\end{align}
\end{proposition}
\begin{proof}
	For any $\phi \in \mathcal{C}^2$, we can apply It\^o's formula with jumps \citep[Chapter~II, Theorem~31]{protter2005stochastic} to $\phi(X_t^i)$. Since each $X^i$ given by \eqref{eq: GENERALISED FINITE PARTICLE SYSTEM IN THE INTEGRATE AND FIRE NEURON MODELS} has \cadlag paths and will jump to $-\xi_k^i$ only at times $\tau_k^i + \varsigma_k^i$, and is equal to $0$ as we approach $(\tau_k^i + \varsigma_k^i)$ from below, by construction, this leads to
	\begin{align}
		\phi(X_t^i) =& \phi(X_0^i) +   \int_{0}^t\phi^{\prime}(X_{s}^i)
		b(s,\,X_s^i,\nu_s^{N},\mathfrak{f}_s^N)\ind_{\{X_s^i < 0 \}}\diff{}s\notag\\
		&+ \int_{0}^t\phi^{\prime}(X_{s}^i)\sigma(s,\,X_s^i)\ind_{\{X_s^i < 0\}}\left(\sqrt{1 - \rho^2(s,\nu_s^N,\mathfrak{f}_s^N)}\diff{}W_s^i + \rho(s,\nu_s^N,\mathfrak{f}_s^N)\diff{}W_s^0\right)\notag\\
		&-\int_{0+}^t\phi^{\prime}(X_{s-}^i)\diff{}\sum_{k \ge 1} \xi_k^i\ind_{[0,t]}(\tau_k^i + \varsigma_k^i)+ \frac{1}{2}\int_{0}^t\phi^{\prime\prime}(X_{s}^i)\sigma^2(s,X_s^i)\ind_{\{X_s^i < 0\}}\diff{}s \notag\\
		&+ \sum_{k \ge 1} \{\phi(-\xi_k^i) - \phi(0) - \phi^{\prime}(0)(-\xi_k^i)\}\ind_{[0,t]}(\tau_k^i + \varsigma_k^i).\label{eq: GENERALISED ITO FORMULA FOR MCKEAN VLASOV SDE EXPANDED}
	\end{align}
	Noting that $\int_{0+}^t\phi^{\prime}(X_{s-}^i)\diff{}\sum_{k \ge 1} \xi_k^i\ind_{[0,t]}(\tau_k^i + \varsigma_k^i) = \sum_{k \ge 1} \phi^{\prime}(0)\xi^i_k\ind_{[0,t]}(\tau_k^i + \varsigma_k^i)$, \eqref{eq: GENERALISED ITO FORMULA FOR MCKEAN VLASOV SDE EXPANDED} simplifies to
	\begin{align}
		\diff{}\phi(X_t^i) =& \phi^{\prime}(X_{t}^i)b(t,\,X_t^i,\nu_t^{N},\mathfrak{f}_t^N)\ind_{\{X_t^i < 0 \}}\diff{}t +\frac{1}{2}\phi^{\prime\prime}(X_{t}^i)\sigma^2(t,X_t^i)\ind_{\{X_t^i < 0 \}}\diff{}t\notag\\
		&+ \phi^{\prime}(X_{t}^i)\sigma(t,X_t^i)\ind_{\{X_t^i < 0\}}\left(\sqrt{1 - \rho^2(t,\nu_t^N,\mathfrak{f}_t^N)}\diff{}W_t^i + \rho(t,\nu_t^N,\mathfrak{f}_t^N)\diff{}W_t^0\right)\notag\\
		&+ \diff\sum_{k \ge 1} \phi(-\xi_k^i)\ind_{[0,t]}(\tau_k^i + \varsigma_k^i) - \phi(0)\diff J_t^{D,i}. \label{eq: GENERALISED ITO FORMULA FOR MCKEAN VLASOV SDE SIMPLIFIED}
	\end{align}
	It is clear that $\phi(X_t^i)\ind_{\{X_t^{i} < 0\}} = \phi(X_t^i) - \phi(0)\ind_{\{X_t^{i} = 0\}}$ as $X^i$ lives on $(-\infty,0]$. Furthermore, by Remark \ref{rem: REMARK ON THE RELATIONSHIP BETWEEN NU, THE FIRING FUNCTION AND THE DELAYED FIRING FUNCTION}, $J_t^i = \ind_{\{X_t^i = 0\}} + J_t^{D,i}$. Therefore, we may rewrite \eqref{eq: GENERALISED ITO FORMULA FOR MCKEAN VLASOV SDE SIMPLIFIED} as
	\begin{align}
		\diff{}\phi(X_t^i)\ind_{\{X_t^i < 0\}} &= \phi^{\prime}(X_{t}^i)b(t,\,X_t^i,\nu_t^{N},\mathfrak{f}_t^N)\ind_{\{X_t^i < 0 \}}\diff{}t +\frac{1}{2}\phi^{\prime\prime}(X_{t}^i)\sigma^2(t,X_t^i)\ind_{\{X_t^i < 0 \}}\diff{}t\notag\\
		&+ \phi^{\prime}(X_{t}^i)\sigma(t,X_t^i)\ind_{\{X_t^i < 0\}}\left(\sqrt{1 - \rho^2(t,\nu_t^N,\mathfrak{f}_t^N)}\diff{}W_t^i + \rho(t,\nu_t^N,\mathfrak{f}_t^N)\diff{}W_t^0\right)\notag\\
		&+ \diff\sum_{k \ge 1} \phi(-\xi_k^i)\ind_{[0,t]}(\tau_k^i + \varsigma_k^i) - \phi(0)\diff J_t^i. \label{eq: GENERALISED ITO FORMULA FOR MCKEAN VLASOV SDE SIMPLIFIED AND INCLUDING THE INDICATOR THAT IT IS LESS THAN 0}
	\end{align}
	After averaging over $i \in \{1,\,\ldots,\,N\}$, we obtain \eqref{eq: STOCHASTIC EVOLUTION EQUATION}.
\end{proof}

The next step is to show the existence of limit points. We show that the set of empirical measures is tight in a suitable topological space; that is, the space of $\mathscr{S}^\prime$-valued \cadlag processes on $[0,T]$, equipped with the $M_1$-topology, as done in \citep{hamblyseanhalfline,hambly2019spde}. Having obtained the evolution equation \eqref{eq: STOCHASTIC EVOLUTION EQUATION} for the empirical distribution, observe that we can decompose it into a continuous part, the instantaneous firing function, and the delayed reinsertion firing function. We write this decomposition as
\begin{equation}\label{eq: DECOMPOSITON OF EMPIRICAL DISTRIBUTION INTO CONTINUOUS PART AND JUMPS}
	\langle\nu_t^N,\phi\rangle = \langle\hat{\nu}_t^N,\phi\rangle - \phi(0)F_t^N + \frac{1}{N}\sum_{i = 1}^N\sum_{k \ge 1} \phi(-\xi_k^i)\ind_{[0,t]}(\tau_k^i+ \varsigma_k^i)
\end{equation}

\subsection{Tightness of the Finite Particle System}\label{sec: TIGHTNESS OF THE FINITE PARTICLE SYSTEM IN THE SPDE NEURON MODELS SECTION}

We first obtain control on the increments of $\hat{\nu}_t^N$.

\begin{proposition}\label{prop: LEMMA ON THE CONTROL OF THE INCREMENTS OF THE PROJECTIONS}
	For all $s,\,t \in [0,T]$ and $\phi \in \mathcal{C}_b^2$, it holds uniformly in $N \ge 1$ and $i \in \{1,\ldots,N\}$ that
	\begin{equation*}
		\E\left[\nnnorm{\langle\hat{\nu}_t^N - \hat{\nu}_s^N,\phi\rangle}^4\right] \le C_T\norm{\partial_x\phi}_{\mathcal{C}^1}^4 \nnnorm{t - s}^2.
	\end{equation*}
	Furthermore, uniformly in $N$,
	\begin{equation*}
		\E\left[\sup_{t \le T}\nnnorm{\langle\hat{\nu}_t^N,\phi\rangle}^4\right] \le C_T\norm{\partial_x\phi}_{\mathcal{C}^1}^4.
	\end{equation*}
\end{proposition}

\begin{proof}
	The proof follows from repeated applications of the Burkholder--Davis--Gundy inequality and Jensen's inequality, where we rely on \supcor{2.4} to ensure the bound holds uniformly in $N$.
\end{proof}

With control on the increments of $\hat{\nu}^N$ that holds uniformly in $N$, we are able to show that the individual components from the decomposition in \eqref{eq: DECOMPOSITON OF EMPIRICAL DISTRIBUTION INTO CONTINUOUS PART AND JUMPS} are tight. This is the content of the following proposition.

\begin{proposition}\label{prop: TIGHTNESS OF THE TUPLE OF THE INDIVIDUAL COMPONENTS OF THE DECOMPOSITION OF NU N APPLIED TO PHI}
	\sloppy{The sequences $\{N^{-1}\sum_{i = 1}^N\sum_{k \ge 1} \phi(-\xi_k^i)\ind_{[0,t]}(\tau_k^i+ \varsigma_k^i)\}_{N \ge 1}$ and $\{\langle\hat{\nu}_t^N,\phi\rangle - \phi(0)F_t^N\}_{N \ge 1}$ are tight on $(D_\R,M_1)$.}
\end{proposition}
\begin{proof}
	To facilitate the presentation of this proof, we introduce the following notation:
	\begin{equation}\label{eq: DEINITION OF MATHCAL G1N AND MATHCAL G2N DEFINITIONS}
		\mathcal{G}_t^{1,N} \coloneqq \langle\hat{\nu}_t^N,\phi\rangle - \phi(0)F_t^N \quad \textnormal{and} \quad \mathcal{G}_t^{2,N} \coloneqq \frac{1}{N}\sum_{i = 1}^N\sum_{k \ge 1} \phi(-\xi_k^i)\ind_{[0,t]}(\tau_k^i+ \varsigma_k^i).
	\end{equation}
	As $\phi$ is fixed, we do not include the explicit dependence of $\mathcal{G}^{1,N}$ and $\mathcal{G}^{2,N}$ on $\phi$ in this notation. We first consider the sequence $\{\mathcal{G}^{1,N}\}_{N \ge 1}$. We will verify that $\{\mathcal{G}^{1,N}\}_{N \ge 1}$ satisfies the conditions of \citep[Theorem~12.12.3]{whitt2002stochastic}. The first condition is satisfied, as for any $\lambda > 0$ we have
	\begin{equation*}
		\prob\left[\sup_{t \le T}\nnnorm{\mathcal{G}_t^{1,N}} > 2 \lambda \right] \le \prob\left[\sup_{t \le T}\nnnorm{\langle\hat{\nu}_t^N,\phi\rangle} > \lambda \right] + \prob\left[\nnnorm{F_T^N} > \lambda/\norm{\phi}_\infty \right] = O(\lambda^{-1}),
	\end{equation*}
	where the last equality follows from Markov's inequality and holds uniformly in $N$ by \cref{prop: BOUNDS ON THE SUP OF THE P TH MOMENTS OF THE SUP NORM OF X} and \cref{prop: LEMMA ON THE CONTROL OF THE INCREMENTS OF THE PROJECTIONS}. As $F^N$ is monotone, it is immaterial to the $M_1$-oscillatory function. Hence, for the second condition, by \citep[Theorem~1]{AVRAM198963}, it is sufficient to verify that
	\begin{equation}\label{eq: TIGHTNESS PROOF OF THE TUPLE EQUATION ONE}
		\Expec{\nnnorm{\langle\hat{\nu}_t^N,\phi\rangle - \langle\hat{\nu}_s^N,\phi\rangle}^4} = O(\nnnorm{t-s}^2) \qquad \forall s,t \in [0,T]
	\end{equation}
	and that, for every $\varepsilon > 0$,
	\begin{equation}\label{eq: TIGHTNESS PROOF OF THE TUPLE EQUATION TWO}
		\lim_{\delta \to 0} \limsup_{N \to \infty} \prob\left[\sup_{t \in (0,\delta)} \nnnorm{\mathcal{G}_t^{1,N} - \mathcal{G}_0^{1,N}} + \sup_{t \in (T-\delta,T)} \nnnorm{\mathcal{G}_t^{1,N} - \mathcal{G}_T^{1,N}} > \varepsilon\right] = 0.
	\end{equation}
	This approach is similar to that used in \citep[Proposition~4.2]{hambly2019spde}, \citep[Proposition~4.1, Proposition~4.2]{ledger2016skorokhod} and \citep[Proposition~5.1]{hamblyseanhalfline}. Equation \eqref{eq: TIGHTNESS PROOF OF THE TUPLE EQUATION ONE} has been shown in \cref{prop: LEMMA ON THE CONTROL OF THE INCREMENTS OF THE PROJECTIONS}. By the decomposition of $\mathcal{G}^{1,N}$ given in \eqref{eq: DEINITION OF MATHCAL G1N AND MATHCAL G2N DEFINITIONS},
	\begin{align*}
		\prob\left[\sup_{t \in (0,\delta)} \nnnorm{\mathcal{G}_t^{1,N} - \mathcal{G}_0^{1,N}} > \varepsilon/2\right] 
		\le& \prob\left[\sup_{t \in (0,\delta)} \nnorm{\langle\hat{\nu}_t^N,\phi\rangle - \langle\hat{\nu}_0^N,\phi\rangle}> \varepsilon/4\right]\\
		&+ \prob\left[F_\delta^N > \frac{\varepsilon}{4\norm{\phi}_\infty}\right]
	\end{align*}
	and likewise for the supremum over $t\in (T - \delta,T)$. Therefore, by first employing Markov's inequality, \eqref{eq: TIGHTNESS PROOF OF THE TUPLE EQUATION TWO} follows from \supprop{4.4} and \cref{prop: LEMMA ON THE CONTROL OF THE INCREMENTS OF THE PROJECTIONS}. Hence, $\{\mathcal{G}^{1,N}\}_{N \ge 1}$ is tight on $(D_\R,M_1)$.
	
	To show the tightness of $\{\mathcal{G}^{2,N}\}_{N \ge 1}$, we shall also show that the sequence satisfies the conditions of \citep[Theorem~12.12.3]{whitt2002stochastic}. However, here the method deviates from the literature cited above. We first observe that for any $t,s \ge 0$,
	\begin{equation}\label{eq: UPPERBOUNDS ON G2N AND A UPPERBOUND ON THE INCREMENTS OF G2N}
		\nnnorm{\mathcal{G}_t^{2,N}} \le \norm{\phi}_\infty F_t^{D,N} \qquad \textnormal{ and } \qquad \nnnorm{\mathcal{G}_t^{2,N} - \mathcal{G}_s^{2,N}} \le \norm{\phi}_\infty \nnnorm{F_t^{D,N} - F_s^{D,N}}.
	\end{equation}
	The first condition is satisfied, as for any $\lambda > 0$ we have
	\begin{equation*}
		\prob\left[\sup_{t \le T}\nnnorm{\mathcal{G}_t^{2,N}} > \lambda \right] \le \prob\left[F_T^{D,N} > \lambda/\norm{\phi}_{\infty} \right] = O(\lambda^{-1}),
	\end{equation*}
	where the last equality follows from Markov's inequality and holds uniformly in $N$ by \cref{prop: BOUNDS ON THE SUP OF THE P TH MOMENTS OF THE SUP NORM OF X}. For the second condition, we need to work slightly harder as the increments are of order 1 in $t$ and $s$ by \supprop{4.2}. Hence, we cannot immediately deduce that the condition is satisfied by applying \citep[Theorem~1]{AVRAM198963}. To verify that the second condition of \citep[Theorem~12.12.3]{whitt2002stochastic} is satisfied, we must show that
	\begin{equation}\label{eq: FIRST EQUATION TO SHOW THE SECOND CONDITION FOR THE SECOND PART OF THE DECOMPOSTION OF NU N}
		\lim_{\delta \to 0} \limsup_{N \to \infty} \prob\left[w_{\delta}(\mathcal{G}^{2,N}) \ge \varepsilon \right] = 0,
	\end{equation}
	and
	\begin{equation}\label{eq: SECOND EQUATION TO SHOW THE SECOND CONDITION FOR THE SECOND PART OF THE DECOMPOSTION OF NU N}
		\lim_{\delta \to 0} \limsup_{N \to \infty} \prob\left[\sup_{t \in (0,\delta)} \nnnorm{\mathcal{G}_t^{2,N} - \mathcal{G}_0^{2,N}} + \sup_{t \in (T-\delta,T)} \nnnorm{\mathcal{G}_t^{2,N} - \mathcal{G}_T^{2,N}} > \varepsilon\right] = 0
	\end{equation}
	for any $\varepsilon > 0$, where $w_\delta$ is the $M_1$-oscillatory function defined by
	\begin{align*}
		w_\delta(x) \coloneqq \sup \{H_x(t_1,t_2,t_3)\,:\, (t - \delta)\vee 0 \le t_1\le t_2\le t_3 \le (t + \delta) \wedge T,\,0 \le t \le T \}, 
	\end{align*}
	where $H_x(t_1,t_2,t_3) \coloneqq \inf_{\lambda \in [0,1]} \nnnorm{x_{t_2} - (1 - \lambda)x_{t_1} - \lambda x_{t_3}}$ for any $x \in D_\R$. Equation \eqref{eq: SECOND EQUATION TO SHOW THE SECOND CONDITION FOR THE SECOND PART OF THE DECOMPOSTION OF NU N} follows immediately from \eqref{eq: UPPERBOUNDS ON G2N AND A UPPERBOUND ON THE INCREMENTS OF G2N} and \supprop{4.1}. 
	
	To show \eqref{eq: FIRST EQUATION TO SHOW THE SECOND CONDITION FOR THE SECOND PART OF THE DECOMPOSTION OF NU N}, the strategy is as follows. We shall construct a piecewise linear function that approximates $\mathcal{G}^{2,N}$. Our approximations will be sufficiently close to $\mathcal{G}^{2,N}$ using the sup-norm, and we will have sufficient control over the $M_1$-oscillatory function of the approximations. To begin, we define $t_j^N \coloneqq jT/N$. Therefore, the piecewise linear approximation to $\mathcal{G}^{2,N}$ is
	\begin{equation*}
		\tilde{\mathcal{G}}^{2,N} \coloneqq \sum_{j = 0}^{N-1} \mathcal{G}_{t_j^N}^{2,N}\ind_{[t_j^N,t_{j+1}^N)} + \mathcal{G}_{T}^{2,N}\ind_{\{T\}}.
	\end{equation*} 
	We observe that for any $0 \le t_1 \le t_2 \le t_3 \le T$,
	\begin{align*}
		H_{\mathcal{G}^{2,N}}(t_1,t_2,t_3) &= \inf_{\lambda \in [0,1]} \nnnorm{\mathcal{G}^{2,N}_{t_2} - (1 - \lambda)\mathcal{G}^{2,N}_{t_1} - \lambda \mathcal{G}^{2,N}_{t_3}}\\
		&\le \inf_{\lambda \in [0,1]} \nnnorm{\tilde{\mathcal{G}}^{2,N}_{t_2} - (1 - \lambda)\tilde{\mathcal{G}}^{2,N}_{t_1} - \lambda \tilde{\mathcal{G}}^{2,N}_{t_3}} + 3\norm{\mathcal{G}^{2,N}-\tilde{\mathcal{G}}^{2,N}}_{\infty}\\
		&\le H_{\tilde{\mathcal{G}}^{2,N}}(t_1,t_2,t_3) + 3\norm{\mathcal{G}^{2,N}-\tilde{\mathcal{G}}^{2,N}}_{\infty}.
	\end{align*}
	Hence, $w_{\delta}(\mathcal{G}^{2,N}) \le w_{\delta}(\tilde{\mathcal{G}}^{2,N}) + 3\norm{\mathcal{G}^{2,N}-\tilde{\mathcal{G}}^{2,N}}_{\infty}$. Now, to estimate the second term, by definition, for any $t \in [t_j^N,t_{j+1}^N)$,
	\begin{equation*}
		\nnnorm{\mathcal{G}_t^{2,N}-\tilde{\mathcal{G}}_t^{2,N}} = \nnnorm{\mathcal{G}_t^{2,N}- \mathcal{G}_{t_j^N}^{2,N}} \le \norm{\phi}_\infty \nnnorm{F_t^{D,N} - F_{t_j^N}^{D,N}} \le \norm{\phi}_\infty \nnnorm{F_{t_{j+1}^N}^{D,N} - F_{t_j^N}^{D,N}},
	\end{equation*}
	where the penultimate inequality follows from \eqref{eq: UPPERBOUNDS ON G2N AND A UPPERBOUND ON THE INCREMENTS OF G2N} and the last inequality follows from the monotonicity of $F^{D,N}$. Therefore, we may conclude that
	\begin{align*}
		\E\left[\norm{\mathcal{G}^{2,N}-\tilde{\mathcal{G}}^{2,N}}_{\infty}^2\right] &\le \sum_{j = 0}^{N-1}\norm{\phi}_\infty^2  \E \left[\nnnorm{F_{t_{j+1}^N}^{D,N} - F_{t_j^N}^{D,N}}^2\right]\\
		&\le \sum_{j = 0}^{N-1} C \norm{\phi}_\infty^2(\nnnorm{t_{j+1}^N-t_{j}^N}N^{-1} + \nnnorm{t_{j+1}^N- t_{j}^N}^{1 + \beta})\\
		& \le C(N^{-1} + N^{-\beta}),
	\end{align*}
	where the penultimate inequality follows from \supprop{4.2} and the constant $C$ is independent of $N$. By applying Markov's inequality,
	\begin{equation}\label{eq: LIMITS ON THE PROBABILITY OF THE SUP NORM OF THE DIFFERENCE BETWEEN TWO APPROXIMATIONS}
		\lim_{\delta \to 0} \limsup_{N \to \infty} \prob\left[\norm{\mathcal{G}^{2,N}-\tilde{\mathcal{G}}^{2,N}}_{\infty} \ge \varepsilon \right] = 0
	\end{equation}
	for any $\varepsilon > 0$. Now, to bound the oscillations of $\tilde{\mathcal{G}}^{2,N}$, choose a $t_1\le t_2 \le t_3 < T$. If we suppose there is a $j$ such that $t_i \in [t_j^N,t_{j+2}^N)$ for every $i$, then it is immediate that $H_{\tilde{\mathcal{G}}^{2,N}}(t_1,t_2,t_3) = 0$. Hence, trivially,
	\begin{equation}\label{eq: BOUNDS OF THE M1 OSCILLATORY FUNCTION OF OUR APPROXIAMTION}
		\prob\left[H_{\tilde{\mathcal{G}}^{2,N}}(t_1,t_2,t_3) \ge \varepsilon\right] \le 2 \norm{\phi}_{\infty}^2 C (T^{-1} + 1) \varepsilon^{-2}\nnnorm{t_1  -t_3}^{1+\beta}
	\end{equation} 
	uniformly in $N$ for any $\varepsilon >0$, where $C$ and $\beta$ are the constants from \supprop{4.2}. In the case when there is some $j,\tilde{j} \in [1:N]$ with $j < \tilde{j}$ and $m \ge 2$ such that $t_1 \in [t_j^N,t_{j+1}^N)$, $t_3 \in [t_{j+m}^N,t_{j+m+1}^N)$ and $t_2 \in [t_{\tilde{j}}^N,t_{{\tilde{j}}+1}^N) \subset [t_j^N,t_{j+m+1}^N)$, then as $H_{\tilde{\mathcal{G}}^{2,N}}(t_1,t_2,t_3)$ only takes the values $0,\, \nnnorm{\mathcal{G}_{t_1}^{2,N} - \mathcal{G}_{t_2}^{2,N}}$ or $\nnnorm{\mathcal{G}_{t_3}^{2,N} - \mathcal{G}_{t_2}^{2,N}}$, we observe that
	\begin{align*}
		\E\left[H_{\tilde{\mathcal{G}}^{2,N}}(t_1,t_2,t_3)^2\right] &\le \E\left[\nnnorm{\tilde{\mathcal{G}}_{t_1}^{2,N} - \tilde{\mathcal{G}}_{t_2}^{2,N}}^2 + \nnnorm{\tilde{\mathcal{G}}_{t_3}^{2,N} - \tilde{\mathcal{G}}_{t_2}^{2,N}}^2\right]\\
		&= \E\left[\nnnorm{\mathcal{G}_{t_j^N}^{2,N} - \mathcal{G}_{t_{\tilde{j}}^N}^{2,N}}^2 + \nnnorm{\mathcal{G}_{t_{j+m}^N}^{2,N} - \mathcal{G}_{t_{\tilde{j}}^N}^{2,N}}^2\right]\\
		&\le \norm{\phi}_{\infty}^2 \E\left[\nnnorm{F_{t_j^N}^{D,N} - F_{t_{\tilde{j}}^N}^{D,N}}^2 + \nnnorm{F_{t_{j+m}^N}^{D,N} - F_{t_{\tilde{j}}^N}^{D,N}}^2\right]\\
		&\le \norm{\phi}_{\infty}^2 C \Biggl[N^{-1}\nnnorm{t_j^N - t_{\tilde{j}}^N} + N^{-1}\nnnorm{t_{j+m}^N - t_{\tilde{j}}^N} \\
		&\hspace{2.2cm} + \nnnorm{t_j^N - t_{\tilde{j}}^N}^{1 + \beta} + \nnnorm{t_{j+m}^N - t_{\tilde{j}}^N}^{1 + \beta} \Biggr]\\
		&\le 2 \norm{\phi}_{\infty}^2 C \left[N^{-1}\left(\frac{(m-1)T}{N}\right) + \left(\frac{(m-1)T}{N}\right)^{1 + \beta} \right]\\
		&\le 2 \norm{\phi}_{\infty}^2 C \left[N^{-1}\nnnorm{t_3 - t_1} + \nnnorm{t_3 - t_1}^{1 + \beta} \right]\\ \intertext{As $(t_3 - t_1) > TN^{-1}$ by assumption,}
		&\le 2 \norm{\phi}_{\infty}^2 C \left[T^{-1}\nnnorm{t_3 - t_1}^2 + \nnnorm{t_3 - t_1}^{1 + \beta} \right]\\
		&\le 2 \norm{\phi}_{\infty}^2 C (T^{-1} + 1)\nnnorm{t_3 - t_1}^{1 + \beta}.
	\end{align*}
	The above computation also holds in the case when $t_3 = T$. Therefore, \eqref{eq: BOUNDS OF THE M1 OSCILLATORY FUNCTION OF OUR APPROXIAMTION} holds for any $0 \le t_1 \le t_2 \le t_3 \le T$. In view of \eqref{eq: BOUNDS OF THE M1 OSCILLATORY FUNCTION OF OUR APPROXIAMTION}, the assumptions of \citep[Theorem~1]{AVRAM198963} are satisfied, and we may conclude that
	\begin{equation}\label{eq: BOUNDS ON THE PROBABILITY OF THE M1 OSCILLATIONS OF APPROXIAMTIONS}
		\prob\left[w_{\delta}(\tilde{\mathcal{G}}^{2,N}) \ge \varepsilon \right] = O(\delta^\beta\varepsilon^{-2})
	\end{equation}
	uniformly in $N$ for any $\varepsilon > 0$. By combining \eqref{eq: LIMITS ON THE PROBABILITY OF THE SUP NORM OF THE DIFFERENCE BETWEEN TWO APPROXIMATIONS} and \eqref{eq: BOUNDS ON THE PROBABILITY OF THE M1 OSCILLATIONS OF APPROXIAMTIONS}, we deduce that \eqref{eq: FIRST EQUATION TO SHOW THE SECOND CONDITION FOR THE SECOND PART OF THE DECOMPOSTION OF NU N} holds. This verifies the second condition of \citep[Theorem~12.12.3]{whitt2002stochastic}, and hence $\{\mathcal{G}^{2,N}\}_{N \ge 1}$ is tight.
\end{proof}

As the components from decomposition \eqref{eq: DECOMPOSITON OF EMPIRICAL DISTRIBUTION INTO CONTINUOUS PART AND JUMPS} are tight, we may use this result to deduce the tightness of their sum.

\begin{proposition}\label{prop: TIGTHNESS OF THE TUPLE WITH THE MEASURES FIRIING FUNCTION AND COMMON BROWNIAN MOTION}
	The sequence $\{(\nu^N,F^N,W^0)\}_{N\ge 1}$ is tight on $(D_{\mathscr{S}^\prime},M_1) \times (D_{\R},M_1) \times (\mathcal{C}_\R,\norm{\cdot}_{\infty})$. Furthermore, any subsequence of $\{(\nu^N,F^N,W^0)\}_{N \ge 1}$ has a further subsequence that converges weakly to a limit point $(\nu^*,F^*,W^0)$.
\end{proposition}
\begin{proof}
	As tightness of the marginals implies tightness of the joint distribution, it will be sufficient to show that each component is individually tight. Trivially, $\{W^0\}_{N}$ is tight as $W^0$ does not change with $N$. We now turn our attention to $\{F^N\}_N$. As monotone functions are immaterial to the $M_1$-oscillatory function, by \supprop{4.4} and \cref{prop: BOUNDS ON THE SUP OF THE P TH MOMENTS OF THE SUP NORM OF X}, we may conclude that $\{F^N\}_N$ satisfies the assumptions of \citep[Theorem~12.12.3]{whitt2002stochastic}. Therefore, $\{F^N\}_N$ is tight.
	
	For the tightness of $\{\nu^N\}_N$, it suffices to show that $\{\nu^N(\phi)\}_N$ is tight in $(D_{\R},M_1)$ for any $\phi \in \mathscr{S}$. Since $(D_{\R},M_1)$ is a Polish space, by Prokhorov's Theorem, tightness is equivalent to sequential pre-compactness. By \eqref{eq: DEINITION OF MATHCAL G1N AND MATHCAL G2N DEFINITIONS}, we may write $\nu^N(\phi) = \mathcal{G}^{1,N} + \mathcal{G}^{2,N}$. The sequence $\{(\mathcal{G}^{1,N},\mathcal{G}^{2,N})\}_N$ is tight in $(D_{\R},M_1) \times (D_{\R},M_1)$ by \cref{prop: TIGHTNESS OF THE TUPLE OF THE INDIVIDUAL COMPONENTS OF THE DECOMPOSITION OF NU N APPLIED TO PHI}. Therefore, for any subsequence $N_j$, we may choose a further subsequence, also denoted by $N_j$ for simplicity, such that $(\mathcal{G}^{1,N},\mathcal{G}^{2,N})$ converges weakly to $ (\mathcal{G}^{1,*},\mathcal{G}^{2,*})$. An application of the Continuous Mapping Theorem, the Portmanteau Theorem, and \supprop{4.2} allows us to deduce that
	\begin{equation}\label{eq: BOUNDS ON INCREMENTS OF THE DELAYED FIRING FUNCTION IN THE TIGHTNESS PROOF}
		\E\nnnorm{\mathcal{G}^{2,*}_t - \mathcal{G}^{2,*}_s}^2 \le C\nnnorm{t-s}^{1 + \beta}
	\end{equation}
	for a co-countable set of times $t,s$ with constants $C$ and $\beta$ that are independent of $t$ and $s$. By Kolmogorov's Continuity Theorem, $\mathcal{G}^{2,*}$ has continuous paths almost surely. Therefore, addition is a $\operatorname{Law}((\mathcal{G}^{1,*},\mathcal{G}^{2,*}))$-almost sure continuous map by the continuity of $\mathcal{G}^{2,*}$ and \citep[Theorem~12.7.3]{whitt2002stochastic}. Hence, by the Continuous Mapping Theorem,
	\begin{equation*}
		\nu^{N_j}(\phi) = \mathcal{G}^{1,N_j} + \mathcal{G}^{2,N_j} \implies \mathcal{G}^{1,*} + \mathcal{G}^{2,*}.
	\end{equation*}
	Therefore, $\{\nu^N(\phi)\}_N$ is sequentially pre-compact and hence, by the Portmanteau Theorem, tight in $(D_\R,M_1)$.
	
	\sloppy{For the second claim, the proof follows along the lines of \citep[Theorem~3.2]{ledger2016skorokhod}. Both $(D_{\mathscr{S}^\prime},M_1)$ and any Polish space are topological Radon spaces, by \citep[Section~3,Exercise~4]{smolyanov1976measures} and \citep[Section~3,Exercise~1]{smolyanov1976measures}, respectively. Hence, $(D_{\mathscr{S}^\prime},M_1) \times (D_{\R},M_1) \times (\mathcal{C}_\R,\norm{\cdot}_{\infty})$ is a topological Radon space. As every compact subset of $(D_{\mathscr{S}^\prime},M_1)$ is metrizable (\citep[Theorem~3.1 (iii)]{ledger2016skorokhod}) and it is a completely regular space (\citep[Proposition~2.7 (iii)]{ledger2016skorokhod}), it is clear that every compact subset of $(D_{\mathscr{S}^\prime},M_1) \times (D_{\R},M_1) \times (\mathcal{C}_\R,\norm{\cdot}_{\infty})$ is metrizable, and it is a completely regular space as Polish spaces are completely regular spaces and the product of completely regular spaces is a completely regular space. Therefore, the second claim follows from \citep[Section~5,Theorem~2]{smolyanov1976measures}.}
\end{proof}

\subsection{Limit Point Representation and Further Properties}\label{sec: LIMIT POINT REPRESENTATION AND PROPERTIES IN THE NEURO SPDE PAPER} 
In view of \cref{prop: TIGTHNESS OF THE TUPLE WITH THE MEASURES FIRIING FUNCTION AND COMMON BROWNIAN MOTION}, we fix a limit point and denote it by $(\nu^*,F^*,W^0)$. We now aim to show that the effects from the reinsertion of the particles are smoothed out by the refractory period, and we see that particles are continuously reinserted into the system in the limit.

\begin{proposition}[Limit point representation]\label{prop: LIMIT POINT REPRESENTATION OF THE DELAYED FIRING FUNCTION}
	Fix a $\phi \in \mathscr{S}$ and let $(\nu^N,F^N,W^0)$ be a subsequence, still indexed by $N$, along which $(\nu^N,F^N,W^0) \implies (\nu^*,F^*,W^0)$. Then on $(D_{\R},M_1) \times (D_{\R},M_1)$, we have $(F^N,N^{-1}\sum_{i = 1}^N\sum_{k \ge 1} \phi(-\xi_k^i)\ind_{[0,t]}(\tau_k^i+ \varsigma_k^i)) \implies (F^*,\E[\phi(-\xi)]\int_0^\cdot \pref(\cdot - s)F^*_s \diff s)$.
\end{proposition}

\begin{proof}
	To facilitate the presentation of this proof, we introduce the following notation:
	\begin{align*}
		&F_t^{N} = \frac{1}{N}\sum_{i = 1}^N\sum_{k \ge 1} \ind_{[0,t]}(\tau_k^i)
		&\tilde{F}_t^{D,N} &= \frac{1}{N}\sum_{i = 1}^N\sum_{k \ge 1} \phi(-\xi_k^i)\ind_{[0,t]}(\tau_k^i+ \varsigma_k^i)\\
		&F_t^{N,k} = \frac{1}{N}\sum_{i = 1}^N \ind_{[0,t]}(\tau_k^i)
		&\tilde{F}_t^{D,N,k} &= \frac{1}{N}\sum_{i = 1}^N \phi(-\xi_k^i)\ind_{[0,t]}(\tau_k^i+ \varsigma_k^i)
	\end{align*}
	The strategy is as follows: we want to show that for any subsequence of $(F^N,\tilde{F}^{D,N})$, we have a further subsequence that converges weakly towards $(F^*,\E[\phi(-\xi)](\pref \ast F^*))$, where we use $(\pref \ast F^*)$ to denote $\int_0^\cdot \pref(\cdot - s)F^*_s \diff s$. To achieve this, we employ a combination of the Portmanteau Theorem and the Continuous Mapping Theorem. We will exploit that the probability of $\tau_k^i$ being less than $T$ decays exponentially in $k$, which allows us to work with finitely many terms when controlling the sums $F^{N}$ and $\tilde{F}^{D,N}$.
	
	We first turn our attention to $\{(F^{N,k})_{k \ge 1}\}_N$. We observe that for any $t \ge s$, $F_t^{N,k} - F_s^{N,k} \le F_t^{N} - F_s^{N}$ and $F_T^{N,k}\le F_T^{N}$. Hence,
	\begin{equation}\label{eq: CONTROL ON THE INCREMENTS OF THE KTH FIRING FUNCITON IN THE TIGHTNESS PROOF}
		\lim_{h \to  0} \limsup_{N \to  +\infty} \prob\left[F_{t+h}^{N,k} - F_t^{N,k} \ge \eta\right] = 0 \quad \textnormal{and} \quad \lim_{h \to  0} \limsup_{N \to  +\infty} \prob\left[F_{t}^{N,k} - F_{t-h}^{N,k} \ge \eta\right] = 0
	\end{equation}
	for any $t \in [0,T]$ by \supprop{4.4}. As monotone functions are immaterial to the $M_1$-oscillatory function, by \cref{prop: BOUNDS ON THE SUP OF THE P TH MOMENTS OF THE SUP NORM OF X} and \eqref{eq: CONTROL ON THE INCREMENTS OF THE KTH FIRING FUNCITON IN THE TIGHTNESS PROOF}, we may conclude that $\{F^{N,k}\}_N$ satisfies the assumptions of \citep[Theorem~12.12.3]{whitt2002stochastic}. Therefore, $\{F^{N,k}\}_N$ is tight in $(D_{\R},M_1)$. By tightness, for every $k$ (fixed) and $\varepsilon > 0$, we may find $A^k \subset D_\R$, compact, such that $\prob\left[F^{N,k}\not \in A^k\right] < \varepsilon 2^{-k}$, uniformly in $N$. The countable product of compact sets is compact in the product topology. Hence, $\prod_{k \ge 1} A^k$ is compact in $(\prod_{k \ge 1}D_{\R},M_1)$, where $(\prod_{k \ge 1}D_{\R},M_1)$ is endowed with the product topology, and $\prob\left[(F^{N,k})_{k \ge 1}\not \in \prod_{k \ge 1}A^k\right] < \varepsilon$ uniformly in $N$. Hence, $\{(F^{N,k})_{k \ge 1}\}_N$ is tight.
	
	Furthermore, $\{F^N\}_N$ and $\{\tilde{F}^{D,N}\}_N$ are tight by \cref{prop: TIGHTNESS OF THE TUPLE OF THE INDIVIDUAL COMPONENTS OF THE DECOMPOSITION OF NU N APPLIED TO PHI} and \cref{prop: TIGTHNESS OF THE TUPLE WITH THE MEASURES FIRIING FUNCTION AND COMMON BROWNIAN MOTION} respectively. Hence, for any subsequence, $N_j$, there exists a further subsequence, also denoted by $N_j$ for simplicity, such that $(F^{N_j},\tilde{F}^{D,N_j})$ converges weakly to $(F^{\dag},\tilde{F}^{D,\dag})$ on $(D_{\R},M_1) \times (D_{\R},M_1)$ and $\{(F^{N_j,k})_k\}_j$ converges weakly to $(F^{\dag,k})_k$ on $(\prod_{k \ge 1}D_{\R},M_1)$. 
	
	We now aim to show that $(F^{\dag},\tilde{F}^{D,\dag})$ has the same distribution as $(F^*,\E[\phi(-\xi)](\pref \ast F^*))$. As the Kolmogorov $\sigma$-Algebra coincides with the Borel $\sigma$-Algebra on $D_\R$ \cite[Theorem~11.5.1]{whitt2002stochastic}, it is sufficient to show that for a dense set of times, $\mathbb{T}$, the following holds for every $m \ge 1$:
	\begin{equation*}
		\E\left[\prod_{\ell = 1}^m g_\ell(F_{t_\ell}^\dag)h_\ell(\tilde{F}_{t_\ell}^{D,\dag})\right] = \E\left[\prod_{\ell = 1}^m g_\ell(F_{t_\ell}^*)h_\ell(\E[\phi(-\xi)](\pref \ast F^*)_{t_\ell})\right],
	\end{equation*}
	for $t_1,\ldots,t_m \in \mathbb{T}$ and $g_1,h_1,\ldots,g_m,h_m \in \mathcal{C}_b(\R) \cap \operatorname{Lip}(\R)$. To this end, we set
	\begin{equation*}
		\mathbb{T} \coloneqq \left\{t \in [0,T]\,:\, \prob\left[F_t^\dag = F_{t-}^\dag\right] =  \prob\left[ F_{t}^{*} = F_{t-}^{*}\right] = \prob\left[F_{t}^{\dag,k} = F_{t-}^{\dag,k}\right] = 1\, \forall \, k\right\}.
	\end{equation*}
	$\mathbb{T}$ is co-countable, hence dense in $[0,T]$. We observe that for any $N_j$ and $M > 0$,
	{\small\begin{align*}
			&\nnnorm{ \E\left[\prod_{\ell = 1}^m g_\ell(F_{t_\ell}^\dag)h_\ell(\tilde{F}_{t_\ell}^{D,\dag})\right] - \E\left[\prod_{\ell = 1}^m g_\ell(F_{t_\ell}^*)h_\ell(\E[\phi(-\xi)](\pref \ast F^*)_{t_\ell})\right]} \notag \\
			&  \le \nnnorm{ \E\left[\prod_{\ell = 1}^m g_\ell(F_{t_\ell}^\dag)h_\ell(\tilde{F}_{t_\ell}^{D,\dag})\right] - \E\left[\prod_{\ell = 1}^m g_\ell(F_{t_\ell}^{N_j})h_\ell(\tilde{F}_{t_\ell}^{D,N_j})\right]} \notag \\
			&  + \nnnorm{ \E\left[\prod_{\ell = 1}^m g_\ell(F_{t_\ell}^{N_j})h_\ell(\tilde{F}_{t_\ell}^{D,N_j})\right] - \E\left[\prod_{\ell = 1}^m g_\ell\left( F_{t_\ell}^{N_j,\Sigma}\right)h_\ell\left(\tilde{F}_{t_\ell}^{D,N_j,\Sigma}\right)\right]} \notag \\
			&  + \nnnorm{ \E\left[\prod_{\ell = 1}^m g_\ell\left( F_{t_\ell}^{N_j,\Sigma}\right)h_\ell\left( \tilde{F}_{t_\ell}^{D,N_j,\Sigma}\right)\right] - \E\left[\prod_{\ell = 1}^m g_\ell\left( F_{t_\ell}^{\dag,\Sigma}\right)h_\ell\left(\E[\phi(-\xi)] \left(\pref \ast  F^{\dag,\Sigma}\right)_{t_\ell}\right)\right]} \notag \\
			& + \nnnorm{ \E\left[\prod_{\ell = 1}^m g_\ell\left( F_{t_\ell}^{\dag,\Sigma}\right)h_\ell\left(\E[\phi(-\xi)] \left(\pref \ast  F^{\dag,\Sigma}\right)_{t_\ell}\right)\right] - \E\left[\prod_{\ell = 1}^m g_\ell(F_{t_\ell}^*)h_\ell(\E[\phi(-\xi)](\pref \ast F^*)_{t_\ell}) \right]} \notag \\
			& = \Romannum{1} + \Romannum{2} + \Romannum{3} + \Romannum{4},
	\end{align*}}
	where we employed the notation $a^\Sigma$ to denote the finite sum $\sum_{k = 1}^M a^k$ for any sequence $(a^k)_k$. We suppress the dependence on $M$ for simplicity.
	
	Now, fix an arbitrary $\varepsilon > 0$. We have
	\begin{equation}\label{eq: FIRST EQUATION IN THE CHARACTERISATION OF LIMITI POINTS PROOF OF F}
		\begin{aligned}
			\Romannum{2} 
			&\le c \sum_{\ell = 1}^m \E \left[\nnnorm{F_{t_\ell}^{N_j} -  F_{t_\ell}^{N_j,\Sigma}} + \nnnorm{\tilde{F}_{t_\ell}^{D,N_j} -  \tilde{F}_{t_\ell}^{D,N_j,\Sigma}}\right] \\
			&\le 2c(1 + \norm{\phi}_{\infty}) \E\left[\sum_{k > M} F_T^{N_j,k}\right],
		\end{aligned}
	\end{equation}
	where $c = \max_{\ell}\{\norm{g_\ell}_\infty,\,\norm{g_\ell}_{\operatorname{Lip}},\,\norm{h_\ell}_\infty,\,\norm{h_\ell}_{\operatorname{Lip}}\}^{2m}$. By \cref{prop: EXPONENTIAL DECAY OF THE HITTING PROBABILITIES}, we may fix an $M$ large enough such that, independent of $N_j$, $\E\left[\sum_{k > M} F_T^{N_j,k}\right] < \varepsilon (8c + 8c\norm{\phi}_{\infty})^{-1}$. Hence, $\Romannum{2} < \varepsilon/4$. Now, as $M$ is independent of $N_j$, we may choose a $j$ large enough to control the other terms. As $t_1,\ldots,t_m \in \mathbb{T}$, by the Continuous Mapping Theorem and weak convergence, there exists a $J_1$ such that for all $j \ge J_1$, we have $\Romannum{1} < \varepsilon/4$.
	
	To control $\Romannum{3}$, we will consider the infinite tuple of (random) measures $(\mu^{N,k})_{k \ge 1}$ on $\prod_{k \ge 1} \mathcal{P}(\Bar{\R}^3)$. Here $ \mu^{N,k} \coloneqq N^{-1} \sum_{i = 1}^N \delta_{\xi_k^i,\tau_k^i,\varsigma_k^i}$ and $\Bar{\R}$ is the one-point compactification of $\R$. Trivially $\{(\mu^{N,k})_{k}\}$ is tight. Without loss of generality, we may suppose that our further subsequence, recall also denoted by $N_j$ for simplicity, is such that $(\mu^{N_j,k})_{k}$ converges weakly towards $(\mu^{\dag,k})_{k}$. By \cref{prop: ASYMPTOTIC DECORRELATION OF THE HITTING TIMES RESETS AND RANDOM WAITING TIMES}, we have that $\mu^{\dag,k} = \prob[\xi \in \diff x] \otimes \mu_2^{\dag,k}(\diff y)\otimes \prob[\varsigma \in \diff z]$ almost surely. We now suppose further that the times $t_1,\ldots, t_m \in \tilde{\mathbb{T}}$, where
	\begin{equation*}
		\tilde{\mathbb{T}} = \left\{ t \in [0,T]\,:\, \prob[\mu_2^{\dag,k}(\{t\}) = 0] = 1\,\forall\,k\right\} \cap \mathbb{T}.
	\end{equation*}
	$\tilde{\mathbb{T}}$ is still co-countable, as for any probability measure $\mu$, the map $t \mapsto \mu([0,t])$ is a \cadlag function. Therefore, by construction of $\tilde{\mathbb{T}}$, the map $\prod_{k \ge 1} \mathcal{P}(\Bar{\R}^3) \to \prod_{k \ge 1} \R$, $(\mu^k)_k \mapsto (\langle\mu^k,\ind_{[0,t]}(y)\rangle)_k$ is $\operatorname{Law}((\mu^{\dag,k}))$-almost surely continuous for every $t \in \tilde{\mathbb{T}}$. Furthermore, as $\mu^{\dag,k}$ may be written as a product of measures and as $\varsigma$ has a density by \cref{ass: ASSUMPTIONS FROM SOJMARK SPDE PAPER APPLIED TO THE NEUROSCIENCE SETTING} \eqref{ass: ASSUMPTIONS FROM SOJMARK SPDE PAPER APPLIED TO THE NEUROSCIENCE SETTING FOUR}, the map $\prod_{k \ge 1} \mathcal{P}(\Bar{\R}^3) \to \prod_{k \ge 1} \R$, $(\mu^k)_k \mapsto (\langle\mu^k,\ind_{[0,t]}(y+z)\rangle)_k$ is $\operatorname{Law}((\mu^{\dag,k}))$-almost surely continuous for every $t \in [0,T]$. Therefore,
	\begin{align}
		&\lim_{j \to \infty} \E\left[\prod_{\ell = 1}^m g_\ell\left( F_{t_\ell}^{N_j,\Sigma}\right)h_\ell\left( \tilde{F}_{t_\ell}^{D,N_j,\Sigma}\right)\right]\notag \\
		&\qquad  = \lim_{j \to \infty} \E\left[\prod_{\ell = 1}^m g_\ell\left( \langle \mu^{N_j,\Sigma},\ind_{[0,t_\ell]}(y)\rangle\right)h_\ell\left( \langle \mu^{N_j,\Sigma},\phi(-x)\ind_{[0,t_\ell]}(y+z)\rangle\right)\right]\label{eq: EQUATION IS THE CHARACTERISATION OF LIMITING POINTS PROOF THAT SHOWS WE HAVE A LIMITING FORM WE LIKE}\\
		&\qquad = \E\left[\prod_{\ell = 1}^m g_\ell\left( \langle \mu^{\dag,\Sigma},\ind_{[0,t_\ell]}(y)\rangle\right)h_\ell\left( \E[\phi(-\xi)]\int_0^{t_\ell}\pref(t_\ell - s)\langle \mu^{\dag,\Sigma},\ind_{[0,s]}(y)\rangle \diff s\right)\right], \notag
	\end{align}
	where the first equality follows from the definition of $F^{N,k}$ and $\tilde{F}^{D,N,k}$, and the second equality follows from the Continuous Mapping Theorem. Furthermore, by the Portmanteau Theorem and the Dominated Convergence Theorem, it is easy to check that the map
	\begin{equation*}
		\prod_{k \ge 1} \mathcal{P}(\Bar{\R}^3) \to \prod_{k \ge 1} \R, \qquad (\mu^k)_k \mapsto \left(\int_0^{t_\ell}\pref(t_\ell - s)\langle \mu^{k},\ind_{[0,s]}(y)\rangle \diff s\right)_k
	\end{equation*}
	is $\operatorname{Law}((\mu^{\dag,k}))$-almost surely continuous. Therefore, going in the reverse direction in \eqref{eq: EQUATION IS THE CHARACTERISATION OF LIMITING POINTS PROOF THAT SHOWS WE HAVE A LIMITING FORM WE LIKE},
	\begin{align}
		&\E\left[\prod_{\ell = 1}^m g_\ell\left( \langle \mu^{\dag,\Sigma},\ind_{[0,t_\ell]}(y)\rangle\right)h_\ell\left( \E[\phi(-\xi)]\int_0^{t_\ell}\pref(t_\ell - s)\langle \mu^{\dag,\Sigma},\ind_{[0,s]}(y)\rangle \diff s\right)\right]\notag\\
		& =\lim_{j \to \infty} \E\left[\prod_{\ell = 1}^m g_\ell\left( \langle \mu^{N_j,\Sigma},\ind_{[0,t_\ell]}(y)\rangle\right)h_\ell\left( \E[\phi(-\xi)]\int_0^{t_\ell}\pref(t_\ell - s)\langle \mu^{N_j,\Sigma},\ind_{[0,s]}(y)\rangle \diff s\right)\right]\notag\\
		& =\lim_{j \to \infty} \E\left[\prod_{\ell = 1}^m g_\ell\left( F_{t_\ell}^{N_j,\Sigma}\right)h_\ell\left(\E[\phi(-\xi)]\int_0^{t_\ell}\pref(t_\ell - s) F_{s}^{N_j,\Sigma}\diff s\right)\right], \label{eq: EQUATION IS THE CHARACTERISATION OF LIMITING POINTS PROOF THAT SHOWS WE HAVE A LIMITING FORM WE LIKE TWO}
	\end{align}
	where the first equality follows from the Continuous Mapping Theorem, and the second equality follows from the definition of $F^{N,k}$. Lastly, by the properties of $M_1$-convergence, the map
	\begin{equation*}
		\prod_{k \ge 1} D_\R \to \prod_{k \ge 1} \R, \qquad (f^k)_k \mapsto \left(\int_0^{t_\ell}\pref(t_\ell - s)f_s^k\diff s\right)_k
	\end{equation*}
	is continuous. Therefore, putting \eqref{eq: EQUATION IS THE CHARACTERISATION OF LIMITING POINTS PROOF THAT SHOWS WE HAVE A LIMITING FORM WE LIKE} and \eqref{eq: EQUATION IS THE CHARACTERISATION OF LIMITING POINTS PROOF THAT SHOWS WE HAVE A LIMITING FORM WE LIKE TWO} together, and applying the Continuous Mapping Theorem one last time, we deduce that
	\begin{align*}
		&\lim_{j \to \infty} \E\left[\prod_{\ell = 1}^m g_\ell\left( F_{t_\ell}^{N_j,\Sigma}\right)h_\ell\left( \tilde{F}_{t_\ell}^{D,N_j,\Sigma}\right)\right]\\
		&\qquad\qquad  =\E\left[\prod_{\ell = 1}^m g_\ell\left( F_{t_\ell}^{\dag,\Sigma}\right)h_\ell\left(\E[\phi(-\xi)]\int_0^{t_\ell}\pref(t_\ell - s) F_{s}^{\dag,\Sigma}\diff s\right)\right].
	\end{align*}
	Therefore, we may find a $J_2$ such that for all $j \ge J_2$, $\Romannum{3} < \varepsilon/4$. Setting $j = \max\{J_1,J_2\}$, we have $\Romannum{1} + \Romannum{3} < \varepsilon/2$.
	
	Now, we only need to control $\Romannum{4}$. By the Continuous Mapping Theorem and the Portmanteau Theorem, $\Romannum{4}$ is bounded by the $\limsup$ over $j$ of
	\begin{equation*}
		\nnnorm{ \E\left[\prod_{\ell = 1}^m g_\ell\left(F_{t_\ell}^{N_j,\Sigma}\right)h_\ell\left(\E[\phi(-\xi)] \left(\pref \ast F^{N_j,\Sigma}\right)_{t_\ell}\right) - g_\ell(F_{t_\ell}^{N_j})h_\ell(\E[\phi(-\xi)](\pref \ast F^{N_j})_{t_\ell}) \right]}.
	\end{equation*}
	Therefore, proceeding similarly as in \eqref{eq: FIRST EQUATION IN THE CHARACTERISATION OF LIMITI POINTS PROOF OF F},
	\begin{equation*}
		\Romannum{4} \le 2c(1 + \norm{\phi}_{\infty}) \limsup_{j} \E\left[\sum_{k > M} F_T^{N_j,k}\right].
	\end{equation*}
	By our choice of $M$ earlier, independently of $N_j$, we have $\E\left[\sum_{k > M} F_T^{N_j,k}\right] < \varepsilon (8c + 8c\norm{\phi}_{\infty})^{-1}$. Therefore, $\Romannum{4} < \varepsilon/4$. Hence, we have shown that $\Romannum{1} + \Romannum{2} + \Romannum{3} + \Romannum{4} < \varepsilon$. As $\varepsilon > 0$ was arbitrary, the proof is now complete.
\end{proof}

We now have all the ingredients to show that our limit point will satisfy \cref{ass: ASSUMPTIONS NEEDED TO PROVE UNIQUENESS WHICH PARALLEL THOSE FROM THE SPDE PAPER}.

\begin{proposition}[Regularity conditions]\label{prop: THE LIMITING PROCESS WILL SATISFY THE ASSUMPTIONS WE HAVE SET OUT}\label{prop: Limiting firing function is sub-Guassian}
	Let $(\nu^N,F^N,W^0)$ be a subsequence, still indexed by $N$, along which $(\nu^N,F^N,W^0) \implies (\nu^*,F^*,W^0)$. Then $(\nu^*,F^*)$ satisfies \cref{ass: ASSUMPTIONS NEEDED TO PROVE UNIQUENESS WHICH PARALLEL THOSE FROM THE SPDE PAPER}. Furthermore, $F_t^*$ is sub-Gaussian for every $t \in [0,T]$.
\end{proposition}

\begin{proof} 
	With the obvious adjustments, \citep[Proposition~5.3]{hamblyseanhalfline} confirms that $\nu_t^*$ is a sub-probability measure for all $t \in [0,T]$, with probability $1$, and each $\nu_t^*$ is supported on $(-\infty,0]$. Next, the uniform exponential decay of the tail probabilities, \cref{ass: ASSUMPTIONS NEEDED TO PROVE UNIQUENESS WHICH PARALLEL THOSE FROM THE SPDE PAPER} \eqref{ass: ASSUMPTIONS NEEDED TO PROVE UNIQUENESS WHICH PARALLEL THOSE FROM THE SPDE PAPER THREE}, follows directly from \supprop{5.2}. Given these properties, one readily deduces that, with probability 1, $t\mapsto \langle \nu^*_t, f \rangle$ is \cadlag for all $f\in \mathcal{C}_b(\mathbb{R})$, so $\nu^*$ belongs to $D_{{\mathbf{M}_{\le 1}}}$ with probability $1$ as required.
	
	Next, \cref{ass: ASSUMPTIONS NEEDED TO PROVE UNIQUENESS WHICH PARALLEL THOSE FROM THE SPDE PAPER} \eqref{ass: ASSUMPTIONS NEEDED TO PROVE UNIQUENESS WHICH PARALLEL THOSE FROM THE SPDE PAPER FOUR} and \eqref{ass: ASSUMPTIONS NEEDED TO PROVE UNIQUENESS WHICH PARALLEL THOSE FROM THE SPDE PAPER FIVE}, follows as in the proof of \citep[Proposition~5.6]{hamblyseanhalfline} with the obvious changes. This also gives us that $F^*$ is non-decreasing. 
	
	It remains to show \cref{ass: ASSUMPTIONS NEEDED TO PROVE UNIQUENESS WHICH PARALLEL THOSE FROM THE SPDE PAPER} \eqref{ass: ASSUMPTIONS NEEDED TO PROVE UNIQUENESS WHICH PARALLEL THOSE FROM THE SPDE PAPER TWO}. Let $\phi_n \in \mathcal{C}^{\infty}(\mathbb{R};[0,1])$ be such that
	\begin{equation*}
		\phi_n \begin{cases}
			=1 &\textrm{ on } [-n,-1/n],\\
			\in(0,1) &\textrm{ on } [-1 - n,-n) \cup (-1/n,0],\\
			=0 &\textrm{ otherwise.}
		\end{cases}
	\end{equation*}
	We set $\mathbb{T} = \{t \in [0,T]\,:\,\prob[\langle\nu_t^*,\phi_{n}\rangle = \langle\nu_{t-}^*,\phi_{n}\rangle] = \prob[F_t^* = F_{t-}^*] = 1 \, \forall\, n\}$. By construction, $\mathbb{T}$ is a co-countable set in $[0,T]$ and, therefore, dense. By the Continuous Mapping Theorem, for all $t \in \mathbb{T}$, we have
	\begin{equation*}
		\langle\nu^N,\phi_n\rangle_t + F_t^N - (\pref \ast F^N)_t \implies \langle\nu^*,\phi_n\rangle_t + F_t^* - (\pref \ast F^*)_t.
	\end{equation*}
	Therefore,
	\begin{align*}
		\E\left[\nnnorm{\langle\nu_t^*,\phi_n\rangle + F_t^* - (\pref \ast F^*)_t - 1}\right] &\le \limsup_{N \to \infty} \E\left[\nnnorm{\langle\nu_t^N,\phi_n\rangle + F_t^N - (\pref \ast F^N)_t - 1}\right]\\
		&\le \limsup_{N \to \infty} \E\left[\nnnorm{\langle\nu_t^N,\phi_n\rangle - \nu_t^N(-\infty,0)}\right]\\ &\qquad+ \limsup_{N \to \infty} \E\left[\nnnorm{(\pref \ast F^N)_t - F_t^{D,N}}\right]\\
		&\le 2\limsup_{N \to \infty} \E\left[\nu_t^N(-\infty,-n) + \nu_t^N(-1/n,0)\right]\\ &\qquad + \limsup_{N \to \infty} \E\left[\nnnorm{(\pref \ast F^N)_t - F_t^{D,N}}\right],
	\end{align*}
	where the first inequality follows from the Portmanteau Theorem, and the second follows from the fact that $\nu_t^N(-\infty,0) + F^N_t - F_t^{D,N} = 1$. By \cref{prop: LIMIT POINT REPRESENTATION OF THE DELAYED FIRING FUNCTION}, a simple application of the Portmanteau Theorem allows us to deduce that $\E[\nnorm{(\pref \ast F^N)_t - F_t^{D,N}}] \to 0 $. By \cref{cor: REGULARITY CONTROL ON THEM EMPIRICAL MEASURES}, there exist $\delta,\tilde{\delta},C > 0$ such that
	\begin{equation*}
		\limsup_{N \to \infty} \E\left[\nu_t^N(-\infty,-n)\right]\le Ce^{-\delta n^2} \quad \text{ and } \quad \limsup_{N \to \infty} \E\left[\nu_t^N(-1/n,0)\right] \le Cn^{-\tilde{\delta}}.
	\end{equation*}
	Therefore, sending $n \to \infty$ and employing the Dominated Convergence Theorem, we may conclude that
	\begin{equation*}
		\E\left[\nnnorm{\nu_t^*(-\infty,0) + F_t^* - (\pref \ast F^*)_t - 1}\right] = 0.
	\end{equation*}
	By choosing a countable, dense subset $\{t_n\}_n \subset \mathbb{T}$, we have
	\begin{equation*}
		\nu_{t_n}^*(-\infty,0) + F_{t_n}^* - (\pref \ast F^*)_{t_n} = 1 \quad \forall\; t_n \quad \text{a.s.}
	\end{equation*}
	As $\nu^*(-\infty,0),\, F^*$, and $(\pref \ast F^*)$ are all \cadlag, we have
	\begin{equation*}
		\nu_t^*(-\infty,0) + F_t^* - (\pref \ast F^*)_t = 1 \quad \forall\; t \ge 0\quad \text{a.s.}
	\end{equation*}
	
	For the second claim, uniformly in $N$ and $t \in [0,T]$, we have a $\delta > 0$ such that
	\begin{equation*}
		\prob\left[F_t^N > \lambda \right] = O(e^{-\delta\lambda^2}).
	\end{equation*}
	Since $F^N \implies F^*$ in $(D_\R,M_1)$ and $T$ is an almost sure continuity point of $F^*$, then by the Continuous Mapping Theorem, $F^N_T \implies F^*_T$ in $\R$. Therefore, by the Portmanteau Theorem,
	\begin{equation*}
		\prob\left[F_T^* > \lambda \right] \le \liminf_{N \to \infty} \prob\left[F_T^N > \lambda\right] \le Ce^{-\delta\lambda^2},
	\end{equation*}
	so $F_T^*$ is sub-Gaussian. As $t \mapsto F_t^*$ is non-decreasing, $F_t^*$ is sub-Gaussian uniformly in $t \in [0,T]$.
\end{proof}

\subsection{Convergence to the Limiting SPDE}\label{sec: SUBSECTION THAT SHOWS THE CONVERGENCE TO THE LIMITTING SPDE}

Beyond identifying the SPDE in the limit, we would like to know that $\nu^N \implies \nu^*$ on $(D_{\mathscr{S}^{\prime}}, M_1)$ also implies $\langle\nu^N, \text{Id}\rangle \eqqcolon M^N \implies M \coloneqq \langle\nu^*, \text{Id}\rangle$ on $(D_{\mathbb{R}}, M_1)$, where $\text{Id}$ is the identity map. Since $\text{Id} \notin \mathscr{S}$, this does not follow from the continuous mapping theorem. Hence, we present the following result, which serves two purposes:
\begin{enumerate}
	\item It allows us to deduce the convergence of the mean process $M^N$ to $M = \langle\nu^*, \text{Id}\rangle$.
	\item It provides a stepping stone for proving the convergence of certain integrals, which is essential in demonstrating that each limit $\nu^*$ satisfies the SPDE \eqref{eqn: THE LIMITTING SPDE EQUATION}.
\end{enumerate}

\begin{proposition}\label{prop: CONVERGENCE OF THE FFD ON A CO COUNTABLE SET OF TIMES NEEDED FOR CONVEREGCNE OF THE MEANS AND INTEGRAL CONVERGENCE IN THE NEURO SPDE PAPER}
	\sloppy
	Let $\psi(t,x,\mu,f)$ be a measurable function that satisfies \cref{ass: ASSUMPTIONS FROM SOJMARK SPDE PAPER APPLIED TO THE NEUROSCIENCE SETTING} \eqref{ass: ASSUMPTIONS FROM SOJMARK SPDE PAPER APPLIED TO THE NEUROSCIENCE SETTING ONE} and \eqref{ass: ASSUMPTIONS FROM SOJMARK SPDE PAPER APPLIED TO THE NEUROSCIENCE SETTING TWO}, and $(\nu^N,F^N,W^0)$ be a subsequence, still indexed by $N$, along which $(\nu^N,F^N,W^0) \implies (\nu^*,F^*,W^0)$. Then there is a co-countable set of times on which the finite-dimensional distributions of $\langle \nu_t^N, \psi(t,\cdot , \nu_t^N , \mathfrak{f}_t^N)\rangle$ converge weakly to those of $\langle \nu_t^*, \psi(t,\cdot , \nu_t^* , \mathfrak{f}_t^*) \rangle$.
\end{proposition}

\begin{proof}
	Our strategy is as follows. We shall construct a co-countable set $\mathbb{T}$ and show that if we fix a $t \in \mathbb{T}$, we have $\langle \nu_t^N, \psi(t,\cdot , \nu_t^N , \mathfrak{f}_t^N)\rangle \implies \langle \nu_t^*, \psi(t,\cdot , \nu_t^* , \mathfrak{f}_t^*) \rangle$. Then, by employing the techniques in \cref{prop: LIMIT POINT REPRESENTATION OF THE DELAYED FIRING FUNCTION}, notably the method used to control $\Romannum{2}$ in \eqref{eq: FIRST EQUATION IN THE CHARACTERISATION OF LIMITI POINTS PROOF OF F}, we may conclude that we have convergence of the finite-dimensional distributions. The proof relies on several technical results presented in the supplement \suppref, namely Corollary 2.4, Proposition 5.1 and Proposition 5.2.
	
	To begin, we fix a bounded $h \in \operatorname{Lip}(\R)$ and fix a $t \in \mathbb{T}$, where
	\begin{equation*}
		\mathbb{T} \coloneqq \left\{t \in [0,T]\,:\, \prob[\langle\nu_t^*,\phi^n\rangle = \langle\nu_{t-}^*,\phi^n\rangle] = 1 \, \text{for all } n \in \N , \text{ where } \phi^n \in \mathscr{S} \text{ dense}\right\}.
	\end{equation*}
	It is clear that $\mathbb{T}$ is co-countable. By \citep[Remark~4.6]{hambly2019spde}, any subsequence of $(\nu^N,F^N)$ has a further subsequence, also denoted by $N$ for simplicity, for which we may assume almost sure convergence on a common probability space. By the triangle inequality,
	\begin{equation}\label{eq: DECOMPOSITON OF THE DIFFERENCE IN THE CONVERGENCE OF THE INTEGRALS PROOF INTO THREE PARTS IN THE SPDE NEURO PAPER}
		\begin{aligned}
			&\nnnorm{\E\left[h(\langle \nu_t^N, \psi(t,\cdot , \nu_t^N , \mathfrak{f}_t^N)\rangle)\right] - \E\left[h(\langle \nu_t^*, \psi(t,\cdot , \nu_t^* , \mathfrak{f}_t^*) \rangle)\right]}\\
			&\hspace{5cm}\lesssim \E\nnnorm{ \langle\nu_t^N - \nu_t^*, \psi(t,\cdot,\nu_t^*,\mathfrak{f}_t^*)\rangle}\\
			&\hspace{5cm}\quad + \E\nnnorm{ \langle\nu_t^N, \psi(t,\cdot,\nu_t^*,\mathfrak{f}_t^N) - \psi(t,\cdot,\nu_t^*,\mathfrak{f}_t^*)\rangle}\\
			&\hspace{5cm}\quad + \E\nnnorm{ \langle\nu_t^N, \psi(t,\cdot,\nu_t^N,\mathfrak{f}_t^N) - \psi(t,\cdot,\nu_t^*,\mathfrak{f}_t^N)\rangle}\\
			&\hspace{5cm}= \E{I}_1^N +\E{I}_2^N + \E{I}_3^N.
		\end{aligned}
	\end{equation}
	The rest of the proof is split into three parts. We will show that each term converges to $0$ as $N \to \infty$.\\
	
	\noindent \underline{Control on ${I}_1^N$:}\\[1ex]
	For $\lambda > 0$, let $\chi_\lambda \in \mathcal{C}^\infty(\mathbb{R};[0,1])$ denote the standard cut-off function which is equal to $1$ on $[-\lambda,\lambda]$ and $\varphi_{1/\lambda} \in \mathcal{C}_c^\infty(\R)$ be the standard mollifier rescaled by $1/\lambda$. We consider the (random) functions, indexed by $\lambda > 0$,
	\begin{equation*}
		\psi_\lambda^*(x) = \chi_\lambda(x) \cdot (\psi(t,\cdot , \nu_t^* , \mathfrak{f}_t^*) \ast \varphi_{1/\lambda})(x).
	\end{equation*}
	
	By our assumptions on $\psi$, it is globally Lipschitz in $x$, and we have $\nnnorm{\psi(t,x,\mu,{f}_t)} \le C_{\psi}(1 + \nnnorm{x} + \langle\mu,\nnnorm{\cdot}\rangle + \nnnorm{f_t})$. Therefore, setting $C(\nu^*) = C_{\psi}(1 + \sup_{s \le T} \langle \nu_s^*, \nnnorm{\cdot}\rangle + F_t^*)$, we may show
	\begin{equation*}
		\nnorm{\psi_\lambda^* - \psi(t,\cdot , \nu_t^* , \mathfrak{f}_t^*)} \le \begin{cases}
			2 \norm{\psi}_{\operatorname{Lip}} \lambda^{-1}\qquad &\text{ on } [-\lambda,\lambda],\\
			2 \norm{\psi}_{\operatorname{Lip}} (\lambda^{-1} + \nnorm{x}) + C(\nu^*) \qquad &\text{ on } [-\lambda,\lambda]^\complement.
		\end{cases}
	\end{equation*}
	
	Therefore, for all $\lambda > 0$, by employing the upper bound above followed by \cref{cor: REGULARITY CONTROL ON THEM EMPIRICAL MEASURES} and \supprop{5.1}, we obtain
	\begin{align*}\label{eq: FUNCTIONAL LLN FOR THE MEAN PROPOSITION PROOF EQUATION ONE}
\E{\nnorm{\langle\nu_t^{N},\psi_\lambda^* - \psi(t,\cdot , \nu_t^* , \mathfrak{f}_t^*)\rangle}} &\le 2 \norm{\psi}_{\operatorname{Lip}} \E{\langle\nu_t^{N},\lambda^{-1}\1_\R + (\lambda^{-1} + \nnorm{x} + C(\nu^*))\1_{(-\infty,-\lambda)}\rangle} \\
		&\le C (\lambda^{-1} + \lambda^{-1} e^{-\delta\lambda^2} + e^{-a\lambda}) + \E[C(\nu^*)\nu_t^N(-\infty,-\lambda)],
	\end{align*}
	where $C,\,a,\,\delta > 0$ are independent of $N$ and $t$. We note $\E[C(\nu^*)^q]$ is finite by \cref{prop: Limiting firing function is sub-Guassian}, and \supprop{5.2} for any $q \ge 1$. Therefore, by employing H\"older's inequality and \cref{cor: REGULARITY CONTROL ON THEM EMPIRICAL MEASURES} to control the last term in the above, we may conclude that 
	\begin{equation}\label{eq: UPPERBOUND ON THE DIFFERENCE BETWEEN THE MOLLIFIED FUNCTION AND THE LIMIT FUNCTION IN THE CONVERGENCE OF THE MEANS PROOF}
		\E{\nnorm{\langle\nu_t^{N},\psi_\lambda^* - \psi(t,\cdot , \nu_t^* , \mathfrak{f}_t^*)\rangle}} = O(\lambda^{-1})
	\end{equation}
	uniformly in $N$. By employing similar arguments, we may conclude that the bound in \eqref{eq: UPPERBOUND ON THE DIFFERENCE BETWEEN THE MOLLIFIED FUNCTION AND THE LIMIT FUNCTION IN THE CONVERGENCE OF THE MEANS PROOF} holds when $\nu^N$ is replaced by $\nu^*$. Therefore, by employing the triangle inequality,
	\begin{equation*}
		\E I_1^N \le \E{\nnorm{\langle\nu_t^{N},\psi_\lambda^* - \psi(t,\cdot , \nu_t^* , \mathfrak{f}_t^*)\rangle}} + \E\nnnorm{ \langle\nu_t^N - \nu_t^*, \psi_\lambda^*\rangle} + \E{\nnorm{\langle\nu_t^{N},\psi_\lambda^* - \psi(t,\cdot , \nu_t^* , \mathfrak{f}_t^*)\rangle}}.
	\end{equation*}
	For any $\delta > 0$, we may fix a $\lambda > 0$ large enough such that the first term and the third term in the above is less than $\delta/3$. As the projection map $\Pi^{\psi_\lambda^*}: (D_{\mathscr{S}^\prime},M_1) \to (D_\R,M_1)$, where $\Pi^{\psi_\lambda^*}(f) = f(\psi_\lambda^*)$ is continuous (\citep[Proposition~2.7]{ledger2016skorokhod}) and $t \in \mathbb{T}$, we have $\langle\nu_t^N - \nu_t^*, \psi_\lambda^*\rangle$ converges almost surely to $0$ as $N \to \infty$. Hence, for this fixed $\lambda$, we may deduce that the second term is less than $\delta/3$ in expectation for all $N$ large by the Dominated Convergence Theorem. Therefore, by first choosing a $\lambda$ large enough, we have for all $N$ large that $\E I_1^N \le \delta$.\\
	
	\noindent \underline{Control on ${I}_2^N$:}\\[1ex]
	By assumption, $\psi$ is globally Lipschitz in its fourth component. This property, along with the fact that $\nu_t^N$ is a sub-probability measure, gives us
	\begin{equation*}
		\nnnorm{\langle\nu_t^N, \psi(t,\cdot,\nu_t^*,\mathfrak{f}_t^N) - \psi(t,\cdot,\nu_t^*,\mathfrak{f}_t^*)\rangle} \le C  \int_0^t \nnnorm{\mathfrak{K}^\prime(t-s)} \nnnorm{F_s^N - F_s^*} \diff s.
	\end{equation*}
	By \cref{prop: Limiting firing function is sub-Guassian}, and \supcor{2.4}, it is immediate that for any $p \ge 1$, there exists a constant $C_p > 0$ such that
	\begin{equation*}
		\sup_N \E\left[\left(\int_0^t \nnnorm{\mathfrak{K}^\prime(t-s)} \nnnorm{F_s^N - F_s^* }\diff s\right)^p\right] < C_p.
	\end{equation*}
	Furthermore, as $F^N \to F^*$ almost surely in $(\DR,M_1)$, by the properties of $M_1$-convergence, we have
	\begin{equation*}
		\int_0^t \nnnorm{\mathfrak{K}^\prime(t-s)} \nnnorm{F_s^N - F_s^*} \diff s \to 0
	\end{equation*}
	almost surely. Hence, by Vitali's Convergence Theorem, we obtain
	\begin{equation*}
		\E\int_0^t \nnnorm{\mathfrak{K}^\prime(t-s)} \nnnorm{F_s^N - F_s^*} \diff s \to 0.
	\end{equation*}
	This proves that $\E I_2^N \to 0$.\\

	\noindent \underline{Control on ${I}_3^N$:}\\[1ex]
	Similarly to the analysis of $I_2^N$, the global Lipschitzness of $\psi$ in the measure component combined with the fact that $\nu_t^N \in \MR$ gives
	\begin{equation}\label{eq: BOUND IN I3N IN THE CONVERGENCE OF INTEGRAL PROOF}
		I_3^N \le C \sup\{\nnnorm{\langle\nu_t^N - \nu_t^*,\phi\rangle}\,:\, \phi \in \mathcal{C}_{d_0}\},
	\end{equation}
	where $\mathcal{C}_{d_0} \coloneqq \{\phi \in \mathcal{C}(\R)\,:\, \norm{\phi}_{\operatorname{Lip}} \le 1,\, \nnnorm{\phi(x)} \le 1 + \nnnorm{x}\}$. Fix a $\delta > 0$, then for any $\lambda > 1$, by the Arzel\`a-Ascoli Theorem, there is a finite family of functions $\phi_1,\ldots,\phi_m \in \mathcal{C}_{d_0}$ supported on $[-\lambda - 1, \lambda + 1]$, for $m = m(\lambda) \in \N$ such that for any $\phi \in \mathcal{C}_{d_0} $,
	\begin{equation}\label{eq: THE SUP NORM DISTANCE BETWEEN OUR FINITE APPROXIMATING SET}
		\sup_{x \in [-\lambda,\lambda]} \nnnorm{\phi(x) - \phi_i(x)} < \delta/2
	\end{equation}
	for some $i \in \{1,\ldots,m\}$. Fixing any $\phi \in \mathcal{C}_{d_0}$ and choosing a $\phi_i$ such that \eqref{eq: THE SUP NORM DISTANCE BETWEEN OUR FINITE APPROXIMATING SET} holds, we have
	\begin{align*}
		\nnnorm{\langle\nu_t^N - \nu_t^*,\phi\rangle} 
		\le& \nnnorm{\int_{-\lambda}^0 (\phi - \phi_i) \diff(\nu_t^N - \nu_t^*)} 
		+\nnnorm{\int_{-\infty}^{-\lambda} (\phi - \phi_i) \diff(\nu_t^N - \nu_t^*)}\\ 
		&+ \nnnorm{\langle\nu_t^N - \nu_t^*,\phi_i\rangle}.
	\end{align*}
	By choice of $\phi_i$, the first term is bounded by $\delta$, uniformly in $N$. By the linear growth condition for functions in $\mathcal{C}_{d_0}$,
	\begin{equation*}
		\nnnorm{\int_{-\infty}^{-\lambda} (\phi - \phi_i) \diff(\nu_t^N - \nu_t^*)} \le C \langle\nu_t^* + \nu_t^N,\nnnorm{\cdot}\ind_{(-\infty,-\lambda]}\rangle.
	\end{equation*}
	Returning to \eqref{eq: BOUND IN I3N IN THE CONVERGENCE OF INTEGRAL PROOF}, we have shown
	\begin{equation*}
		\E I_3^N \le C \delta + C \E \langle\nu_t^* + \nu_t^N,\nnnorm{\cdot}\ind_{(-\infty,-\lambda]}\rangle + C \sum_{i = 1}^m\E\nnnorm{\langle\nu_t^N - \nu_t^*,\phi_i\rangle} . 
	\end{equation*}
	\supprop{5.1} gives that the middle term will vanish uniformly in $N$ as $\lambda \to \infty$. Hence, we fix a $\lambda$ sufficiently large so that the middle term is bounded by $\delta$, uniformly in $N$. As $\lambda$ is now fixed, so is $m$. Since the $\phi_i$'s have compact support, we may apply the same mollification argument as in the case of $I_1^N$ to deduce the final term will be smaller than $\delta$ for all $N$ large. Therefore, for all $N$ sufficiently large, $\E I_3^N < C\delta$ for some constant $C$ independent of $N$. Hence, $\E I_3^N \to 0$.
	
	This completes the proof.
\end{proof}

A particular case we may consider is $\psi(t,\cdot,\mu,f) = \phi \in \mathcal{C}^2(\R)$ with $\norm{\partial_x\phi}_{\mathcal{C}^1} < \infty$. Although $\phi$ itself may be unbounded as we only assume $\norm{\partial_x\phi}_{\mathcal{C}^1} < \infty$, the statements and proofs of \cref{prop: LEMMA ON THE CONTROL OF THE INCREMENTS OF THE PROJECTIONS}, \cref{prop: TIGHTNESS OF THE TUPLE OF THE INDIVIDUAL COMPONENTS OF THE DECOMPOSITION OF NU N APPLIED TO PHI}, and \cref{prop: TIGTHNESS OF THE TUPLE WITH THE MEASURES FIRIING FUNCTION AND COMMON BROWNIAN MOTION} remain valid for this $\phi$ if we replace $\norm{\phi}_{\infty}$ with $\sup_{x \in [-R,R]} \nnnorm{\phi(x)}$, since $\xi$ has compact support. As a consequence, $\{\langle \nu^N,\phi\rangle\}_N$ is tight and, by the proposition above, converges weakly in the $M_1$-topology towards $\langle\nu^*,\phi\rangle$. A special case arises when $\phi = \operatorname{Id}$, the identity function. In this situation, we obtain the convergence of the mean process.

\begin{lemma}\label{lem: CONVERGENCE OF THE INTEGRALS NEEDED TO APPLY THE RESULTS OF FABRICE ET AL IN THE NEURO SPDE PAPER}
	Let $\psi(t,x,\mu,f)$ be a measurable function that satisfies \cref{ass: ASSUMPTIONS FROM SOJMARK SPDE PAPER APPLIED TO THE NEUROSCIENCE SETTING} \eqref{ass: ASSUMPTIONS FROM SOJMARK SPDE PAPER APPLIED TO THE NEUROSCIENCE SETTING ONE} and $\phi \in \mathscr{S}$. Then for a subsequence $(\nu^N,F^N,W^0)$, still indexed by $N$, along which $(\nu^N,F^N,W^0) \implies (\nu^*,F^*,W^0)$, we have 
	\begin{equation*}
		\int_0^\cdot \langle\nu^N_s,\psi(s,\cdot,\nu^N_s,\mathfrak{f}^N_s)\phi \rangle \diff s \implies \int_0^\cdot \langle\nu^*_s, \psi(s,\cdot,\nu^*_s,\mathfrak{f}^*_s)\phi\rangle \diff s
	\end{equation*}
	on $\DR$ endowed with the $M_1$-topology.
\end{lemma}

\begin{proof}
	By \citep[Remark~4.6]{hambly2019spde}, any subsequence of $(\nu^N,F^N)$ has a further subsequence, also denoted by $N$ for simplicity, for which we may assume almost sure convergence on a common probability space. Fixing this subsequence and choosing a function $ h \in \operatorname{Lip}(\DR;\R)$, by the definition of the $M_1$-topology, we have
	\begin{equation}\label{eq: BOUND ON THE DIFFENCE BETWEEN THE EXPECTATIONS OF THE FUNCTIONAL INTEGRALS IN THE NEURO SPDE PAPER}
		\begin{aligned}
			&\nnnorm{\E\left[h\left(\int_0^\cdot \langle\nu^N_s,\psi(s,\cdot,\nu^N_s,\mathfrak{f}^N_s)\phi \rangle \diff s\right)\right] -  \E\left[h\left(\int_0^\cdot \langle\nu^*_s,\psi(s,\cdot,\nu^*_s,\mathfrak{f}^*_s)\phi \rangle \diff s\right)\right]}\\
			&\hspace{3cm}\le \int_0^T \E\left[\nnnorm{\langle\nu^N_t,\psi(t,\cdot,\nu^N_t,\mathfrak{f}^N_t)\phi\rangle - \langle\nu^*_t,\psi(t,\cdot,\nu^*_t,\mathfrak{f}^*_t)\phi \rangle}\right] \diff t.
		\end{aligned}
	\end{equation}
	
	At this point, the aim is to apply the Dominated Convergence Theorem. We will show that the expectation in the above converges to $0$ as $N \to \infty$ and is bounded uniformly in $N$. The rest of this proof is nearly verbatim to that of \cref{prop: CONVERGENCE OF THE FFD ON A CO COUNTABLE SET OF TIMES NEEDED FOR CONVEREGCNE OF THE MEANS AND INTEGRAL CONVERGENCE IN THE NEURO SPDE PAPER}. The only difference lies in that, as $\psi(t,x,\mu,f)\phi(x)$ need not be globally Lipschitz in $x$, we must adapt the mollification argument to show that the first term in the decomposition in \eqref{eq: DECOMPOSITON OF THE DIFFERENCE IN THE CONVERGENCE OF THE INTEGRALS PROOF INTO THREE PARTS IN THE SPDE NEURO PAPER} is zero. First, decomposing the difference in the expectation in \eqref{eq: BOUND ON THE DIFFENCE BETWEEN THE EXPECTATIONS OF THE FUNCTIONAL INTEGRALS IN THE NEURO SPDE PAPER} into three parts as in \eqref{eq: DECOMPOSITON OF THE DIFFERENCE IN THE CONVERGENCE OF THE INTEGRALS PROOF INTO THREE PARTS IN THE SPDE NEURO PAPER} and choosing a $t \in \mathbb{T}$ from \cref{prop: CONVERGENCE OF THE FFD ON A CO COUNTABLE SET OF TIMES NEEDED FOR CONVEREGCNE OF THE MEANS AND INTEGRAL CONVERGENCE IN THE NEURO SPDE PAPER}, we have
	
	\noindent \underline{New control on ${I}_1^N$:}\\[1ex]
	Fix a $\delta > 0$. We have $\nnnorm{\psi(t,x,\mu,{f}_t)} \le C(1 + \nnnorm{x} + \langle\mu,\nnnorm{\cdot}\rangle + \nnnorm{f_t})$. Therefore, setting $C(\nu^*) = 1 + \langle \nu_t^*, \nnnorm{\cdot}\rangle + F_T^*$, we observe there exists a $\lambda = \lambda(\delta) \gg 1$ such that
	\begin{equation*}
		\sup_{x \in [-\lambda,\lambda]^\complement} \nnorm{\psi(t,x,\nu_t^*,\mathfrak{f}_t^*)\phi(x)} \le C(\nu^*)\delta/4 \qquad \forall\, t \in [0,T].
	\end{equation*}
	Now, we choose the standard family of mollifiers $\varphi^\varepsilon \in \mathcal{C}_c^\infty(\R)$ and consider the (random) mollifications
	\begin{equation*}
		\psi_t^\varepsilon = \psi(t,\cdot,\nu_t^*,\mathfrak{f}_t^*) \ast \varphi^\varepsilon \qquad \textnormal{for} \qquad 0 < \varepsilon < 1.
	\end{equation*}
	As $\nu^*$ and $\nu^N$ are sub-probability measures,
	\begin{equation*}
		{I}_1^N \le  \nnnorm{\langle\nu_t^N - \nu_t^*,\psi_t^\varepsilon \phi\rangle} + 2 \norm{\phi}_{\infty} \sup_{\nnnorm{x} < 2\lambda} \nnnorm{\psi_t^\varepsilon - \psi_t} + C(\nu^*)\delta.
	\end{equation*}
	As $\psi_t^\varepsilon \to \psi(t,\cdot,\nu_t^*,\mathfrak{f}_t^*)\phi$ uniformly in $x \in [-2\lambda,2\lambda]$, by the Dominated Convergence Theorem we have for all $\varepsilon \ll 1$, the second term will be less than $\delta$ in expectation. Independent of $\varepsilon$, there is a set of full measure such that $\psi_t^\varepsilon\phi \in \mathscr{S}$. As $t \in \mathbb{T}$, we may deduce that the first term converges almost surely to $0$ as $N \to \infty$. Hence, it will be less than $\delta$ in expectation by the Dominated Convergence Theorem for all $N$ sufficiently large. Therefore, by first choosing an $\varepsilon$ small enough, we have for all $N$ large $\E I_1^N \le C(\delta + \delta + \delta)$, where $C$ depends only on $\E[C(\nu^*)]$. We note $\E[C(\nu^*)]$ is finite by \cref{prop: Limiting firing function is sub-Guassian}, and \supprop{5.2}.\\
	
	As $\psi(t,x,\mu,f)\phi(x)$ is still Lipschitz in $\mu$ and $f$, we still have $\E I_2^N$ and $\E I_3^N$ converge to $0$ for all $t \in \mathbb{T}$ by the results in \cref{prop: CONVERGENCE OF THE FFD ON A CO COUNTABLE SET OF TIMES NEEDED FOR CONVEREGCNE OF THE MEANS AND INTEGRAL CONVERGENCE IN THE NEURO SPDE PAPER}. 
	
	Lastly by the linear growth on $\psi$ and the fact that $\phi \in \mathscr{S}$, we observe
	\begin{align*}
		&\E\left[\nnnorm{\langle\nu^N_t,\psi(t,\cdot,\nu^N_t,\mathfrak{f}^N_t)\phi \rangle - \langle\nu^*_t,\psi(t,\cdot,\nu^*_t,\mathfrak{f}^*_t)\phi \rangle}\right] \\
		&\hspace{5cm}\le C \E\left[1 + \langle \nu_t^N, \nnnorm{\cdot}\rangle + F_T^N + \langle \nu_t^*, \nnnorm{\cdot}\rangle + F_T^*\right]
	\end{align*}
	where $C > 0$ is independent of $t$. The right-hand side in the above is bounded uniformly in $N$ by \cref{cor: REGULARITY CONTROL ON THEM EMPIRICAL MEASURES}, and \supprop{5.1}. Therefore, by the Dominated Convergence Theorem, we have the integral in \eqref{eq: BOUND ON THE DIFFENCE BETWEEN THE EXPECTATIONS OF THE FUNCTIONAL INTEGRALS IN THE NEURO SPDE PAPER} converges to $0$ as $N \to \infty$. This completes the proof.
\end{proof}

As $\sigma$ and $\rho$ are bounded by \cref{ass: ASSUMPTIONS FROM SOJMARK SPDE PAPER APPLIED TO THE NEUROSCIENCE SETTING} \eqref{ass: ASSUMPTIONS FROM SOJMARK SPDE PAPER APPLIED TO THE NEUROSCIENCE SETTING ONE} and \eqref{ass: ASSUMPTIONS FROM SOJMARK SPDE PAPER APPLIED TO THE NEUROSCIENCE SETTING THREE}, we may deduce that the convergence result from the previous proposition also holds in the case when the integrands are squared. This is formalized in the following corollary.

\begin{corollary}\label{cor: CONVERGENCE OF THE QUADRATIC VARIATION INTEGRALS NEEDED TO APPLY THE RESULTS OF FABRICE ET AL IN THE NEURO SPDE PAPER}
	Let $\phi \in \mathscr{S}$ and $(\nu^N,F^N,W^0)$ be a subsequence, still indexed by $N$, which converges weakly to $(\nu^*,F^*,W^0)$. Then, we have
	\begin{equation*}
		\int_0^\cdot \langle\nu^N_s, \sigma(s,\cdot)\rho(s,\nu^N_s,\mathfrak{f}^N_s)\phi \rangle^2 \diff s \implies \int_0^\cdot \langle\nu^*_s, \sigma(s,\cdot)\rho(s,\nu^*_s,\mathfrak{f}^*_s)\phi \rangle^2 \diff s 
	\end{equation*}
	on $\DR$ endowed with the $M_1$-topology.
\end{corollary}

We now have all the ingredients to prove \cref{thm: THE LIMITING SPDE THEOREM MAIN RESULT OF THE WORK}. 

\begin{proof}[Proof of \cref{thm: THE LIMITING SPDE THEOREM MAIN RESULT OF THE WORK}]
	Fix an arbitrary $\phi \in \mathscr{S}$. Returning to \eqref{eq: STOCHASTIC EVOLUTION EQUATION}, we first focus on the last term. By the Burkholder--Davis--Gundy inequality, we obtain
	\begin{align*}
		&\E\left[\sup_{t \le T}\nnnorm{\int_0^t\frac{1}{N}\sum_{i  =1}^N \sigma(s,X_s^i)\sqrt{1 - \rho^2(s,\nu_s^N,\mathfrak{f}_s^N)}\partial_x\phi(X_s^i)\ind_{\{X_s^i < 0\}}\diff{}W_s^i}\right]\\
		&\hspace{2cm}\le C^{\textnormal{BDG}} \E\left[\left(\frac{1}{N^2}\sum_{i = 1}^N\int_0^T\sigma(s,X_s^i)^2(1 - \rho^2(s,\nu_s^N,\mathfrak{f}_s^N))(\partial_x\phi(X_s^i))^2\diff{s}\right)^{1/2}\right]\\
		&\hspace{2cm} \le \frac{C^{\textnormal{BDG}}C_\sigma\norm{\partial_x\phi}_{\infty}N^{1/2}T^{1/2}}{N}.
	\end{align*}
	In particular, we may conclude that this term converges weakly towards $0$ in $\DR$ with the $M_1$-topology. By \cref{prop: TIGTHNESS OF THE TUPLE WITH THE MEASURES FIRIING FUNCTION AND COMMON BROWNIAN MOTION}, we may find a subsequence $\{(\nu^{N_j}, F^{N_j},W^0)\}$ which converges weakly toward $(\nu^*,F^*,W^0)$. By \cref{lem: CONVERGENCE OF THE INTEGRALS NEEDED TO APPLY THE RESULTS OF FABRICE ET AL IN THE NEURO SPDE PAPER}, individually, each of the $\diff t$ integrals in  \eqref{eq: STOCHASTIC EVOLUTION EQUATION} converge weakly along this subsequence towards their counterparts where $\nu^N$ is replaced by $\nu^*$ in $\DR$ endowed with the $M_1$-topology. Given \cref{lem: CONVERGENCE OF THE INTEGRALS NEEDED TO APPLY THE RESULTS OF FABRICE ET AL IN THE NEURO SPDE PAPER} and \cref{cor: CONVERGENCE OF THE QUADRATIC VARIATION INTEGRALS NEEDED TO APPLY THE RESULTS OF FABRICE ET AL IN THE NEURO SPDE PAPER}, we may apply \citep[Corollary~2.14]{sojmark2023weak} to conclude that the stochastic integral with respect to $W^0$ in \eqref{eq: STOCHASTIC EVOLUTION EQUATION} converges weakly towards its analogous counterpart. Furthermore, by \cref{prop: LIMIT POINT REPRESENTATION OF THE DELAYED FIRING FUNCTION}, we also have functional weak convergence of the firing function and the delayed reinsertion term. As all limiting terms are continuous apart from the $F^*$ term, we may conclude by the continuous mapping theorem that the entire right-hand side of \eqref{eq: STOCHASTIC EVOLUTION EQUATION} converges weakly in $\DR$ endowed with the $M_1$-topology towards the right-hand side of \eqref{eqn: THE LIMITTING SPDE EQUATION}. Similarly, by the continuous mapping theorem, the left-hand side of \eqref{eq: STOCHASTIC EVOLUTION EQUATION} converges weakly in $\DR$ endowed with the $M_1$-topology towards the left-hand side of \eqref{eqn: THE LIMITTING SPDE EQUATION}. This precisely means that $(\nu^*,F^*,W^0)$ is a solution to the SPDE \eqref{eqn: THE LIMITTING SPDE EQUATION}. Lastly, \cref{prop: THE LIMITING PROCESS WILL SATISFY THE ASSUMPTIONS WE HAVE SET OUT} gives us that $\nu^*$ is a \cadlag $\mathbf{M}_{\le 1}(\R_-)$-valued process with probability $1$ and $(\nu^*,F^*)$ satisfy \cref{ass: ASSUMPTIONS NEEDED TO PROVE UNIQUENESS WHICH PARALLEL THOSE FROM THE SPDE PAPER}. This completes the proof of \cref{thm: THE LIMITING SPDE THEOREM MAIN RESULT OF THE WORK}.
\end{proof}

Having established that any limit point is a solution to the SPDE \eqref{eqn: THE LIMITTING SPDE EQUATION}, satisfying the required regularity conditions of Assumption \ref{ass: ASSUMPTIONS NEEDED TO PROVE UNIQUENESS WHICH PARALLEL THOSE FROM THE SPDE PAPER}, it remains to confirm the desired relationship between the firing function and the flux of mass across the firing threshold.

\begin{proof}[Proof of Proposition \ref{lem: THE WEAK FLUX CONDITION}]
	Consider any subsequence $(\nu^N,F^N,W^0)$, still indexed by $N$, along which we have $(\nu^N,F^N,W^0) \Rightarrow (\nu^*,F^*,W^0)$. Let $\psi_{\varepsilon} \in \mathcal{C}^{\infty}(\mathbb{R}; [0,1])$ be a smooth approximation to the indicator function of $[-\varepsilon/2, \varepsilon/2]$, with $\psi_{\varepsilon}(x) = 1$ on $[-\varepsilon/2, \varepsilon/2]$ and $\psi_{\varepsilon}(x) = 0$ outside $[-\varepsilon, \varepsilon]$. Moreover, we can choose $\psi_{\varepsilon}$ such that $\nnnorm{\partial_x \psi_{\varepsilon}} \le C\varepsilon^{-1}$ for some constant $C > 0$ that is uniform in $\varepsilon$. Since $\psi_{\varepsilon} \in \mathscr{S}$, we can use it as a test function in the SPDE \eqref{eqn: THE LIMITTING SPDE EQUATION} to obtain the following equation:
	\begin{equation}\label{eq: SPDE EQUATION APPLED TO OUR APPROXIMATING FUNCTION IN THE PROOF OF THE WEAK FLUX CONDITION}
		\begin{split}                    
			F_t^* &=
			\int_0^t\langle\nu_u^*,\,b(u,\cdot,\nu_u^*,\mathfrak{f}_u^*)\partial_x\psi_\varepsilon\rangle\diff{u}
			+\frac{1}{2}\int_0^t\langle\nu_u^*,\sigma^2(u,\cdot)\partial_{xx}\psi_\varepsilon\rangle\diff{u}\\
			&+\int_0^t\langle\nu_u^*,\,\sigma(u,\cdot)\rho(u,\nu_u^*,\mathfrak{f}_u^*)\partial_x\psi_\varepsilon\rangle\diff{W_u^0} + \langle\nu_0,\,\psi_\varepsilon\rangle-\langle\nu_t^*,\,\psi_\varepsilon\rangle
		\end{split}
	\end{equation} 
	for all sufficiently small $\varepsilon$ as $\psi_{\varepsilon}(-\xi) = 0$. Since $\nu_0$ is deterministic and has an $L^2$ density, we have convergence in $L^1$ and almost surely of $\langle\nu_0, \, \psi_{\varepsilon}\rangle$ towards $0$ as $\varepsilon \to  0$. By employing \cref{prop: BOUNDARY DECAY OF THE EMPIRICAL MEAUSRES IN EXPECTATION} and the continuous mapping theorem, we have 
	\begin{equation}\label{eq: L1 BOUNDARY ESTIMATE ON THE LIMITING MEASURE IN THE WEAK FLUX CONDITION PROOF}
		\E[\nu_t^*(-\varepsilon,0)] = O(t^{-\delta/2}\varepsilon^{ 1 + \beta}) \qquad \text{ and } \qquad \E\int_0^t \nu_s^*(-\varepsilon,0) \diff s = O(\varepsilon^{ 1 + \beta}),
	\end{equation}
	for some constants $\delta,\, \beta > 0$ that are independent of $\varepsilon$ and $\delta \le 1$. By slightly augmenting the arguments in \cite[Lemma~A.4]{hambly2019spde}, we can employ a Borel-Cantelli argument to show that
	\begin{equation}\label{eq: ALMOST SURE BOUNDARY ESTIMATE ON THE LIMITING MEASURE IN THE WEAK FLUX CONDITION PROOF}
		\lim_{\varepsilon \downarrow 0} \varepsilon^{-p} \nu_t^*(-\varepsilon,0) = 0 \qquad \textnormal{ and } \qquad \lim_{\varepsilon \downarrow 0} \varepsilon^{-p} \int_0^t \nu_s^*(-\varepsilon,0) \diff s = 0
	\end{equation}
	almost surely for all $1 \le p \le 1 + \beta/2$, where $\beta$ is the constant from \eqref{eq: L1 BOUNDARY ESTIMATE ON THE LIMITING MEASURE IN THE WEAK FLUX CONDITION PROOF}. Therefore, by \eqref{eq: L1 BOUNDARY ESTIMATE ON THE LIMITING MEASURE IN THE WEAK FLUX CONDITION PROOF} and \eqref{eq: ALMOST SURE BOUNDARY ESTIMATE ON THE LIMITING MEASURE IN THE WEAK FLUX CONDITION PROOF}, $\langle\nu_t^*, \, \psi_{\varepsilon}\rangle$ converges to $0$ almost surely and in $L^1$ as $\varepsilon \downarrow 0$. Since the stochastic integral in \eqref{eq: SPDE EQUATION APPLED TO OUR APPROXIMATING FUNCTION IN THE PROOF OF THE WEAK FLUX CONDITION} is a martingale, it disappears when we take increments and conditional expectations. Therefore, we only need to control the terms on the first line of \eqref{eq: SPDE EQUATION APPLED TO OUR APPROXIMATING FUNCTION IN THE PROOF OF THE WEAK FLUX CONDITION}. For brevity, we show the almost sure convergence, but similar arguments apply to show convergence in $L^1$.
	
	By \cref{ass: ASSUMPTIONS FROM SOJMARK SPDE PAPER APPLIED TO THE NEUROSCIENCE SETTING}, we have $\nnnorm{b(u,x,\nu_u^*,\mathfrak{f}_u^*)} \le C_b(1 + \nnorm{x}+ \langle \nu_u^*,\nnorm{\cdot}\rangle + F_T^*)$. Therefore,
	\begin{equation*}
		\int_0^t \nnnorm{\langle\nu_u^*,\,b(u,\cdot,\nu_u^*,\mathfrak{f}_u^*)\partial_x\psi_\varepsilon\rangle}\diff{u} \le C \int_0^t (1 + \langle \nu_u^*,\nnorm{\cdot}\rangle + F_T^*) \varepsilon^{-1} \nu_u^*(-\varepsilon,0)\diff{u}.
	\end{equation*}
	By applying H\"older's inequality, we get
	\begin{align*}
		&\int_0^t \nnnorm{\langle\nu_u^*,\,b(u,\cdot,\nu_u^*,\mathfrak{f}_u^*)\partial_x\psi_\varepsilon\rangle}\diff{u}\\ 
		&\hspace{3cm}\le C \left(\int_0^t(1 + \langle \nu_u^*,\nnorm{\cdot}^q\rangle + (F_T^*)^q) \diff u\right)^{\frac{1}{q}} \left(\int_0^t \varepsilon^{-p}\nu_u^*(-\varepsilon,0) \diff u\right)^{\frac{1}{p}} 
	\end{align*}
	for all conjugate exponents $p$ and $q$. The first term above is finite in expectation for all $q \ge 1$ by \cref{prop: THE LIMITING PROCESS WILL SATISFY THE ASSUMPTIONS WE HAVE SET OUT}, hence is finite almost surely. The second term in the above converges to $0$ almost surely by \eqref{eq: ALMOST SURE BOUNDARY ESTIMATE ON THE LIMITING MEASURE IN THE WEAK FLUX CONDITION PROOF} for all $p > 1$ close enough to $1$. Therefore, by fixing an appropriate $p$, we have almost sure convergence to $0$ of the term with $b$ in \eqref{eq: SPDE EQUATION APPLED TO OUR APPROXIMATING FUNCTION IN THE PROOF OF THE WEAK FLUX CONDITION}.
	
	Lastly, we note that we can write
	\begin{equation*}
		\langle\nu_u^*,\sigma^2(u,\cdot)\partial_{xx}\psi_\varepsilon\rangle = \langle\nu_u^*,\partial_x[\sigma^2(u,\cdot)\partial_{x}\psi_\varepsilon]\rangle - \langle\nu_u^*,\partial_x[\sigma^2(u,\cdot)]\partial_{x}\psi_\varepsilon\rangle,
	\end{equation*}
	and the second term on the right-hand side converges to $0$ almost surely as $\varepsilon \to 0$ by the same arguments as above. Therefore, we have shown
	\begin{equation*}
		\E\left[\left.F_t^*- F_s^*\right| \mathcal{F}_s\right]  - \frac{1}{2} \int_s^t \langle\nu_u^*,\partial_x[\sigma^2(u,\cdot)\partial_{x}\psi_\varepsilon]\rangle \diff u = o(1) \qquad \text{as} \quad \varepsilon \to 0 \quad \text{almost surely.}
	\end{equation*}
	This completes the proof.
\end{proof}


\section{Uniqueness of the SPDE}\label{sec: Uniqueness}\label{sec: UNIQUENESS SECTION IN THE GENERALISED INTEGRATE AND FIRE NEURON MODELS}
This section is devoted to the proofs of \cref{thm: UNIQUENESS PROP} and \cref{thm: CONDITIONAL LLN STATING THAT WE CONVERGE TO A UNIQUE LIMIT POINT THAT DEPENDS ON THE COMMON NOISE}. Our approach builds upon the techniques developed in \citep{hamblyseanhalfline} and \citep{hambly2019spde}. For the parts of our estimates that are similar, we refer to those works, but we must give careful consideration to controlling the additional terms that arise from SPDE \eqref{eqn: THE LIMITTING SPDE EQUATION} compared to the problems in \citep{hamblyseanhalfline, hambly2019spde}.

We begin by showing that solutions to SPDE \eqref{eqn: THE LIMITTING SPDE EQUATION} have an $L^2$ density. Exploiting this, we then look to obtain an energy estimate in $H^{-1}$, the dual of $H_0^1 = \mathcal{W}^{1,2}(\R_-)$. This is a natural space to work in, as the space of finite signed measures (with the topology induced by the finite variation norm) is embedded into $H^{-1}$, we have control over solutions in $L^2$, and the nonlinearities take the form of integrals against the solution. For any two solutions, $\nu$ and $\tilde{\nu}$, to the SPDE \eqref{eqn: THE LIMITTING SPDE EQUATION} which satisfy \cref{ass: ASSUMPTIONS NEEDED TO PROVE UNIQUENESS WHICH PARALLEL THOSE FROM THE SPDE PAPER}, we derive an energy estimate in $H^{-1}$ for their difference $\Delta_t = \nu_t - \tilde{\nu}_t$. We do this by first considering a convenient mollification of $\Delta_t$, that is,
\begin{equation*}
	(\Teps\Delta_t)(x) \coloneqq \int_\R G_\varepsilon(x,y) \diff \Delta_t(y), \quad \text{ where } \quad G_\varepsilon(x,y) = p_\varepsilon(x-y) - p_{\varepsilon}(x+y)
\end{equation*}
with $p_\varepsilon(x) \coloneqq (2\pi\varepsilon)^{-1/2} e^{-\frac{x^2}{2\varepsilon}}$. Then, by \citep[Proposition~6.5]{hamblyseanhalfline}, we have
\begin{equation}\label{eq: BOUND ON THE H -1 NORM}
	\norm{\Delta_t}_{H^{-1}} \le \liminf_{\varepsilon \to  0} \norm{\partial_x^{-1}\mathcal{T}_{\varepsilon}\Delta_t}_2, \qquad \text{ where } \qquad (\partial_x^{-1}\mathcal{T}_{\varepsilon}\Delta_t)(x) \coloneqq \int_{-\infty}^x (\mathcal{T}_{\varepsilon}\Delta_t)(y) \diff y.
\end{equation}
For simplicity of notation, we use $\norm{\cdot}_2$ to denote either the $L^2(-\infty,0)$ or $L^2(\R)$ norm. Generally, the space will be clear from context.

\subsection{Existence of an \texorpdfstring{$L^2$}{ L2} Density Uniformly in Time}
\begin{proposition}[Uniform {$L^2$} energy estimate]\label{prop:L2 ENERGY ESTIMATE}
	Let $(\nu,F,W^0)$ be any process that solves the SPDE \eqref{eqn: THE LIMITTING SPDE EQUATION} and satisfies \cref{ass: ASSUMPTIONS NEEDED TO PROVE UNIQUENESS WHICH PARALLEL THOSE FROM THE SPDE PAPER}. Then $ (\Teps \nu_t)(x) :=\langle \nu_t , G_\varepsilon(x,\cdot ) \rangle $ satisfies
	\begin{equation*}
		\sup_{t \in [0,T]} \sup_{\varepsilon > 0} \norm{\mathcal{T}_\varepsilon\nu_t}_2^2 < \infty \qquad \text{ with probability } 1.
	\end{equation*}
\end{proposition}

\begin{proof} 
	Fix $x \in \R$ and take $\phi(y) := p_\varepsilon(x - y)$ as a test function in \eqref{eqn: THE LIMITTING SPDE EQUATION} to obtain an equation for $(\tBar\nu_t)(x) \coloneqq \langle \nu_t , p_\varepsilon(x - \cdot) \rangle$. Proceeding similarly to  \citep[Proposition~5.3]{hambly2019spde} and \cite[Proposition~7.1]{hamblyseanhalfline}, It{\^o}'s formula and some manipulations yield
	\begin{align}\label{eq:P EPSILON APPLIED TO THE SPDE EQUATION AND DOING TAYLOR EXPANSION AROUND OF THE DIFFUSION COEFFICIENTS AND ALL SQUARED}
		\begin{split}
			\diff (\tBar \nu_t)^2=
			&-2 (\tBar \nu_t)(b_t \partial_x (\tBar \nu_t) - \partial_xb_t\Bar{\mathcal{H}}_{t,\varepsilon}^b + \Bar{\mathcal{E}}_{t,\varepsilon}^b)\diff{t}\\
			&+ (\tBar \nu_t) \partial_x(\sigma_t^2\partial_x (\tBar \nu_t)  - \partial_x \sigma_t^2\Bar{\mathcal{H}}_{t,\varepsilon}^{\sigma^2} + \Bar{\mathcal{E}}_{t,\varepsilon}^{\sigma^2})\diff{t}\\
			&-2\rho_t(\tBar \nu_t)(\sigma_t\partial_x (\tBar \nu_t)  - \partial_x \sigma_t\Bar{\mathcal{H}}_{t,\varepsilon}^{\sigma} + \Bar{\mathcal{E}}_{t,\varepsilon}^{\sigma})\diff{W_t^0} \\
			&+\rho_t^2(\sigma_t\partial_x (\tBar \nu_t)  - \partial_x \sigma_t\Bar{\mathcal{H}}_{t,\varepsilon}^{\sigma} + \Bar{\mathcal{E}}_{t,\varepsilon}^{\sigma})^2\diff{t} \\
			&+2(\tBar \nu_t) \langle \mures, p_\varepsilon(x - y)\rangle \mathfrak{r}_t\diff t
			-2(\tBar \nu_t) p_\varepsilon\diff{}F_t.
		\end{split}
	\end{align}
	where the error terms $\Bar{\mathcal{E}}$ and $\Bar{\mathcal{H}}$ (for brevity, the dependence on $x$ is omitted) are as defined in \citep[Lemma~A.2]{hambly2019spde}. Unlike the aforementioned works, the last line of \eqref{eq:P EPSILON APPLIED TO THE SPDE EQUATION AND DOING TAYLOR EXPANSION AROUND OF THE DIFFUSION COEFFICIENTS AND ALL SQUARED} has two additional finite variation terms due to the resetting of the neurons.
	
	Noting that the last term is negative, we can simply discard it. Turning instead to the first term in the last line of \eqref{eq:P EPSILON APPLIED TO THE SPDE EQUATION AND DOING TAYLOR EXPANSION AROUND OF THE DIFFUSION COEFFICIENTS AND ALL SQUARED}, by integrating over $\R$, we need to control an expression of the form
	\begin{equation}
		\int_\R \int_0^t (\tBar \nu_s)(x) p_\varepsilon(x+z)\int_0^s \pref^{\prime}(s - u) F_u \diff u \diff s \diff x
	\end{equation}
	uniformly in $z \in [\gamma, R]$. As $\tBar \nu_s = (p_\varepsilon \ast \nu_s)$, we first observe that
	\begin{equation}\label{eq: FIRST EQUATION IN THE PROCESS TO BOUND THE FINITE VARIATION TERM IN THE SPDE}
		\int_\R (\tBar \nu_s)(x) p_\varepsilon(x+z) \diff x = (p_\varepsilon \ast ( p_\varepsilon \ast \nu_s))(-z) = \Bar{\mathcal{T}}_{2\varepsilon} \nu_s(-z),
	\end{equation}
	which follows from the associativity of the convolution operator and the fact that $p_\varepsilon \ast p_\varepsilon = p_{2\varepsilon}$. By the Fundamental Theorem of Calculus,
	\begin{equation*}
		\nnorm{\Bar{\mathcal{T}}_{2\varepsilon}\nu_s(-z)} = \nnnorm{-\int_{-z}^{\varepsilon^{1/4}} \partial_x \Bar{\mathcal{T}}_{2\varepsilon}\nu_s(x) \diff x + \Bar{\mathcal{T}}_{2\varepsilon}\nu_s(\varepsilon^{1/4})} \le C \norm{\partial_x \Bar{\mathcal{T}}_{2\varepsilon}\nu_s}_2 + \nnorm{\Bar{\mathcal{T}}_{2\varepsilon}\nu_s(\varepsilon^{1/4})},
	\end{equation*}
	where the last inequality follows from an application of H\"older's inequality. The constant $C$ holds uniformly in $ \varepsilon < 1$ and $z \in [\gamma, R]$. By the definition of $p_\varepsilon$, $\nnorm{\Bar{\mathcal{T}}_{2\varepsilon}\nu_s(\varepsilon^{1/4})} \le (4\pi\varepsilon)^{-1/2}e^{-\frac{1}{4\varepsilon^{1/2}}}$. Furthermore, as $\partial_x \Bar{\mathcal{T}}_{2\varepsilon}\nu_s = p_\varepsilon \ast \partial_x \Bar{\mathcal{T}}_{\varepsilon}\nu_s$, convolution with $p_\varepsilon$ contracts the $L^2$-norm, \citep[Proposition~6.1]{hamblyseanhalfline}, and $ \partial_x \Bar{\mathcal{T}}_{\varepsilon}\nu_s$ is in $L^2(\R)$, we have $\norm{\partial_x \Bar{\mathcal{T}}_{2\varepsilon}\nu_s}_2 \le \norm{\partial_x \Bar{\mathcal{T}}_{\varepsilon}\nu_s}_2$. Therefore, we have shown that there is a constant $C$, which holds uniformly in $ \varepsilon < 1$ and $z \in [\gamma, R]$, such that
	\begin{equation}\label{eq: SECOND EQUATION IN THE PROCESS TO BOUND THE FINITE VARIATION TERM IN THE SPDE}
		\nnorm{\Bar{\mathcal{T}}_{2\varepsilon}\nu_s(-z)} \le C \left(\norm{\partial_x \Bar{\mathcal{T}}_{\varepsilon}\nu_s}_2 + \varepsilon^{-\frac{1}{2}}e^{-\frac{1}{4\varepsilon^{1/2}}}\right).
	\end{equation}
	Returning to \eqref{eq:P EPSILON APPLIED TO THE SPDE EQUATION AND DOING TAYLOR EXPANSION AROUND OF THE DIFFUSION COEFFICIENTS AND ALL SQUARED}, employing \eqref{eq: FIRST EQUATION IN THE PROCESS TO BOUND THE FINITE VARIATION TERM IN THE SPDE} and the upper bound in \eqref{eq: SECOND EQUATION IN THE PROCESS TO BOUND THE FINITE VARIATION TERM IN THE SPDE}, we obtain
	\begin{equation}\label{eq:BOUND ON THE FINITE VARIATION TERM IN THE SPDE}
		\begin{aligned}
			&\nnnorm{\int_\R \int_0^t \int_0^s (\tBar \nu_s)(x) p_\varepsilon(x+z) \pref^{\prime}(s - u) F_u \diff u \diff s \diff x}\\
			&\hspace{1cm}\le C \int_0^t \left(\norm{\partial_x \Bar{\mathcal{T}}_{\varepsilon}\nu_s}_2 + \varepsilon^{-\frac{1}{2}}e^{-\frac{1}{4\varepsilon^{1/2}}}\right) F_s \norm{\pref^{\prime}}_1 \diff s \\
			&\hspace{1cm}\le \theta \int_0^t\norm{\partial_x \Bar{\mathcal{T}}_{\varepsilon}\nu_s}_2^2 \diff s + C_\theta F_t^2 \norm{\pref^{\prime}}_1^2 + \varepsilon^{-\frac{1}{2}}e^{-\frac{1}{4\varepsilon^{1/2}}} F_t \norm{\pref^{\prime}}_1, 
		\end{aligned}
	\end{equation}
	where the last line follows from applying Young's inequality with free parameter $\theta$. 
	
	As in \citep[Proposition~3.5]{hambly2019spde}, integration by parts takes care of the first line of \eqref{eq:P EPSILON APPLIED TO THE SPDE EQUATION AND DOING TAYLOR EXPANSION AROUND OF THE DIFFUSION COEFFICIENTS AND ALL SQUARED}, since the tails of $\Bar{\mathcal{T}}_\varepsilon\nu_t$ vanish at $\pm\infty$ by \cref{ass: ASSUMPTIONS NEEDED TO PROVE UNIQUENESS WHICH PARALLEL THOSE FROM THE SPDE PAPER} \eqref{ass: ASSUMPTIONS NEEDED TO PROVE UNIQUENESS WHICH PARALLEL THOSE FROM THE SPDE PAPER THREE}, so
	\begin{equation}
		\notag -\int_\R \int_0^t 2b_s(\tBar \nu_s)\partial_x (\tBar \nu_s) \diff s \diff x 
		= \int_0^t \int_\R \partial_x b_s (\tBar \nu_s)^2 \diff s \diff x \le C_b \int_0^t \int_\R \norm{\tBar \nu_s}_2^2 \diff s. \label{eq:UPPER BOUND ON THE DRIFT TERM WITH B IN THE SAME FAITH AS THE SPDE PAPER}
	\end{equation}
	Thus, integrating over $x$ in \eqref{eq:P EPSILON APPLIED TO THE SPDE EQUATION AND DOING TAYLOR EXPANSION AROUND OF THE DIFFUSION COEFFICIENTS AND ALL SQUARED}, employing the above bound and \eqref{eq:BOUND ON THE FINITE VARIATION TERM IN THE SPDE}, using the fact that $\nnorm{\Bar{\mathcal{H}}_{t,\varepsilon}^g} \lesssim \Bar{\mathcal{T}}_{2\varepsilon}\nu_t$ by \citep[Lemma~8.2]{hamblyseanhalfline}, and employing Young's inequality with free parameter $\theta$, we obtain
	\begin{equation}\label{eq: UPPER BOUND ON THE SPDE TESTED AGAIANST THE GUASSIAN KERNEL SQUARED AND INTEGRATED AGAINST X}
		\begin{split}
			\begin{aligned}
				\norm{\tBar\nu_t}_2^2 \le& \norm{\tBar\nu_0}_2^2 + C_{\theta}\int_0^t \norm{\tBar\nu_s}_2^2 \diff s + C_{\theta}\int_0^t \norm{\Bar{\mathcal{T}}_{2\varepsilon}\nu_s}_2^2 \diff s \\
				&+ C_{\theta} \int_0^t \left(\norm{\Bar{\mathcal{E}}_{s,\varepsilon}^b}_2^2 + \norm{\Bar{\mathcal{E}}_{s,\varepsilon}^{\sigma^2}}_2^2 + \norm{\Bar{\mathcal{E}}_{s,\varepsilon}^{\sigma}}_2^2\right) \diff s \\
				&+\int_0^t \int_{\R} (\rho_s^2\sigma_s^2 + \theta\rho_s^2\sigma_s^2 - \sigma_s^2 + \theta)\nnorm{\partial_x(\tBar\nu_s)}^2 \diff x \diff s \\
				&+\int_0^t \int_{\R} \rho_s\partial_x \sigma_s (\tBar\nu_s)^2 + 2\rho_s(\tBar\nu_s)(\partial_x \sigma_s\Bar{\mathcal{H}}_{s,\varepsilon}^{\sigma} - \Bar{\mathcal{E}}_{s,\varepsilon}^{\sigma}) \diff x\diff W_s^0 \\
				&+C_\theta F_t^2 \norm{\pref^{\prime}}_1^2 + \varepsilon^{-\frac{1}{2}}e^{-\frac{1}{4\varepsilon^{1/2}}} F_t \norm{\pref^{\prime}}_1.
			\end{aligned}
		\end{split}
	\end{equation}
	Here we have used the Stochastic Fubini Theorem (using the tail decay of $\nu$) to swap the order of integration for the stochastic integral, and we have integrated by parts in the $\diff x$-integral. As $\rho$ is bounded away from $1$ and $\sigma$ is bounded away from $0$, we may find a $\theta$ sufficiently small such that the third term in \eqref{eq: UPPER BOUND ON THE SPDE TESTED AGAIANST THE GUASSIAN KERNEL SQUARED AND INTEGRATED AGAINST X} is negative and hence can be discarded. Now, first taking the supremum over $r \le t$ and then expectations in \eqref{eq: UPPER BOUND ON THE SPDE TESTED AGAIANST THE GUASSIAN KERNEL SQUARED AND INTEGRATED AGAINST X}, we get
	\begin{equation}\label{eq: UPPER BOUND ON THE SPDE TESTED AGAIANST THE GUASSIAN KERNEL SQUARED AND INTEGRATED AGAINST X AFTER TAKING THE SUPREMUM AND TAKING EXPECTATIONS}
		\begin{split}
			\begin{aligned}
				\E \sup_{r \le t} \norm{\tBar\nu_r}_2^2 \le& C_0\norm{\tBar\nu_0}_2^2 + tC_0  \E \sup_{r \le t} \norm{\tBar\nu_r}_2^2  + tC_0  \E \sup_{r \le t} \norm{\Bar{\mathcal{T}}_{2\varepsilon}\nu_r}_2^2\\
				&+ C_0 \E\int_0^t \norm{\Bar{\mathcal{E}}_{s,\varepsilon}^b}_2^2 + \norm{\Bar{\mathcal{E}}_{s,\varepsilon}^{\sigma^2}}_2^2 + \norm{\Bar{\mathcal{E}}_{s,\varepsilon}^{\sigma}}_2^2 \diff s\\
				&+C_0 \E[F_t^2] +  C_0 \varepsilon^{-\frac{1}{2}}e^{-\frac{1}{4\varepsilon^{1/2}}} \E[F_t],
			\end{aligned}
		\end{split}
	\end{equation}
	where we have used \citep[Lemma~8.5]{hamblyseanhalfline} for the stochastic integral. In turn,
	\begin{align}\label{eq: UPPER BOUND ON THE SPDE TESTED AGAIANST THE GUASSIAN KERNEL SQUARED AND INTEGRATED AGAINST X AFTER TAKING THE SUPREMUM AND TAKING EXPECTATIONS AND THEN TAKING THE LIMINF}
		f(t) \coloneqq \liminf_{\varepsilon \to  0}\E \sup_{r \le t} \norm{\tBar\nu_r}_2^2 &\le C_0\norm{V_0}_2^2 + 2t C_0  \liminf_{\varepsilon \to  0} \E \sup_{r \le t} \norm{\tBar\nu_r}_2^2  + C_0 \E[F_t^2]\notag \\
		&= C_0 \left(\norm{V_0}_2^2  + \E[F_t^2]\right) + 2tC_0f(t), 
	\end{align}
	where the inequality follows from applying \citep[Lemma~A.2]{hambly2019spde} and \citep[Proposition~6.1]{hamblyseanhalfline}, as $V_0 \in L^2(\R)$ by \cref{ass: ASSUMPTIONS FROM SOJMARK SPDE PAPER APPLIED TO THE NEUROSCIENCE SETTING} \eqref{ass: ASSUMPTIONS FROM SOJMARK SPDE PAPER APPLIED TO THE NEUROSCIENCE SETTING FOUR}. For $t \le T_0 \coloneqq 1/4C_0$, we thus have a bound on $f(t)$. By iteratively repeating the argument for $t \in [T_k,T_{k+1} \wedge T]$, where $T_k = k / 4 C_0$ and $k = 1,\ldots, \lceil 4C_0T\rceil$ analogously to \citep[Proposition~7.1]{hamblyseanhalfline}, we conclude that $\liminf_{\varepsilon \to  0}\E \sup_{r \le T} \norm{\tBar\nu_r}_2^2 < \infty$. Since $\Bar{\mathcal{T}}_{\tilde{\varepsilon}}\nu_t \ge \Bar{\mathcal{T}}_{{\varepsilon}}\nu_t$ for any $\tilde{\varepsilon} < \varepsilon$, and noting that $0 \le \mathcal{T}_{\varepsilon}\nu_t \le \tBar\nu_t$, this completes the proof.
\end{proof} 

Hence, we may conclude that any process that solves the SPDE \eqref{eqn: THE LIMITTING SPDE EQUATION} and satisfies \cref{ass: ASSUMPTIONS NEEDED TO PROVE UNIQUENESS WHICH PARALLEL THOSE FROM THE SPDE PAPER} has a density.

\begin{proposition}[$L^2$ density]\label{prop:L2_density}
	Let $(\nu,F,W^0)$ be any process that solves the SPDE \eqref{eqn: THE LIMITTING SPDE EQUATION} and satisfies \cref{ass: ASSUMPTIONS NEEDED TO PROVE UNIQUENESS WHICH PARALLEL THOSE FROM THE SPDE PAPER}. Then, with probability $1$, for every $t \in [0,T]$, there exists $V_t \in L^2(\R)$ such that $V_t$ is supported on $(-\infty, 0]$ and is a density of $\nu_t$. Furthermore, $\sup_{t \in [0,T]} \norm{V_t}_2 < \infty$ with probability $1$.
\end{proposition}

\begin{proof}
	The result follows directly from \cref{prop:L2 ENERGY ESTIMATE} combined with \citep[Proposition~6.2]{hamblyseanhalfline}.
\end{proof}

Before proceeding, we introduce the following notation to simplify the computations to follow.

\begin{remark}[Notation]\label{rem: OSQ 1 NOTAITON}
	We will use the notation $\osq$ to denote any family of $L^2(-\infty,0)$-valued processes $\{(f_{t,\varepsilon})_{t\in [0,T]}\}_{\varepsilon > 0}$ satisfying
	\begin{equation*}
		\E\int_0^T \norm{f_{t,\varepsilon}}_{L^2(-\infty,0)} \diff t \to  0 \qquad \text{ as } \qquad \varepsilon \to  0.
	\end{equation*}
\end{remark}

The following lemma is crucial for handling the boundary effects in the smoothed estimate of $\partial_x^{-1}\Teps\Delta_T$. It ensures that these terms vanish as $\varepsilon \to 0$, which is essential to deduce uniqueness. The result employs  the behaviour of the solutions' mass near the boundary, \cref{ass: ASSUMPTIONS NEEDED TO PROVE UNIQUENESS WHICH PARALLEL THOSE FROM THE SPDE PAPER} \eqref{ass: ASSUMPTIONS NEEDED TO PROVE UNIQUENESS WHICH PARALLEL THOSE FROM THE SPDE PAPER THREE}.

\begin{lemma}[Boundary estimate]\label{lem: BOUNDARY ESTIMATE IN THE UNIQUENESS PROOFS}
	Let $\mu$ be any measure satisfying \cref{ass: ASSUMPTIONS NEEDED TO PROVE UNIQUENESS WHICH PARALLEL THOSE FROM THE SPDE PAPER} and let $g_t(y)$ be a (stochastic) function with $\nnorm{g_t(y)} \lesssim 1 + \nnorm{y} + M_t + F_t^\mu$, where $M_t = \langle\mu_t,\psi\rangle$ for some $\psi \in \mathcal{C}^2(\R)$ with $\norm{\partial_x \psi}_{\mathcal{C}^1} < +\infty$. Then
	\begin{equation*}
		\E \int_0^T\int_{-\infty}^0 \nnorm{\langle\mu_t,g_t(\cdot)p_{\varepsilon}(x + \cdot)\rangle}^2 \diff x \diff t \to  0 \qquad \text{ as } \qquad \varepsilon \to  0.
	\end{equation*}
\end{lemma}

\begin{proof}
	By \cref{ass: ASSUMPTIONS NEEDED TO PROVE UNIQUENESS WHICH PARALLEL THOSE FROM THE SPDE PAPER} \eqref{ass: ASSUMPTIONS NEEDED TO PROVE UNIQUENESS WHICH PARALLEL THOSE FROM THE SPDE PAPER TWO} and employing Gr\"onwall's inequality, there is a constant $C > 0$ such that $F_T^\mu < C$. Therefore, $\nnorm{g_t(y)} \lesssim 1 + \nnorm{y} + M_t$. Now, as $\mu$ satisfies the boundary regularity of \citep[Assumption~2.3]{hambly2019spde}, the proof follows directly from \citep[Lemma~5.5]{hambly2019spde}.
\end{proof}

\subsection{The Smoothed \texorpdfstring{$H^{-1}$}{ H-1} Estimate}
Equation \eqref{eq: REWRITING THE FIRING FUNCTION ONLY IN TERMS OF THE MEAUSRE} shows that the firing function may be reformulated exclusively in terms of the measure flow. As a result, we may compare two firing functions and rewrite it as a function of the distance between the corresponding measure flows using a suitable metric. We observe by \cref{ass: ASSUMPTIONS NEEDED TO PROVE UNIQUENESS WHICH PARALLEL THOSE FROM THE SPDE PAPER} \eqref{ass: ASSUMPTIONS NEEDED TO PROVE UNIQUENESS WHICH PARALLEL THOSE FROM THE SPDE PAPER TWO} and employing Jensen's inequality,
\begin{equation*}
	\nnnorm{F_t  -\tilde{F}_t}^2 \lesssim \nnnorm{\Delta_t(-\infty,0)}^2 + t\int_0^t \nnnorm{F_s  -\tilde{F}_s}^2 \diff s
\end{equation*}
for any $t > 0$. Consequently, a straightforward application of Gr\"onwall's inequality gives us
\begin{equation}\label{eq: SECOND EQUATION ON SHOWING THAT WE CAN BOUND THE FIRING FUNCTION IN TERMS OF THE MEASURE}
	\nnnorm{F_t  -\tilde{F}_t}^2 \lesssim \nnnorm{\Delta_t(-\infty,0)}^2 + t\int_0^t \nnnorm{\Delta_s(-\infty,0)}^2e^{t - s} \diff s.
\end{equation}
Given that $\nnnorm{\Delta_t(-\infty,0)} \le d_1(\nu_t,\tilde{\nu}_t)$ by definition of the $d_1$-metric, \eqref{eq: SECOND EQUATION ON SHOWING THAT WE CAN BOUND THE FIRING FUNCTION IN TERMS OF THE MEASURE} can be simplified to
\begin{equation}\label{eq: THIRD EQUATION ON SHOWING THAT WE CAN BOUND THE FIRING FUNCTION IN TERMS OF THE MEASURE}
	\nnnorm{F_t  -\tilde{F}_t}^2 \lesssim d_1(\nu_t,\tilde{\nu}_t)^2 + \int_0^t d_1(\nu_s,\tilde{\nu}_s)^2 \diff s.
\end{equation}
Finally, integrating \eqref{eq: THIRD EQUATION ON SHOWING THAT WE CAN BOUND THE FIRING FUNCTION IN TERMS OF THE MEASURE} over the interval $[0,T]$, we obtain
\begin{equation}\label{eq: RELATIONSHIP BETWEEN THE FIRING FUNCTIONS}
	\int_0^T \nnorm{F_s - \tilde{F}_s}^2 \diff s \le C_T \int_0^T d_1(\nu_s,\tilde{\nu}_s)^2 \diff s.
\end{equation}
This result effectively bounds the difference between firing functions in terms of the $d_1$-metric distance between their corresponding measure flows and is employed in the proof of the following proposition.

\begin{proposition}[$H^{-1}$ estimate for the mollifications]
	As $\varepsilon \downarrow 0$, we have
	\begin{align}\label{eq: SMOOTHED H-1 ESTIMATE}
		\E\norm{\ad\Teps\Delta_{t\wedge t_n}}_2^2 + &c_0 \E \int_0^{t\wedge t_n} \norm{\Teps\Delta_s}_2^2 \diff s \le
		c_n \E \int_0^{t\wedge t_n} d_0(\nu_s,\tilde{\nu}_s)\norm{\ad\Teps\Delta_s}_2 \diff s \notag \\
		&+c_n\E \int_0^{t\wedge t_n} d_1(\nu_s,\tilde{\nu}_s)^2\diff s + o(1)
	\end{align}
	for a fixed $c_0 > 0$ and with $c_n$ depending only on $n$, and $(t_n)_n$ is a sequence of stopping times such that $t_n \uparrow T$ as $n \to  \infty$.
\end{proposition}
\begin{proof}
	
	For any fixed $x < 0$, we may use $y \mapsto G_{\varepsilon}(x,y)$ as a test function in \eqref{eqn: THE LIMITTING SPDE EQUATION} to obtain an equation for
	$(\Teps \nu_t)(x) =\langle \nu_t , G_\varepsilon(x,\cdot ) \rangle $.  Using that
	\begin{equation*}
		\partial_y G_{\varepsilon}(x,y) = - \partial_x G_{\varepsilon}(x,y) - 2 \partial_x p_\varepsilon (x+y)
		\quad \text{ and } \quad
		\partial_{yy} G_{\varepsilon}(x,y) = \partial_{xx} G_{\varepsilon}(x,y),
	\end{equation*}
	and integrating over the domain $(-\infty, x)$ to introduce the antiderivative (as well as invoking \citep[Lemma~8.3]{hamblyseanhalfline} to swap the order of the space and time integrals), we get
	\begin{align}\label{eqn: THE LIMITTING SPDE EQUATION APPLIED TO THE DIRICHLET HEAT KERNEL AND MAKING THE DERIVATIVES BE FUNCTIONS OF X WITH THE ANTIDERIVATIVE}
		\begin{split}
			\diff{}\ad\Teps\nu_t =&
			- \langle\nu_t,b_t G_{\varepsilon}(x,\cdot)\rangle\diff{t}
			+\frac{1}{2}\partial_{x}\langle\nu_t,\sigma_t^2 G_{\varepsilon}(x,\cdot)\rangle\diff{t}\\
			&-\rho_t \langle\nu_t,\sigma_t G_{\varepsilon}(x,\cdot)\rangle\diff{W_t^0}
			+ \langle \mures, \ad G_{\varepsilon}(x,\cdot)\rangle \mathfrak{r}_t\diff t\\
			&-2 \langle\nu_t,b_t p_{\varepsilon}(x+\cdot)\rangle\diff{t}
			-2\rho_t\langle\nu_t,\sigma_t p_{\varepsilon}(x+\cdot)\rangle\diff{W_t^0}.
		\end{split}
	\end{align}
	Employing \citep[Lemma~A.1]{hambly2019spde}, this becomes
	\begin{align}\label{eqn: THE LIMITTING SPDE EQUATION APPLIED TO THE DIRICHLET HEAT KERNEL AND MAKING THE DERIVATIVES BE FUNCTIONS OF X WITH THE ANTIDERIVATIVE AND THE OSQ1 TERMS}
		\begin{split}
			\diff{}\ad\Teps\nu_t =&
			- b_t \Teps\nu_t\diff{t}
			+\frac{1}{2}\partial_{x}\left(\sigma_t^2 \Teps\nu_t + \mathcal{E}_{t,\varepsilon}^{\sigma^2}\right)\diff{t}
			-\rho_t \sigma_t \Teps\nu_t \diff{W_t^0}\\
			&+ \langle \mures, \ad G_{\varepsilon}(x,\cdot)\rangle \mathfrak{r}_t\diff t
			+\osq\diff{t}
			+\osq\diff{W_t^0}, 
		\end{split}
	\end{align}
	where we have used the $\osq$ notation defined in \cref{rem: OSQ 1 NOTAITON}, and $\mathcal{E}_{t,\varepsilon}^g$ are the error terms from applying this transformation. Note that the $\mathcal{E}_{t,\varepsilon}^b$ and $\mathcal{E}_{t,\varepsilon}^{\sigma}$ terms are $\osq$ and are captured in the last line. Write $\Delta \coloneqq \nu - \tilde{\nu}$, where $\tilde{\nu}$ is any other solution to \eqref{eqn: THE LIMITTING SPDE EQUATION} satisfying \cref{ass: ASSUMPTIONS NEEDED TO PROVE UNIQUENESS WHICH PARALLEL THOSE FROM THE SPDE PAPER}, and define $\tilde{\mathcal{E}}_{t,\varepsilon}^b$ analogously to ${\mathcal{E}}_{t,\varepsilon}^b$ but with $\tilde{\nu}$ instead. Taking the difference in \eqref{eqn: THE LIMITTING SPDE EQUATION APPLIED TO THE DIRICHLET HEAT KERNEL AND MAKING THE DERIVATIVES BE FUNCTIONS OF X WITH THE ANTIDERIVATIVE AND THE OSQ1 TERMS} and applying It\^o's formula to $(\ad\Teps\Delta_t)^2$, we arrive at
	\begin{align}\label{eqn: APPLYING ITOS FORMULA TO THE LIMITTING SPDE EQUATION APPLIED TO THE DIRICHLET HEAT KERNEL AND MAKING THE DERIVATIVES BE FUNCTIONS OF X WITH THE ANTIDERIVATIVE AND THE OSQ1 TERMS AND THE DIFFERENCE OF TWO SOLUTIONS}
		\begin{split}
			\diff{}(\ad\Teps\Delta_t)^2 =&
			-2(\ad\Teps\Delta_t) (\tilde{b}_t \Teps\Delta_t + \delta_t^b \Teps \nu_t)\diff{t}\\
			&+(\ad\Teps\Delta_t)\partial_{x}(\sigma_t^2 \Teps\Delta_t + \mathcal{E}_{t,\varepsilon}^{\sigma^2}  - \mathcal{E}_{t,\varepsilon}^{\tilde{\sigma}^2})\diff{t}\\
			&+ 2 (\ad\Teps\Delta_t)\langle \mures, \ad G_{\varepsilon}(x,\cdot)\rangle (\mathfrak{r}_t - \tilde{\mathfrak{r}}_t)\diff t\\
			&+\sigma_t^2 (\tilde{\rho}_t\Teps\Delta_t + \delta_t^\rho \Teps\nu_t )^2\diff t + \sigma_t(\tilde{\rho}_t\Teps\Delta_t + \delta_t^\rho \Teps\nu_t ) \cdot \osq \diff t\\
			&+ (\ad\Teps\Delta_t) \cdot \osq \diff t+ \osq^2\diff{t}\\ 
			&- 2 \sigma_t (\ad\Teps\Delta_t) (\tilde{\rho}_t\Teps\Delta_t + \delta_t^\rho \Teps\nu_t ) + (\ad\Teps\Delta_t) \cdot \osq \diff{W_t^0},
		\end{split}
	\end{align}
	where $\tilde{g}_t = g(t,x,\tilde{\nu}_{t},\tilde{\mathfrak{f}}_t)$, $\tilde{\mathfrak{r}}_t = \int_0^t \pref (t -s) \mathrm{d} \tilde{F}_s $ and we set $\delta_t^g \coloneqq g(t,x,{\nu}_{t},\mathfrak{f}_t) - g(t,x,\tilde{\nu}_{t},\tilde{\mathfrak{f}}_t)$. Integrating over $\R_-$ we recover the $L^2$-norm of $\ad\Teps\Delta_t$, so it remains to control all the terms on the right-hand side of \eqref{eqn: APPLYING ITOS FORMULA TO THE LIMITTING SPDE EQUATION APPLIED TO THE DIRICHLET HEAT KERNEL AND MAKING THE DERIVATIVES BE FUNCTIONS OF X WITH THE ANTIDERIVATIVE AND THE OSQ1 TERMS AND THE DIFFERENCE OF TWO SOLUTIONS}. We proceed in six steps. To simplify the expressions, we shall use the notation
	\begin{align*}
		\nnorm{M}_{t,*} &\coloneqq \sup_{s \le t} \{1 + \nnorm{M_s} + \nnorm{\tilde{M}_s}\},
		\quad \nnorm{F}_{t,*} \coloneqq \sup_{s \le t} \{1 + \nnorm{F_s} + \nnorm{\tilde{F}_s}\},
		\quad \textnormal{and} \\ 
		\norm{\mathcal{T}\nu}_{t,*} &\coloneqq \sup_{s \le t} \sup_{\varepsilon > 0} \norm{\Teps\nu_s}_2.
	\end{align*}

	We note that the structure of the following estimates parallels that of \citep[Proposition~5.3]{hambly2019spde}
	and \citep[Proposition~7.1]{hamblyseanhalfline}, but there are several new terms to deal with.
	
	\medskip
	
	\noindent \underline{Step 1.}
	We control the first two lines of
	\eqref{eqn: APPLYING ITOS FORMULA TO THE LIMITTING SPDE EQUATION APPLIED TO THE DIRICHLET HEAT KERNEL AND MAKING THE DERIVATIVES BE FUNCTIONS OF X WITH THE ANTIDERIVATIVE AND THE OSQ1 TERMS AND THE DIFFERENCE OF TWO SOLUTIONS}
	together.
	Using $2(\ad\Teps\Delta_t)(\Teps\Delta_t) = \partial_x(\ad\Teps\Delta_t)^2$
	and the boundary conditions
	$\tilde{\mathcal{E}}_{t,\varepsilon}^g(0) = \mathcal{E}_{t,\varepsilon}^g(0) = \Teps\Delta_t(0) = 0$
	(which follow from $G_\varepsilon(x,0) = G_\varepsilon(0,y) = 0$),
	we integrate by parts and use the linear growth condition on $b$,
	the bound $\nnorm{\delta_t^b} \le C_b(d_0(\nu_t,\tilde{\nu}_t) + \nnorm{\mathfrak{f}_t - \tilde{\mathfrak{f}}_t})$,
	and Young's inequality with parameter $\theta$ to obtain
	\begin{align*}
		-2\int_{-\infty}^0 &(\ad\Teps\Delta_t)\tilde{b}_t\Teps\Delta_t \diff x
		- 2\int_{-\infty}^0 \delta_t^b (\ad\Teps\Delta_t)\Teps\nu_t \diff x \\
		&+ \int_{-\infty}^0 (\ad\Teps\Delta_t)\partial_x\!\bigl(\sigma_t^2\Teps\Delta_t
		+ \mathcal{E}_{t,\varepsilon}^{\sigma^2} - \mathcal{E}_{t,\varepsilon}^{\tilde{\sigma}^2}\bigr)\diff x \\
		\le\;& -\norm{\sigma_t\Teps\Delta_t}_2^2 + 3\theta\norm{\Teps\Delta_t}_2^2
		+ C_\theta(1 + \nnorm{M}_{t,*}^2 + \nnorm{F}_{t,*}^2)\norm{\ad\Teps\Delta_t}_2^2 \\
		&+ C_b\norm{\mathcal{T}\nu}_{t,*}\, d_0(\nu_t,\tilde{\nu}_t)\norm{\ad\Teps\Delta_t}_2
		+ C_\theta\norm{\mathcal{T}\nu}_{t,*}^2\nnorm{\mathfrak{f}_t - \tilde{\mathfrak{f}}_t}^2 \\
		&+ C_\theta\norm{\mathcal{E}_{t,\varepsilon}^{\sigma^2}}_2^2
		+ C_\theta\norm{\mathcal{E}_{t,\varepsilon}^{\tilde{\sigma}^2}}_2^2.
	\end{align*}
	Note that the term involving $\nnorm{\mathfrak{f}_t - \tilde{\mathfrak{f}}_t}^2$ satisfies
	\begin{equation*}
		\int_0^t \norm{\mathcal{T}\nu}_{s,*}^2 \nnorm{\mathfrak{f}_s - \tilde{\mathfrak{f}}_s}^2 \diff s
		\le \norm{\mathcal{T}\nu}_{t,*}^2 \norm{\mathfrak{K}^{\prime}}_1^2 \int_0^t \nnorm{F_s - \tilde{F}_s}^2 \diff s.
	\end{equation*}

	\noindent \underline{Step 2.} To gain control over the third term in \eqref{eqn: APPLYING ITOS FORMULA TO THE LIMITTING SPDE EQUATION APPLIED TO THE DIRICHLET HEAT KERNEL AND MAKING THE DERIVATIVES BE FUNCTIONS OF X WITH THE ANTIDERIVATIVE AND THE OSQ1 TERMS AND THE DIFFERENCE OF TWO SOLUTIONS}, we first seek to obtain control on the $L^1(\R_-)$-norm of $\ad G_\varepsilon(x,-z)$ uniformly in $\varepsilon < 1$ and $z \in [\gamma,R]$. First, we observe that for any $z > 0$,
	\begin{equation*}
		\int_{-\infty}^0 \ad G_\varepsilon(x,-z) \diff x= \int_{-\infty}^0\int_{-\infty}^x G_{\varepsilon}(y,-z)\diff y \diff x \le \int_{-\infty}^0\int_{-\infty}^x p_{\varepsilon}(y + z) \diff y \diff x,
	\end{equation*}
	where the inequality follows from the definition of $G_{\varepsilon}$. The inner integral above represents the probability that a normal random variable with mean $-z$ and variance $\varepsilon$ is smaller than $x$. By employing the well-known result $\prob[\mathcal{N}(0,1) > x] \le \frac{1}{2}e^{-\frac{x^2}{2}}$ for any $x > 0$, we know that if $x < -R$, then $\prob[\mathcal{N}(-z,\varepsilon) < x] \le 2^{-1}e^{-\frac{(x+z)^2}{2\varepsilon}}$ for any $z \in [\gamma,R]$. Therefore,
	\begin{equation*}
		\int_{-\infty}^0 \ad G_\varepsilon(x,-z) \diff x \le R + \int_{-\infty}^{-R} 2^{-1}e^{-\frac{(x+z)^2}{2\varepsilon}} \diff x < R + \pi
	\end{equation*}
	for any $\varepsilon < 1$ and $z \in [\gamma,R]$. Hence, $\norm{\langle \mures, \ad G_{\varepsilon}(x,\cdot)\rangle}_{L^1(\R_-)} < R + \pi$. Now, as $\langle \mures, \ad G_{\varepsilon}(x,\cdot)\rangle$ is non-negative and bounded by 1, by first applying the Cauchy-Schwarz inequality to the space integral and then Young's inequality with parameter $1/2$ to the $\diff s$-integral, we have
	\begin{align*}
		&2  \nnnorm{\int_0^t\int_{-\infty}^0 (\ad\Teps\Delta_s)\langle \mures, \ad G_{\varepsilon}(x,\cdot)\rangle \int_0^s \pref^{\prime}(s -u) (F_u - \tilde{F}_u)\diff u \diff x \diff s}\\
		&\hspace{3cm}\lesssim \int_0^t \norm{\ad\Teps\Delta_s}_2\norm{\langle \mures, \ad G_{\varepsilon}(x,\cdot)\rangle}_{L^1(\R_-)}\\
		&\hspace{3.5cm}\times\nnnorm{\int_0^s \pref^{\prime}(s -u) (F_u - \tilde{F}_u)\diff u} \diff s \\
		&\hspace{3cm}\lesssim \int_0^t \norm{\ad\Teps\Delta_s}_2^2 \diff s + \int_0^t \nnnorm{\int_0^s \pref^{\prime}(s -u) (F_u - \tilde{F}_u)\diff u}^2 \diff s \\
		&\hspace{3cm}\lesssim \int_0^t \norm{\ad\Teps\Delta_s}_2^2 \diff s + \norm{\pref^{\prime}}_1^2\int_0^t \nnnorm{F_s - \tilde{F}_s}^2 \diff s.
	\end{align*}
	
	\noindent \underline{Step 3.} To gain control over the fourth term in \eqref{eqn: APPLYING ITOS FORMULA TO THE LIMITTING SPDE EQUATION APPLIED TO THE DIRICHLET HEAT KERNEL AND MAKING THE DERIVATIVES BE FUNCTIONS OF X WITH THE ANTIDERIVATIVE AND THE OSQ1 TERMS AND THE DIFFERENCE OF TWO SOLUTIONS}, we observe that by expanding the brackets and applying Young's inequality with parameter $\theta$ to control the cross terms, we have
	\begin{align*}
		\int_{-\infty}^0\sigma_t^2 (\tilde{\rho}_t\Teps\Delta_t + \delta_t^\rho \Teps\nu_t )^2 \diff x  
		\le& C_\theta \norm{\mathcal{T}\nu}_{t,*}^2(d_1(\nu_t,\tilde{\nu}_t)^2 + \nnorm{\mathfrak{f}_t - \tilde{\mathfrak{f}}_t}^2)  \\
		&+ \norm{\sigma_t\tilde{\rho}_t\Teps\Delta_t}_2^2 + \theta\norm{\Teps\Delta_t}_2^2,\\
		\int_{-\infty}^0 \nnorm{\sigma_t(\tilde{\rho}_t\Teps\Delta_t + \delta_t^\rho \Teps\nu_t ) \cdot \osq }\diff x 
		\le& C_\theta \norm{\mathcal{T}\nu}_{t,*}^2(d_1(\nu_t,\tilde{\nu}_t)^2 + \nnorm{\mathfrak{f}_t - \tilde{\mathfrak{f}}_t}^2) \\
		&+ \theta\norm{\Teps\Delta_t}_2^2 + \norm{\osq}_2^2 .
	\end{align*}
	We used the fact that $\nnorm{\delta_t^\rho}^2 \lesssim d_1(\nu_t,\tilde{\nu}_t)^2 + \nnorm{\mathfrak{f}_t - \tilde{\mathfrak{f}}_t}^2$ in the above.
	
	\noindent \underline{Step 4.} As the stochastic integral in \eqref{eqn: APPLYING ITOS FORMULA TO THE LIMITTING SPDE EQUATION APPLIED TO THE DIRICHLET HEAT KERNEL AND MAKING THE DERIVATIVES BE FUNCTIONS OF X WITH THE ANTIDERIVATIVE AND THE OSQ1 TERMS AND THE DIFFERENCE OF TWO SOLUTIONS} is a true martingale, it vanishes when taking expectations. Therefore, by first taking expectations, then integrating over $x < 0$, and then applying Fubini's Theorem, it follows from Steps 1-4 and Young's inequality with parameter $\theta$ that
	\begin{align}
		\E\norm{\ad\Teps\Delta_t}_2^2 \le&
		C_\theta \E \int_0^t (\nnorm{M}_{s,*}^2 + \nnorm{F}_{s,*}^2)\norm{\ad\Teps\Delta_s}_2^2 \diff s\nonumber\\
		&+ C_\theta \E \int_0^t \norm{\mathcal{T}\nu}_{s,*}^2 \nnorm{\mathfrak{f}_s - \tilde{\mathfrak{f}}_s}^2 \diff s + C \norm{\pref^{\prime}}_1^2 \E \int_0^t \nnnorm{F_s - \tilde{F}_s}^2 \diff s \nonumber\\
		&+ C_b \E \int_0^t \norm{\mathcal{T}\nu}_{s,*}d_0(\nu_s,\tilde{\nu}_s)\norm{\ad\Teps\Delta_s}_2 \diff s\nonumber\\
		&+ C_\theta \E \int_0^t \norm{\mathcal{T}\nu}_{s,*}^2 d_1(\nu_s,\tilde{\nu}_s)^2 \diff s \nonumber\\
		&+ \E \int_0^t\int_{-\infty}^0 \{\sigma_s^2\tilde{\rho}_s^2 - \sigma_s^2 + \theta\} \nnorm{\Teps\Delta_s}^2 \diff x \diff s + o(1) \label{eq: FIRST UPPER BOUND ON THE EXPECTED VALUE OF THE ANTIDERIVATIVE L2 NORM SQUARED}
	\end{align}
	as $\varepsilon \downarrow 0$ for a constant $C_\theta$ which depends only on the parameter $\theta$. We also remark that, as the stochastic integral in \eqref{eqn: APPLYING ITOS FORMULA TO THE LIMITTING SPDE EQUATION APPLIED TO THE DIRICHLET HEAT KERNEL AND MAKING THE DERIVATIVES BE FUNCTIONS OF X WITH THE ANTIDERIVATIVE AND THE OSQ1 TERMS AND THE DIFFERENCE OF TWO SOLUTIONS} is a true martingale, \eqref{eq: FIRST UPPER BOUND ON THE EXPECTED VALUE OF THE ANTIDERIVATIVE L2 NORM SQUARED} holds for any bounded stopping time by Steps 1-4 and the Optional Stopping Theorem.
	
	\noindent \underline{Step 5.} As $\rho$ is bounded away from 1 and $\sigma$ is bounded away from $0$, we may take the free parameter $\theta$ to be small enough such that for every $s$ and $x$,
	\begin{equation*}
		\sigma(s,x)^2\rho(s,\tilde{\nu},\tilde{\mathfrak{f}}_s)^2 - \sigma(s,x)^2 + \theta \le -c_0
	\end{equation*}
	for a fixed constant $c_0 > 0$. Now, we consider the stopping times
	\begin{equation*}
		t_n \coloneqq \inf\{t > 0\,:\, \nnorm{M}_{t,*}^2 > n,\, \nnorm{F}_{t,*}^2 > n, \, \textnormal{or} \, \norm{\mathcal{T}\nu}_{t,*} ^2 > n\} \wedge T
	\end{equation*}
	for $n \ge 1$. By \cref{prop:L2 ENERGY ESTIMATE}, $t_n \uparrow T$ almost surely. We observe that at the time $t\wedge t_n$,
	\begin{align*}
		C_\theta \E \int_0^{t\wedge t_n} \norm{\mathcal{T}\nu}_{s,*}^2 \nnorm{\mathfrak{f}_s - \tilde{\mathfrak{f}}_s}^2 \diff s &\le C_\theta n \E \int_0^{t\wedge t_n} \nnnorm{\int_0^s \mathfrak{K}^{\prime}(s - u)(F_u - \tilde{{F}}_u) \diff u}^2 \diff s\\
		&\le  C_\theta n \norm{\mathfrak{K}^{\prime}}_1^2 \E \int_0^{t\wedge t_n} \nnorm{F_s - \tilde{F}_s}^2 \diff s,
	\end{align*}
	where the last inequality follows from an application of Fubini's Theorem and Jensen's inequality. Hence, evaluating \eqref{eq: FIRST UPPER BOUND ON THE EXPECTED VALUE OF THE ANTIDERIVATIVE L2 NORM SQUARED} at $t\wedge t_n$, we obtain
	\begin{align}\label{eq: FIRST UPPER BOUND ON THE EXPECTED VALUE OF THE ANTIDERIVATIVE L2 NORM SQUARED AT THE STOPPING TIME}
		\begin{split}
			&\E\norm{\ad\Teps\Delta_{t\wedge t_n}}_2^2 + c_0 \E \int_0^{t\wedge t_n} \norm{\Teps\Delta_s}_2^2 \diff s \\
			&\le Cn\E \int_0^{t\wedge t_n} \norm{\ad\Teps\Delta_s}_2^2 \diff s +Cn\E \int_0^{t\wedge t_n} \nnnorm{F_s - \tilde{F}_s}^2 +  d_1(\nu_s,\tilde{\nu}_s)^2\diff s \\
			&\quad+ Cn \E \int_0^{t\wedge t_n} d_0(\nu_s,\tilde{\nu}_s)\norm{\ad\Teps\Delta_s}_2 \diff s + o(1).
		\end{split}
	\end{align}
	Then \eqref{eq: SMOOTHED H-1 ESTIMATE} follows from employing \eqref{eq: RELATIONSHIP BETWEEN THE FIRING FUNCTIONS} and a simple application of Gr\"onwall's inequality, where $c_n \coloneqq Cne^{CnT}$ (for a possibly larger C than that in \eqref{eq: FIRST UPPER BOUND ON THE EXPECTED VALUE OF THE ANTIDERIVATIVE L2 NORM SQUARED AT THE STOPPING TIME}).
\end{proof}

\subsection{Uniqueness of the SPDE -- Proof of \texorpdfstring{\cref{thm: UNIQUENESS PROP}}{ Uniqueness}}
At this point, we have all the ingredients to prove \cref{thm: UNIQUENESS PROP}. We have already shown that every solution to the SPDE \eqref{eqn: THE LIMITTING SPDE EQUATION} will have a density in $L^2$. By a simple modification of \citep[Lemma~5.2]{hambly2019spde}, as our process lives on $(-\infty,0)$, we have a universal constant $c >0$ such that for any $s \le T$, $\delta \in (0,1)$, and $\lambda > 0$,
\begin{align*}
	d_0(\nu_s,\tilde{\nu}_s) &\le c\lambda(1 + \delta^{-1})\norm{\partial_x^{-1} \Teps\Delta_s}_2 + c\delta^{\frac{1}{2}}\norm{\frac{\diff \Delta_s}{\diff x}}_2 + f_s(\lambda) + g_s(\varepsilon),\\
	d_1(\nu_s,\tilde{\nu}_s) &\le c\sqrt{\lambda + \delta^{-1}}\norm{\partial_x^{-1} \Teps\Delta_s}_2 + c\delta^{\frac{1}{2}}\norm{\frac{\diff \Delta_s}{\diff x}}_2 + f_s(\lambda) + g_s(\varepsilon),
\end{align*}
where $(f_s(\lambda))_{s \le T}$ is a process such that for all $a > 0$, $\E\int_0^T f_s(\lambda)^2 \diff s \le c_ae^{-a\lambda}$ for some $c_a > 0$, and $(g_s(\varepsilon))_{s \le T}$ is a process such that $\E\int_0^T g_s(\varepsilon)\diff s = o(1)$ as $\varepsilon \downarrow 0$. Substituting the bounds for $d_0(\nu_s,\tilde{\nu}_s)$ and $d_1(\nu_s,\tilde{\nu}_s)$ into \eqref{eq: SMOOTHED H-1 ESTIMATE} and then applying Young's inequality to all the terms except for the $\norm{\partial_x^{-1} \Teps\Delta_s}_2^2$ term, we obtain
\begin{align}\label{eq: SMOOTHED H-1 ESTIMATE EXPANDED TO REMOVE THE D TERMS}
	\begin{split}
		\E\norm{\ad\Teps\Delta_{t\wedge t_n}}_2^2 &\le
		\delta c_n^\prime \E \int_0^{t\wedge t_n} \norm{\frac{\diff \Delta_s}{\diff x}}_2^2\diff s - c_0 \E \int_0^{t\wedge t_n} \norm{\Teps\Delta_s}_2^2 \diff s \\
		&+c_n^\prime(1 + \lambda + \lambda \delta) \int_0^{t} \E\norm{\partial_x^{-1} \Teps\Delta_{s\wedge t_n}}_2^2\diff s +c_{a,n}^\prime e^{-a\lambda}+ o(1)
	\end{split}
\end{align}
as $\varepsilon \downarrow 0$. The Monotone Convergence Theorem gives
\begin{equation*}
	\E \int_0^{t\wedge t_n} \norm{\Teps\Delta_s}_2^2 \diff s \to  \E \int_0^{t\wedge t_n} \norm{\frac{\diff \Delta_s}{\diff x}}_2^2\diff s 
\end{equation*}
as $\varepsilon \downarrow 0$ because $\norm{\Teps\Delta_s}_2 \nearrow \norm{\frac{\diff \Delta_s}{\diff x}}_2$. Therefore, fixing $\delta = c_0/2c_n^{\prime}$, we have that, for all $\varepsilon > 0$ small enough, the difference of the first two terms will be negative. Thus, applying Gr\"onwall's inequality to \eqref{eq: SMOOTHED H-1 ESTIMATE EXPANDED TO REMOVE THE D TERMS}, we get
\begin{equation*}
	\E\norm{\ad\Teps\Delta_{t\wedge t_n}}_2^2 \le c_{a,n}^\prime e^{c_n^\prime(1 + \lambda + \lambda \delta)T}e^{-a\lambda} + o(1)
\end{equation*}
as $\varepsilon \downarrow 0$. Hence, it follows from Fatou's lemma and \eqref{eq: BOUND ON THE H -1 NORM} that
\begin{equation*}
	\E \norm{\Delta_{t \wedge t_n}}_{H^{-1}} \le c_{a,n,T}^\prime e^{\lambda(1 + \delta)T-a\lambda}.
\end{equation*}
Therefore, setting $a = 2(1 + \delta^{-1})T$ and sending $\lambda \to  \infty$, we obtain $\E \norm{\Delta_{t \wedge t_n}}_{H^{-1}} = 0 $ for all $t \in [0,T]$. Lastly, sending $n \uparrow \infty$, we have $t_n \uparrow \infty$, and consequently, we have $\nu_t = \tilde{\nu}_t$ for all $t \in [0,T]$. Finally, by \eqref{eq: RELATIONSHIP BETWEEN THE FIRING FUNCTIONS}, we must then also have $F = \tilde{F}$. This completes the proof of \cref{thm: UNIQUENESS PROP}.$\hfill \Box$

\subsection{Strong Solutions of the SPDE -- Proof of \texorpdfstring{\cref{thm: CONDITIONAL LLN STATING THAT WE CONVERGE TO A UNIQUE LIMIT POINT THAT DEPENDS ON THE COMMON NOISE}}{ Strong Solutions}}\label{sec: proof_thm_2.6}

We begin by confirming that, for any limit point $(\nu^*, F^*, W^0)$ from the particle system  \eqref{eq: GENERALISED FINITE PARTICLE SYSTEM IN THE INTEGRATE AND FIRE NEURON MODELS}, we can indeed realise $\nu^*$ as a random variable with values in $D_{\mathbf{M}_{\le 1}}$ with respect to the Kolmogorov $\sigma$-algebra. Set $\prob^* \coloneqq \operatorname{Law}(\nu^*, F^*, W^0)$. By the first part of the proof of Proposition \ref{prop: THE LIMITING PROCESS WILL SATISFY THE ASSUMPTIONS WE HAVE SET OUT}, there exists a set $A \subset D_{\mathscr{S}^\prime}$ such that $\prob^*(A\times D_\R\times \mathcal{C}_\R)= 1$ and $\nu^*$ is an element of $D_{\mathbf{M}_{\le 1}}$ on $A\times D_\R\times \mathcal{C}_\R$. Next, we argue that the map
\begin{align*}
	(\nu,F,\omega) \mapsto \begin{cases}
		\nu\; \text{if} \;\nu \in A\\
		0\; \text{otherwise}
	\end{cases}
\end{align*}
from $D_{\mathscr{S}^\prime} \times D_\R\times \mathcal{C}_\R$ to $D_{\mathbf{M}_{\le 1}}$ is measurable. Since $D_{\mathbf{M}_{\le 1}}$ has the Kolmogorov $\sigma$-algebra and $\mathbf{M}_{\le 1}$ the Borel $\sigma$-algebra for the topology of weak convergence, the finite intersections of the sets $U_{t,\phi,x,\delta} = \{\mu \in D_{\mathbf{M}_{\le 1}}: \nnnorm{\mu_t(\phi) - x} < \delta\}$ for $\phi \in \mathcal{C}_b$, $x \in \R$, $t \in [0,T]$, and $\delta > 0$ form a $\pi$-system which generates the $\sigma$-algebra on $D_{\mathbf{M}_{\le 1}}$. For any $\phi \in \mathcal{C}_b$, we may find a sequence $\phi^n \in \mathscr{S}$ which is uniformly bounded and converges pointwise to $\phi$. Thus, when $|x| \ge \delta$,
\begin{align*}
	\{\nu^*\mathbbm{1}_{A\times D_\R\times \mathcal{C}_\R} \in U_{t,\phi,x,\delta}\} &=  \left(A \cap B\right)\times D_\R\times \mathcal{C}_\R \qquad \text{where}\\
	B & = \bigcup_{M \in \N} \bigcup_{N \in \N} \bigcap_{n > N} \left\{\nu \in  D_{\mathscr{S}^\prime}: \nnnorm{\nu_t(\phi^n) - x} < \delta - 1/M\right\}
\end{align*}
and similarly when $|x| < \delta$. Sets of this form are measurable in $D_{\mathscr{S}^\prime} \times D_\R\times \mathcal{C}_\R$.

Consequently, setting $(\hat{\nu},\hat{F},\hat{W}^0) = (\nu^*, F^*, W^0)\mathbbm{1}_{A\times D_\R\times \mathcal{C}_\R}$, we obtain a random variable $(\hat{\nu},\hat{F},\hat{W}^0)$ with values in $D_{\mathbf{M}_{\le 1}}\times D_\R\times \mathcal{C}_\R$ which satisfies the SPDE \eqref{eqn: THE LIMITTING SPDE EQUATION} in \cref{thm: THE LIMITING SPDE THEOREM MAIN RESULT OF THE WORK}, with probability $1$, along with the regularity conditions of \cref{ass: ASSUMPTIONS NEEDED TO PROVE UNIQUENESS WHICH PARALLEL THOSE FROM THE SPDE PAPER}.

Now, consider any two solutions  $(\nu, F, W^0)$ and $(\tilde{\nu}, \tilde{F}, \tilde{W}^0)$ of the SPDE \eqref{eqn: THE LIMITTING SPDE EQUATION} in the sense of \cref{thm: THE LIMITING SPDE THEOREM MAIN RESULT OF THE WORK}, possibly defined on different probability spaces. With the topology of weak convergence of measures, we have that $\mathbf{M}_{\le 1}$ is a Polish space. Hence, $D_{\mathbf{M}_{\le 1}}$ equipped with the $J_1$ Skorokhod topology is a Polish space (\citep[Theorem~3.5.6]{ethier2009markov}) satisfying that the Borel $\sigma$-algebra coincides with the Kolmogorov $\sigma$-algebra (\citep[Proposition~3.7.1]{ethier2009markov}). As above, we can take $(\nu,F)$ and $(\tilde{\nu},\tilde{F})$ to be valued in $D_{\mathbf{M}_{\le 1}}\times D_\R$, so we are in the setting of \cite{kurtz2014yamada}. Write $\mathbb{P}^1=\operatorname{Law}(\nu,F,W^0)$ and $\mathbb{P}^2=\operatorname{Law}(\tilde\nu,\tilde F,\tilde W^0)$.
As in the proof of \citep[Theorem~1.5]{kurtz2014yamada}, we can find two Borel maps $\Psi^i:\mathcal{C}_\R\times[0,1]\to D_{\mathbf{M}_{\le 1}}\times D_\R$
such that, for a Brownian motion $W$ and $\Xi^i  \sim \operatorname{Unif}[0,1]$  with
$(W,\Xi^1,\Xi^2)$ mutually independent, we have $(\Psi^i(W,\Xi^i),W)\sim\mathbb{P}^i$, $i=1,2$. By the independence and the joint law equivalences, we can deduce that each $\Psi^i(W,\Xi^i)$ is adapted to the filtration generated by $W$ and $(\Xi^1,
\Xi^2)$, for which $W$ is a Brownian motion, and that they solve the SPDE and satisfy \cref{ass: ASSUMPTIONS NEEDED TO PROVE UNIQUENESS WHICH PARALLEL THOSE FROM THE SPDE PAPER}. Thus, \cref{thm: UNIQUENESS PROP} applies and gives $\Psi^1(W,\Xi^1) =\Psi^2(W,\Xi^2)$ with probability $1$. Firstly, this implies $\mathbb{P}^1 = \mathbb{P}^2$, which proves the uniqueness in law. Secondly, by the almost sure equality and the mutual independence of $(\Xi^1, \Xi^2, W)$, it follows from \citep[Lemma~A.2]{kurtz2007yamada} that $\Psi^i(W,\Xi^i)=Q(W)$ for a measurable map $Q: \mathcal{C}_\R \to D_{\mathbf{M}_{\le 1}}\times D_\R$. In particular, $\mathbb{P}^1 =\operatorname{Law}(Q(W^0),W^0)$, so $(\nu,F)=Q(W^0)$ with probability 1, as desired. Using that the Borel $\sigma$-algebra for $(D_{\mathscr{S}^\prime},M_1)$ coincides with the Kolmogorov $\sigma$-algebra (\citep[Proposition~2.7(iv)]{ledger2016skorokhod}) and that the Kolmogorov $\sigma$-algebra on $D_{\mathbf{M}_{\le 1}}$ is generated by the topology of weak convergence on $\mathbf{M}_{\le 1}$, we have that the inclusion map $I: D_{\mathbf{M}_{\le 1}}\times D_\R \rightarrow D_{\mathscr{S}^\prime} \times D_\R $ is measurable, so we can take $Q$ to map to $D_{\mathscr{S}^\prime}\times D_\mathbb{R}$.

Since there is a unique law, we conclude that $(\nu^N,F^N,W^0)$ converges weakly to $(\nu,F,W^0)$ with $(\nu,F)=Q(W^0)$, and the convergence of the mean process $M^N$ then follows from \cref{sec: SUBSECTION THAT SHOWS THE CONVERGENCE TO THE LIMITTING SPDE}. Thus, the proof is complete.\qed


\paragraph{Funding details:} Part of this work was supported by the Engineering and Physical Sciences Research Council [EP/S023925/1].

\paragraph{Disclosure statement:} The authors report there are no competing interests to declare. 

\paragraph{Declaration of generative AI use:} The authors report generative AI was not used in their research or preparation of this manuscript.

\paragraph{Data availability statement:} No data were created or analysed in this work. Data sharing is not applicable to this article.


\printbibliography

\end{document}